\documentclass{amsart}%
\usepackage{amssymb}
\usepackage{amsmath}
\usepackage{amsfonts}
\usepackage[all,cmtip]{xy}
\usepackage[margin=1.34in]{geometry}
\usepackage{graphicx}
\usepackage[backref]{hyperref}%
\newtheorem{theorem}{Theorem}
\newtheorem*{nonumtheorem}{Theorem}
\theoremstyle{remark}
\newtheorem{remark}{Remark}
\newtheorem{example}{Example}
\theoremstyle{plain}
\newtheorem*{acknowledgement}{Acknowledgement}

\newtheorem{corollary}[theorem]{Corollary}

\newtheorem{definition}[theorem]{Definition}

\newtheorem{lemma}[theorem]{Lemma}

\newtheorem{proposition}[theorem]{Proposition}

\numberwithin{equation}{section}
\begin{document}
\title[On the residue symbol]{On the local residue symbol\linebreak in the style of Tate and Beilinson}
\author{Oliver Braunling}
\address{Albert Ludwig University of Freiburg, Eckerstra\ss e 1, D-79104 Freiburg, Germany}
\email{oliver.braeunling@math.uni-freiburg.de}
\thanks{This work has been partially supported by the DFG SFB/TR45 \textquotedblleft%
{Periods, moduli spaces, and arithmetic of algebraic varieties}%
\textquotedblright\ and the Alexander von Humboldt Stiftung.}

\begin{abstract}
Tate gave a famous construction of the residue symbol on curves by using some
non-commutative operator algebra in the context of algebraic geometry. We
explain Beilinson's multidimensional generalization, which is not so
well-documented in the literature. We provide a new approach using Hochschild homology.

\end{abstract}
\maketitle

Suppose $X/k$ is a smooth proper algebraic curve over a field. One can define
the residue of a rational $1$-form $\omega$ at a closed point $x$ as
\begin{equation}
\operatorname*{res}\nolimits_{x}\omega=\operatorname*{Tr}\nolimits_{\kappa
(x)/k}a_{-1}\text{,}\qquad\text{where}\qquad\omega=\sum a_{i}t^{i}%
\ \mathrm{d}t\label{la3}%
\end{equation}
in terms of a local coordinate $t$, i.e. by picking an isomorphism
$\operatorname*{Frac}\widehat{\mathcal{O}}_{X,x}\simeq\kappa(x)((t))$. This
works, but is unwieldy since it depends on the choice of the isomorphism and
one needs to prove that it is well-defined, cf. Serre \cite[Ch. II]{serrecft}.
One could ask for a bit more:\medskip

\textbf{Aim:}\ Construct the local residue symbol without ever needing to
choose coordinates.\medskip

J. Tate \cite{MR0227171} has pioneered an approach which circumvents choices
of coordinates at all times by employing ideas in the style of functional
analysis: The local field%
\begin{equation}
\widehat{\mathcal{K}}_{X,x}:=\operatorname*{Frac}\widehat{\mathcal{O}}%
_{X,x}=\underset{s\neq0}{\underrightarrow{\operatorname*{colim}}}\underset
{i}{\underleftarrow{\lim}}\,\mathcal{O}_{X,x}/\mathfrak{m}_{x}^{i}\left\langle
s^{-1}\right\rangle \label{la20}%
\end{equation}
carries a canonical topology, defined by viewing it as an ind-pro limit of
finite-dimensional discrete $k$-vector spaces. This topology needs no
assumptions on the base field, e.g. it could be just a finite field. We get a
non-commutative algebra of continuous $k$-vector space endomorphisms $E$. Via
multiplication operators $x\mapsto f\cdot x$ the functions $f\in
\widehat{\mathcal{K}}_{X,x}$ embed into $E$. Using the ideal of compact
operators, Tate shows that $E$ has a canonical central extension $\widehat{E}$
as a Lie algebra by a formal element \textsf{c} so that%
\begin{equation}
\lbrack f,g]_{\widehat{E}}=\operatorname*{res}\nolimits_{x}f\mathrm{d}%
g\cdot\text{\textsf{c}.}\label{la2}%
\end{equation}
Tate now uses the left-hand side as an intrinsically coordinate-independent
definition for the residue (R. Hartshorne advertises this as `clever' in his
textbook \cite[Ch. III, \S 7]{MR0463157}). For an $n$-dimensional smooth
proper algebraic variety $X/k$, the global residue%
\[
H^{n}(X,\Omega_{X/k}^{n})\longrightarrow k
\]
is induced from $n$-dimensional local residue symbols. There is the
conventional approach to this using A. Grothendieck's residue symbol
\cite{MR0222093}, however A. Beilinson \cite{MR565095} has shown that one can
also describe this map by a beautiful multidimensional generalization of
Tate's approach. He interprets the commutators which appear in Tate's theory
as low-degree avatars of the differential in Lie homology. As such, one can
give an explicit formula for the higher residue in terms of cascading
commutators, roughly generalizing Eq. \ref{la2}.

\subsection{\label{minisect_IntroduceMainResult}The results}

Beilinson's theory of ad\`{e}les allows us to reduce the construction of his
residue to a certain axiomatic structure, which we will call a \emph{cubically
decomposed algebra} $A^{n}$ (it does not carry a name in Beilinson's paper
\cite{MR565095}). It then all boils down to constructing a functional on Lie
homology%
\[
\phi_{Beil}:\quad H_{n+1}((A^{n})_{Lie},k)\longrightarrow k\text{.}%
\]
This is done via homological algebra $-$ a concrete determination of the map
leads to a differential on the $(n+1)$-st page of a certain spectral sequence.
It is a notoriously unpleasant problem to make such maps explicit; also this
formulation is quite exotic among other treatments of residues, making it
difficult to compare them.

We propose a new viewpoint: Under mild assumptions, we exhibit $A^{n}$ as an
iterated algebra extension of simpler cubically decomposed algebras $A^{i}$.
Now define%
\begin{equation}
\phi_{C}:\quad HC_{n}(A^{n})\overset{\delta\Lambda}{\longrightarrow}%
HC_{n-1}(A^{n-1})\overset{\delta\Lambda}{\longrightarrow}\cdots\overset
{\delta\Lambda}{\longrightarrow}HC_{0}(A^{0})\longrightarrow k\text{,}%
\label{lEqPhiC}%
\end{equation}
where $\Lambda:A^{n}\rightarrow A^{n}/A^{n-1}$ is a kind of Toeplitz operator
mechanism, and $\delta$ a connecting homomorphism $\delta:HC_{\ast}%
(A^{n}/A^{n-1})\rightarrow HC_{\ast-1}(A^{n-1})$ coming from the algebra
extension $A^{n-1}\hookrightarrow A^{n}\twoheadrightarrow A^{n}/A^{n-1}$.
Modulo maps relating Lie with Hochschild/cyclic homology and identifying
differential forms along $\Omega_{R/k}^{n}\simeq HH_{n}(R)$, the two
constructions produce the same map. The main idea is to view the Lie homology
as the Hodge $n$-part of Hochschild homology and get rid of relative Lie
homology by Hochschild excision, see \S \ref{section_Etiology} for a detailed
explanation. Concretely:

\begin{nonumtheorem}
[Lie-to-Hochschild Comparison]Suppose $A$ is a unital $n$-fold cubically
decomposed algebra over $k$ which has local units on all levels. Let
$\mathfrak{g}$ denote its Lie algebra. Then there are canonical maps, making
the diagram%
\[%
\xymatrix@R=0.1in
{
& H_{n}(\mathfrak{g},\mathfrak{g}) \ar[dl] \ar[dr] \\
H_{n+1}(\mathfrak{g},k) \ar[dddr]_{{\phi}_{Beil}} \ar[dr] & & HH_{n}%
(A) \ar[dl] \ar[dddl]^{{\phi}_{HH}} \\
& HC_{n}(A) \ar[dd]_{{\phi}_{C}} \\
\\
& k
}%
\]
commutative. Here $\phi_{Beil}$ is Beilinson's construction in \cite{MR565095}%
, and $\phi_{HH}$ and $\phi_{C}$ are constructed in this paper; $\phi_{C}$
being as in Line \ref{lEqPhiC}.
\end{nonumtheorem}

This is glued from the triangles of Cor. \ref{marker_CorCompat1} and Cor.
\ref{marker_AllComparisonDiagramCor}. Applied to the concrete task of
describing residues, this leads to our Local Formula, Theorem
\ref{thm_localformula}, unravelling all these maps in concrete terms, once
local coordinates are chosen. This is our multi-dimensional generalization of
Equation \ref{la2}.\medskip

We also give our own interpretation of the Lie homology mechanism in
Beilinson's \cite{MR565095} in \S \ref{section_Etiology}. No such attempt of
an explanation seems to exist in the literature, and\ I hope that future
readers of \cite{MR565095} will find this helpful.\medskip

We also give our own version of a reciprocity-like vanishing theorem. I tried
to find the correct formulation of such a result on the level of cubically
decomposed algebras. The `abstract residue formula' of Arbarello, de Concini
and Kac \cite[\S 2]{MR1013132} may be regarded as its ancestor.

\begin{nonumtheorem}
[Cube Reciprocity\ Law]Let $A$ be a unital $n$-fold cubically decomposed
algebra with local units on all levels. Let $P^{\pm}\in A$ be idempotents such
that%
\[
P^{+}+P^{-}=1\qquad\text{and}\qquad P^{\pm}A\in I_{1}^{\pm}\text{.}%
\]
If $R\subseteq A$ is a sub-algebra such that $P^{+}A$ (or $P^{-}A$) is a left
$R$-submodule of $A$, then for all $r\in HH_{n}(R)$:%
\[
\phi_{C}(r)=0\text{.}%
\]

\end{nonumtheorem}

See Theorem \ref{thm_multreclaw}. Applied to geometry, this result implies,
for example, the vanishing of the sum of residues of a rational $1$-form on an
integral proper curve. It is a possible abstraction and generalization of the
corresponding vanishing theorem in Tate's paper \cite{MR0227171}.

\section{\label{marker_sect_TateOriginalConstruction}Tate's original
construction}

\subsection{Operator ideals and the snake lemma}

We shall quickly recall the classical construction of Tate \cite{MR0227171},
from a perspective which points naturally to the multidimensional
generalization. Let $X/k$ be a smooth algebraic curve. For every closed point
$x\in X$, the completed stalk of the structure sheaf is a complete discrete
valuation ring with residue field $\kappa(x)$. By Cohen's Structure Theorem
there is an isomorphism%
\begin{equation}
\widehat{\mathcal{K}}_{X,x}:=\operatorname*{Frac}\widehat{\mathcal{O}}%
_{X,x}\simeq\kappa(x)((t))\text{,}\label{lbb1}%
\end{equation}
however there is no canonical isomorphism.

\begin{example}
[how not to do it]Attempting to construct the residue via a map%
\begin{equation}
\widehat{\mathcal{K}}_{X,x}\ni\sum a_{i}t^{i}\mapsto a_{-1}\in\kappa\left(
x\right) \label{lbb2}%
\end{equation}
quickly leads us into trouble. One could use expansions in $(t+t^{2})$ instead
of $t$ for example, or any other isomorphism in Eq. \ref{lbb1}. Even worse,
there is no canonical copy of $\kappa(x)$ inside $\widehat{\mathcal{K}}_{X,x}%
$:\ Suppose the residue field is $\mathbf{Q}(s)$ so that $\widehat
{\mathcal{O}}_{X,x}=\mathbf{Q}(s)[[t]]$. Then $\mathbf{Q}(s+t)\subset
\widehat{\mathcal{O}}_{X,x}$ is a subfield which is mapped isomorphically to
$\mathbf{Q}(s)$ modulo the maximal ideal $\mathfrak{m}_{x}=(t)$. Thus, we get
isomorphisms%
\[
\widehat{\mathcal{K}}_{X,x}\simeq\mathbf{Q}(s)((t))\qquad\text{and}%
\qquad\widehat{\mathcal{K}}_{X,x}\simeq\mathbf{Q}(s+t)((t))\text{,}%
\]
both of which qualify for the residue definition as in Eq. \ref{lbb2}. Hence,
both the choice of a coefficient field and the choice of a uniformizing
variable $t$ are non-canonical. Only in exceptional situations one \emph{does}
have a canonical coefficient field, notably if the residue field is perfect of
positive characteristic or algebraic over the rationals \cite[Ch. II
\S 5.2-5.4]{MR1915966}.
\end{example}

Without needing to choose such an isomorphism, $\widehat{\mathcal{K}}_{X,x}$
has a canonical topology coming from the presentation $\widehat{\mathcal{K}%
}_{X,x}=\underrightarrow{\lim}_{s}\underleftarrow{\lim}_{i}\mathcal{O}%
_{X,x}/\mathfrak{m}_{x}^{i}\left\langle \frac{1}{s}\right\rangle $, where we
regard each $\mathcal{O}_{X,x}/\mathfrak{m}_{x}^{i}$ as a \emph{discrete}
$k$-vector space. This turns the inner pro-limit into a linearly compact
$k$-vector space and the ind-limit over all finitely generated $\widehat
{\mathcal{O}}_{X,x}$-submodules of $\widehat{\mathcal{K}}_{X,x}$ into a
linearly locally compact $k$-vector space.

We can now regard $\widehat{\mathcal{K}}_{X,x}$ as a infinite-dimensional
topological $k$-vector space. The topology differs from the ones
conventionally used in functional analysis over $\mathbf{R}$ or $\mathbf{C}$
because it is generated from an open neighbourhood basis of $0$ which consists
of linear subspaces; they are called \emph{lattices}:

\begin{definition}
\label{def_onedim_lattice}A \emph{lattice} in a finite-dimensional
$\widehat{\mathcal{K}}_{X,x}$-vector space $V$ is a finitely generated
$\widehat{\mathcal{O}}_{X,x}$-submodule $L\subseteq V$ so that $\widehat
{\mathcal{K}}_{X,x}\cdot L=V$.
\end{definition}

Using the topology, we get the associative operator algebra of continuous
$k$-linear endomorphisms%
\begin{equation}
E:=\{\phi:\widehat{\mathcal{K}}_{X,x}\rightarrow\widehat{\mathcal{K}}%
_{X,x}\mid\phi\text{ is }k\text{-linear and continuous}\}\text{.}\label{la21}%
\end{equation}

\begin{definition}
\label{def_onedim_compactanddiscrete}We call an operator $\phi\in E$

\begin{enumerate}
\item \emph{compact} if there is a lattice $L$ with $\operatorname*{im}%
\phi\subseteq L$;

\item \emph{discrete} if there is a lattice $L$ with $L\subseteq\ker\phi$.
\end{enumerate}
\end{definition}

These classes of operators form two-sided ideals $I^{+},I^{-}$ in $E$.
Moreover, we have $I^{+}+I^{-}=E$. Write $I_{tr}:=I^{+}\cap I^{-}$ for their
intersection. Thus, we get a short exact sequence of $E$-bimodules,%
\begin{equation}
0\longrightarrow I_{tr}\longrightarrow I^{+}\oplus I^{-}\longrightarrow
E\longrightarrow0\text{.}\label{la22}%
\end{equation}
We may formally \textquotedblleft exterior tensor\textquotedblright\ this with
another copy of $E$, giving a commutative diagram with exact rows:%
\begin{equation}
\xymatrix{ 0 \ar[r] & {(I^{+}\wedge E)\cap (I^{-}\wedge E)} \ar[r] \ar[d]^{[-,-]} & \left( I^{+}\wedge E\right) \oplus \left( I^{-}\wedge E\right) \ar[r] \ar[d]^{[-,-]} & E\wedge E \ar[r] \ar[d]^{[-,-]} & 0 \\ 0 \ar[r] & I_{tr} \ar[r] & I^{+}\oplus I^{-} \ar[r] & E \ar[r] & 0 }\label{la6}%
\end{equation}
(for $V\subseteq W$ a subspace of a vector space, $V\wedge W$ denotes the
subspace of $%
{\textstyle\bigwedge\nolimits^{2}}
W$ generated by vectors $v\wedge w$ with $v\in V,w\in W$.) The snake lemma
gives us a canonical morphism, call it $(\ast)$, and thus%
\begin{equation}
\phi:\widehat{\mathcal{K}}_{X,x}\wedge\widehat{\mathcal{K}}_{X,x}%
\longrightarrow\ker(E\wedge E\rightarrow E)\overset{(\ast)}{\longrightarrow
}\operatorname*{coker}([\ldots]\rightarrow I_{tr})\overset{\operatorname*{tr}%
}{\longrightarrow}k\text{.}\label{la7}%
\end{equation}
The local rational functions $\widehat{\mathcal{K}}_{X,x}\subset E$ are viewed
as the respective multiplication operator $x\mapsto f\cdot x$, which is
clearly continuous. Functions commute, i.e. $[f,g]=0$, so the left-hand side
arrow indeed exists. On the other hand, traces satisfy $\operatorname*{tr}%
([X,Y])=0$, so the trace on the right-hand side factors through the cokernel.
Tate now proves that $\phi(f\wedge g)=\operatorname*{res}\nolimits_{x}%
f\mathrm{d}g$. See Lemma \ref{Lemma_LocalTateFormulaDimOne} for the proof.
\cite[\S 2]{MR0227171}.

\begin{remark}
Tate's original paper \cite{MR0227171} actually defines $I^{+},I^{-}$ (called
$E_{1},E_{2}$ in \emph{loc. cit.}) slightly differently. He fixes a special
open, the ring of integers $\widehat{\mathcal{O}}_{X,x}\subset\widehat
{\mathcal{K}}_{X,x}$, and instead of compactness he demands an operator to map
the entire space into this open, up to a finite-dimensional discrepancy. See
also Definition \ref{def_HigherAdeleOperatorIdeals}. But this comes down to
the same as the topological definition we use here. The presentation using a
topological language is taken from \cite[\S 1.2]{BFM_Conformal} ($I^{+},I^{-}$
are called $Hom_{+},Hom_{-}$ in \emph{loc. cit.}).
\end{remark}

\subsection{Finite-potent trace}

We have tacitly swept a detail under the rug: Since $E$ is
infinite-dimensional, a general operator in $E$ will not have a well-defined
trace. Clearly finite-rank operators will still have a trace, but in Tate's
construction the operators in $I_{tr}$ a priori need not be of finite rank. In
functional analysis one would now hope for the ideal of nuclear operators, but
the ind-pro type topologies are not rich enough to give a convergence
condition on the operator spectrum any interesting content. Instead, Tate uses
the philosophy that any nilpotent operator should have trace zero, even if it
is not of finite rank. We briefly summarize Tate's operator trace
\cite{MR0227171} as we will also need it later:\medskip

Let $F_{0}$ be a field and $V$ an $F_{0}$-vector space. Call an endomorphism
$g\in\operatorname*{End}\nolimits_{F_{0}}\left(  V\right)  $
\emph{finite-potent} if there is some $n\geq1$ such that the image $g^{n}V$ is
finite-dimensional over $F_{0}$. An $F_{0}$-vector subspace $\Gamma
\subseteq\operatorname*{End}\nolimits_{F_{0}}\left(  V\right)  $ is called a
\emph{finite-potent family} if there is some $n\geq1$ such that $(g_{1}%
\circ\cdots\circ g_{n})V$ is finite-dimensional for any choice of
$g_{1},\ldots,g_{n}\in\Gamma$.

\begin{proposition}
[\cite{MR0227171}]\label{BT_PropTateTraceConstruction}(Tate) For every $F_{0}%
$-vector space $V$ and every finite-potent $g\in\operatorname*{End}%
\nolimits_{F_{0}}\left(  V\right)  $ there is a unique element, denoted
$\operatorname*{tr}\nolimits_{V}g\in F_{0}$ (and called \emph{Tate trace}),
such that the following rules hold:

\begin{description}
\item[T1] If $V$ is finite-dimensional, $\operatorname*{tr}\nolimits_{V}g$ is
the usual trace.

\item[T2] If $W\subseteq V$ is any $F_{0}$-vector subspace and $gW\subseteq W
$, we have~$\operatorname*{tr}\nolimits_{V}g=\operatorname*{tr}\nolimits_{W}%
g+\operatorname*{tr}\nolimits_{V/W}g$.

\item[T3] If $g$ is nilpotent, $\operatorname*{tr}\nolimits_{V}g=0$.

\item[T4] Suppose $\Gamma\subseteq\operatorname*{End}\nolimits_{F_{0}}\left(
V\right)  $ is a finite-potent family. Then $\operatorname*{tr}\nolimits_{V}%
\mid_{\Gamma}$ is $F_{0}$-linear, i.e. $\operatorname*{tr}\nolimits_{V}\left(
af+bg\right)  =a\operatorname*{tr}\nolimits_{V}f+b\operatorname*{tr}%
\nolimits_{V}g$ for all $a,b\in F_{0}$ and $f,g\in\Gamma$.
(\footnote{Mysteriously, in general the linearity axiom \textbf{T4} fails. A
concrete counter-example is given by Pablos Romo in \cite{MR2319783}. See also
\cite{MR2360831}, \cite{MR3177048} for a more thorough discussion. However,
this need not concern us; the non-linearity will never show up in the
applications of the above proposition in this paper.})

\item[T5] Suppose $f:V\rightarrow V^{\prime}$ and $g:V^{\prime}\rightarrow V$
are $F_{0}$-vector space homomorphisms and the composition $f\circ g$ is
finite-potent on $V^{\prime}$. Then the reverse composition $g\circ f$ is
finite-potent on $V$ and $\operatorname*{tr}\nolimits_{V^{\prime}}\left(
f\circ g\right)  =\operatorname*{tr}\nolimits_{V}\left(  g\circ f\right)  $.
\end{description}
\end{proposition}

\begin{example}
Consider $F_{0}:=k$ and $V:=k[t,t^{-1}]$. Then $f\in\operatorname*{End}%
\nolimits_{F_{0}}(V)$ given by $t^{i}\mapsto t^{-i}$ for $i\geq0$ and
$t^{i}\mapsto0$ for $i<0$ is a finite-potent morphism which is \textit{not}
finite-rank, so the usual trace is not applicable. The vector $t^{0}$ spans a
$1$-dimensional $f$-stable subspace and on the vector space quotient
$k[t,t^{-1}]/k\left\langle t^{0}\right\rangle $ the induced operator
$\overline{f}$ is nilpotent, so by \textbf{T1} and \textbf{T2} we get
$\operatorname*{tr}\nolimits_{V}f=1$.
\end{example}

\begin{lemma}
[{\cite[Thm. 2]{MR0227171}}]\label{Lemma_LocalTateFormulaDimOne}$\phi(f\wedge
g)=\operatorname*{res}\nolimits_{x}f\mathrm{d}g$.
\end{lemma}

\begin{proof}
We just need to follow the snake morphism in Eq. \ref{la6}. For this we need
to split the surjection in the top row of Eq. \ref{la6}, i.e. pick idempotents
$P^{\pm}$ on $E$ such that $P^{\pm}E\subseteq I^{\pm}$ so that $P^{+}%
+P^{-}=\mathbf{1}$. Then unwinding the snake morphism yields%
\[
\xymatrix{ & (P^{+}f \wedge g) \oplus (-P^{-}f \wedge g) \ar[r] \ar[d] & f \wedge g \\ [P^{+}f, g] \ar[r] & [P^{+}f, g] \oplus -[P^{-}f, g] }
\]
and so the composition of maps in Eq. \ref{la7} unwinds to the concrete
formula%
\begin{equation}
\phi:\widehat{\mathcal{K}}_{X,x}\wedge\widehat{\mathcal{K}}_{X,x}\rightarrow
k\qquad\qquad\phi(f\wedge g)=\operatorname*{tr}[P^{+}f,g]\label{la8}%
\end{equation}
(or $-\operatorname*{tr}[P^{-}f,g]$ equivalently). It follows immediately that
this formula is independent of the choice of a particular $P^{+}$. We may pick
any isomorphism $\widehat{\mathcal{K}}_{X,x}\simeq\kappa(x)((t))$. Suppose $x$
is a $k$-rational point, i.e. $\kappa(x)=k$. In order to distinguish between
$t^{i}$ as a multiplication operator or as a topological basis element of
$\widehat{\mathcal{K}}_{X,x}$, let us write $\mathbf{t}^{i}$ for the
latter.\ Then take for example $P^{+}(\mathbf{t}^{i}):=\delta_{i\geq
0}\mathbf{t}^{i}$. This clearly lies in $I^{+}$, $P^{-}:=\mathbf{1}-P^{+}$
lies in $I^{-}$ and we compute%
\[
\lbrack P^{+}t^{i},t^{j}]\mathbf{t}^{\lambda}=\delta_{\lambda+i+j\geq
0}\mathbf{t}^{\lambda+i+j}-\delta_{\lambda+i\geq0}\mathbf{t}^{\lambda
+i+j}=\delta_{-j\leq\lambda+i<0}\mathbf{t}^{\lambda+i+j}\text{.}%
\]%
\begin{equation}%
{\includegraphics[
height=0.4359in,
width=3.9738in
]%
{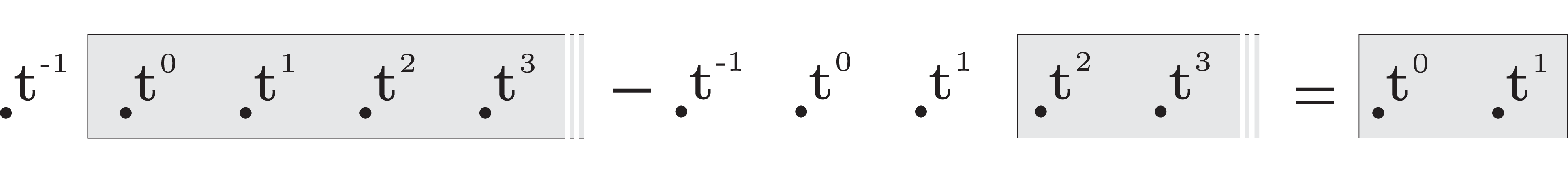}%
}%
\label{lTRACEPIC1}%
\end{equation}
Suppose $j=1$, then $[P^{+}t^{i},t]\mathbf{t}^{\lambda}=\delta_{-1\leq
\lambda+i<0}\mathbf{t}^{\lambda+i+1}$. This has a non-trivial invariant
subspace iff $i=-1$, so $\phi(t^{i}\wedge t)=0$ for $i\neq-1$. For $i=-1$ we
get $[P^{+}t^{-1},t]\mathbf{t}^{\lambda}=\delta_{-1\leq\lambda-1<0}%
\mathbf{t}^{\lambda}$, so $k\left\langle \mathbf{t}^{0}\right\rangle $ is a
$1$-dimensional invariant subspace and therefore $\phi(t^{-1}\wedge t)=1$:
Just like $\operatorname*{res}t^{i}\mathrm{d}t=\delta_{i=-1}$. If $x$ is an
arbitrary closed point, $\kappa(x)/k$ is a finite field extension. The above
computation still applies if we work with $\kappa(x)$-vector spaces. Writing
$\kappa(x)$ itself as a $[\kappa(x):k]$-dimensional $k$-vector space yields
the formula $\operatorname*{res}t^{i}\mathrm{d}t=[\kappa(x):k]\delta_{i=-1}$.
\end{proof}

The map $\phi:\widehat{\mathcal{K}}_{X,x}\wedge\widehat{\mathcal{K}}%
_{X,x}\rightarrow k$ induces a functional $H_{2}((\widehat{\mathcal{K}}%
_{X,x})_{Lie},k)^{\ast}\cong H^{2}((\widehat{\mathcal{K}}_{X,x})_{Lie},k)$ and
the resulting Lie central extension is the one arising from pushing out Eq.
\ref{la22} by Tate's trace,%
\[
\xymatrix{ 0 \ar[r] & I_{tr} \ar[r] \ar[d] & I^{+} \oplus I^{-} \ar[r] \ar[d] & E \ar[r] \ar[d] & 0 \\ 0 \ar[r] & k \ar[r] & \widehat{E} \ar[r] & E \ar[r] & 0 }
\]

\begin{definition}
The central extension $\widehat{E}$ in the lower row is \emph{Tate's central
extension}.
\end{definition}

\section{Ad\`{e}les\label{section_Adeles}}

\subsection{For curves\label{marker_ForCurves}}

Let $X/k$ be an integral smooth proper algebraic curve. Tate \cite{MR0227171}
uses the language of ad\`{e}les of the curve $-$ a technique borrowed from
number theory. We write $\prod_{x\in U^{1}}\widehat{\mathcal{O}}_{X,x}$ as a
shorthand for the $\mathcal{O}_{X}$-module sheaf%
\[
U\mapsto\left.
{\textstyle\prod\nolimits_{x\in U^{1}}}
\right.  \widehat{\mathcal{O}}_{X,x}\qquad\text{for }U\text{ any Zariski open
set,}%
\]
where $\widehat{\mathcal{O}}_{X,x}$ is the $\mathfrak{m}_{x}$-adically
completed local ring and $U^{p}$ denotes the set of codimension $p$ points in
$U$. The restriction map to smaller opens is the factorwise identity so that
the sheaf is flasque. There is an exact sequence of $\mathcal{O}_{X}$-module
sheaves%
\begin{equation}
0\longrightarrow\mathcal{O}_{X}\overset{\operatorname*{diag}}{\longrightarrow
}k(X)\oplus\left.
{\textstyle\prod\nolimits_{x\in U^{1}}}
\right.  \widehat{\mathcal{O}}_{X,x}\overset{\operatorname*{diff}%
}{\longrightarrow}\left.
{\textstyle\prod\nolimits_{x\in U^{1}}^{\prime}}
\right.  \widehat{\mathcal{K}}_{X,x}\longrightarrow0\text{,}\label{la9}%
\end{equation}
where $\mathcal{O}_{X}$ is the structure sheaf, $k\left(  X\right)  $ the
locally constant sheaf of rational functions, $\widehat{\mathcal{K}}%
_{X,x}:=\operatorname*{Frac}\widehat{\mathcal{O}}_{X,x}$, and the prime
superscript in the rightmost sheaf abbreviates the condition that for all but
finitely many $x\in U^{1}$ we demand sections to lie in the subspace
$\widehat{\mathcal{O}}_{X,x}\subset\widehat{\mathcal{K}}_{X,x}$. It is clear
that the sequence is exact and that it is actually a flasque resolution of
$\mathcal{O}_{X}$. Moreover, the global sections of the sheaves are
classically known as%
\[%
\begin{tabular}
[c]{l|l|l}%
sheaf side & ad\`{e}le side & \\\hline
$H^{0}(X,k(X))$ & $k(X)$ & function field of the curve\\
$H^{0}(X,\left.
{\textstyle\prod\nolimits_{x\in U^{1}}}
\right.  \widehat{\mathcal{O}}_{X,x})$ & $\mathbf{A}_{X}^{0}$ & integral
ad\`{e}le ring\\
$H^{0}(X,\left.
{\textstyle\prod\nolimits_{x\in U^{1}}^{\prime}}
\right.  \widehat{\mathcal{K}}_{X,x})$ & $\mathbf{A}_{X}$ & ad\`{e}le ring
\end{tabular}
\]
The ad\`{e}le approach to the theory of curves is due to Weil, we refer to
\cite{serrecft}, \cite{MR0227171} for a presentation of this formalism. The
same technique works for arbitrary quasi-coherent sheaves by tensoring. As a
result of the resolution in Eq. \ref{la9} we obtain for example%
\[
H^{0}(X,\mathcal{O}_{X})=\mathbf{A}_{X}^{0}\cap k(X)\qquad H^{1}%
(X,\mathcal{O}_{X})=\mathbf{A}_{X}/(\mathbf{A}_{X}^{0}+k(X))\text{.}%
\]
In particular, in order to describe the global residue map%
\[
H^{1}(X,\Omega_{X/k}^{1})\longrightarrow k
\]
we can employ such an ad\`{e}le resolution of the sheaf $\Omega_{X/k}^{1}$ to
give elements of the left-hand side a concrete representation, cf. Tate
\cite{MR0227171}.

\subsection{In general}

Parshin generalized this method to surfaces by introducing two-dimensional
ad\`{e}les \cite{MR0419458}, \cite{ParshinAdeleTheory}. Beilinson's paper
\cite{MR565095} provides the multidimensional technology. We need to recall
this for later use:\medskip

We mostly follow the notation in \cite{MR565095}. Let $X$ be a Noetherian
scheme. For points $\eta_{0},\eta_{1}\in X$ we write $\eta_{0}>\eta_{1}$ if
$\overline{\{\eta_{0}\}}\ni\eta_{1}$, $\eta_{1}\neq\eta_{0}$. Denote by
$S\left(  X\right)  _{n}:=\{(\eta_{0}>\cdots>\eta_{n}),\eta_{i}\in X\}$ the
set of chains of length $n+1$. Let $K_{n}\subseteq S\left(  X\right)  _{n}$ be
an arbitrary subset. For any point $\eta\in X$ define $\left.  _{\eta
}K\right.  :=\{(\eta_{1}>\cdots>\eta_{n})$ s.t. $(\eta>\eta_{1}>\cdots
>\eta_{n})\in K_{n}\}$, a subset of $S\left(  X\right)  _{n-1}$. Let
$\mathcal{F}$ be a \textit{coherent} sheaf on $X$. For $n=0$ and $n\geq1 $ we
define inductively%
\begin{align}
A(K_{0},\mathcal{F}):=  &
{\displaystyle\prod\nolimits_{\eta\in K_{0}}}
\underleftarrow{\lim}_{i}\mathcal{F}\otimes_{\mathcal{O}_{X}}\mathcal{O}%
_{X,\eta}/\mathfrak{m}_{\eta}^{i}\label{TATEMATRIX_l6}\\
A(K_{n},\mathcal{F}):=  &
{\displaystyle\prod\nolimits_{\eta\in X}}
\underleftarrow{\lim}_{i}A(\left.  _{\eta}K_{n}\right.  ,\mathcal{F}%
\otimes_{\mathcal{O}_{X}}\mathcal{O}_{X,\eta}/\mathfrak{m}_{\eta}^{i}%
)\text{.}\nonumber
\end{align}
The sheaves $\mathcal{F}\otimes_{\mathcal{O}_{X}}\mathcal{O}_{X,\eta
}/\mathfrak{m}_{\eta}^{i}$ are usually only quasi-coherent, so we complement
this partial definition as follows: For a \textit{quasi-coherent} sheaf
$\mathcal{F}$ we define $A(K_{n},\mathcal{F}):=\underrightarrow
{\operatorname*{colim}}_{\mathcal{F}_{j}}A(K_{n},\mathcal{F}_{j})$, where
$\mathcal{F}_{j}$ runs through all coherent subsheaves of $\mathcal{F}$ (and
hereby reducing to Eq. \ref{TATEMATRIX_l6}). As it is built successively from
ind-limits and Mittag-Leffler pro-limits, $A(K_{n},-)$ is a covariant exact
functor from quasi-coherent sheaves to abelian groups. Next, we observe that
$S(X)_{\bullet}$ carries a natural structure of a simplicial set (omitting the
$i$-th entry in a flag yields faces; duplicating the $i$-th entry in a flag
degeneracies). This turns%
\[
\mathbf{A}^{\bullet}(U,\mathcal{F}):=A(S(U)_{\bullet},\mathcal{F}%
)\qquad\text{(for }U\text{ Zariski open)}%
\]
into a sheaf of cosimplicial abelian groups (actually even cosimplicial
$\mathcal{O}_{X}$-module sheaves) and via the unreduced Dold-Kan
correspondence into a complex of sheaves, which we may denote by
$\mathbf{A}_{\mathcal{F}}^{i}$.

\begin{theorem}
[{\cite[\S 2]{MR565095}}]\label{lX_BeilinsonResolutionThm}(Beilinson) For a
Noetherian scheme $X$ and a quasi-coherent sheaf $\mathcal{F}$ on $X$, the
$\mathbf{A}^{i}(-,\mathcal{F})$ are flasque sheaves and%
\[
0\longrightarrow\mathcal{F}\longrightarrow\mathbf{A}_{\mathcal{F}}%
^{0}\longrightarrow\mathbf{A}_{\mathcal{F}}^{1}\longrightarrow\cdots
\]
is a flasque resolution.
\end{theorem}

See Huber \cite{MR1105583}, \cite{MR1138291} for a detailed proof. There are
also discussions circling around this construction in H\"{u}bl-Yekutieli
\cite{MR1374916}, Osipov \cite{MR2314612}, Parshin \cite{MR697316}. A very
interesting perspective on the relation of the Grothendieck residue complex
and ad\`{e}les can be found in Yekutieli \cite{MR2007399}. Beilinson actually
defines $S(X)_{n}$ so that also degenerate chains with $\eta_{i}=\eta_{i+1}$
are allowed, but one can check that this yields a slightly larger, but
quasi-isomorphic complex \cite[\S 5.1]{MR1138291}.

\begin{example}
Suppose $X$ is an integral smooth proper curve and $\triangle$ the set of all
flags. We may read the sets of codimension $p$ points $X^{p}$ as length one
flags. One computes%
\begin{align*}
A(X^{0},\mathcal{O}_{X})  &  =k(X)\qquad A(X^{1},\mathcal{O}_{X})=\left.
{\textstyle\prod\nolimits_{x\in X^{1}}}
\right.  \widehat{\mathcal{O}}_{X,x}\\
A(\triangle,\mathcal{O}_{X})  &  =\left.
{\textstyle\prod\nolimits_{x\in U^{1}}^{\prime}}
\right.  \widehat{\mathcal{K}}_{X,x}%
\end{align*}
so that Thm. \ref{lX_BeilinsonResolutionThm} reduces to the Eq. \ref{la9}.
\end{example}

It is also instructive to have a detailed look at a computation in dimension two:

\begin{example}
[generic behaviour]\label{example_GenericBehaviour}For a commutative and
unital ring $R$ and a prime $P\subset R$, $\underrightarrow
{\operatorname*{colim}}_{f\notin P}R[f^{-1}]$ is the localization $R_{P}$. For
any such $f$, $R[f^{-1}]=\underrightarrow{\operatorname*{colim}}%
_{i}R\left\langle f^{-i}\right\rangle $ for $i\rightarrow\infty$, where
$\left\langle f^{-i}\right\rangle $ denotes the $R $-sub\emph{module}
generated by $f^{-i}$\ inside $R[f^{-1}]$. Combining both colimits writes
$R_{P}$ as a colimit of finitely generated $R$-modules. We shall abbreviate
this colimit by writing $\underrightarrow{\operatorname*{colim}}_{f\notin
P}R\left\langle f^{-\infty}\right\rangle $. Now suppose
$X:=\operatorname*{Spec}k[s,t]$ and $\triangle:=\{(0)>(s)>(s,t)\}\in S\left(
X\right)  _{2}$ is a singleton set. Then
\begin{align*}
A(\triangle,\mathcal{O}_{X})  &  =A(\left.  _{(0)}\triangle\right.
,k(s,t))=\underset{f\notin(0)}{\underrightarrow{\operatorname*{colim}}%
}A(\left.  _{(0)}\triangle\right.  ,k[s,t]\left\langle f^{-\infty
}\right\rangle )\\
&  =\underset{f\notin(0)}{\underrightarrow{\operatorname*{colim}}}\underset
{i}{\underleftarrow{\lim}}A(\left.  _{(s)(0)}\triangle\right.
,k[s,t]\left\langle f^{-\infty}\right\rangle \otimes k[s,t]_{(s)}/(s^{i}))\\
&  =\underset{f\notin(0)}{\underrightarrow{\operatorname*{colim}}}\underset
{i}{\underleftarrow{\lim}}\underset{g\notin(s)}{\underrightarrow
{\operatorname*{colim}}}A(\left.  _{(s)(0)}\triangle\right.
,k[s,t]\left\langle f^{-\infty}\right\rangle \left\langle g^{-\infty
}\right\rangle /(s^{i}))\\
&  =\underset{f\notin(0)}{\underrightarrow{\operatorname*{colim}}}\underset
{i}{\underleftarrow{\lim}}\underset{g\notin(s)}{\underrightarrow
{\operatorname*{colim}}}\underset{j}{\underleftarrow{\lim}}k[s,t]_{(s,t)}%
\left\langle f^{-\infty}\right\rangle \left\langle g^{-\infty}\right\rangle
/(s^{i})/(s,t)^{j}\text{.}%
\end{align*}
and this yields%
\begin{align*}
&  =\underset{f\notin(0)}{\underrightarrow{\operatorname*{colim}}}\underset
{i}{\underleftarrow{\lim}}\underset{g\notin(s)}{\underrightarrow
{\operatorname*{colim}}}k[[s,t]]\left\langle f^{-\infty}\right\rangle
\left\langle g^{-\infty}\right\rangle /(s^{i})\\
&  =\underset{f\notin(0)}{\underrightarrow{\operatorname*{colim}}}\underset
{i}{\underleftarrow{\lim}}k[[s,t]][\left(  k[s,t]-(s)\right)  ^{-1}%
]\left\langle f^{-\infty}\right\rangle /(s^{i})\\
&  =\underset{f\notin(0)}{\underrightarrow{\operatorname*{colim}}%
}k((t))[[s]]\left\langle f^{-\infty}\right\rangle =k((t))((s))\text{.}%
\end{align*}
Note that this computation has not provided us with a canonical isomorphism to
$k((t))((s))$. Already writing $\mathbf{A}_{k}^{2}$ as $\operatorname*{Spec}%
k[s,t]$ involved the choice of coordinates $s,t$.
\end{example}

The structural similarities to the entire discussion in
\S \ref{marker_sect_TateOriginalConstruction} are more than obvious. Again, we
get an isomorphism $\simeq k((t))((s))$ making it tempting to define a
two-dimensional residue as%
\[
\operatorname*{res}\nolimits_{t}\operatorname*{res}\nolimits_{s}%
f\mathrm{d}s\wedge\mathrm{d}t=a_{-1,-1}\qquad\text{where}\qquad f=%
{\textstyle\sum\nolimits_{s,t}}
a_{k,l}s^{k}t^{l}\text{.}%
\]
While this would work (cf. \cite{MR0419458}, \cite{ParshinAdeleTheory}, but
beware of the topological pitfalls explained by Yekutieli \cite{MR1213064}) it
is a priori again entirely unclear whether this construction is independent of
the choice of the isomorphism.

\begin{example}
[exceptional behaviour]\label{example_exceptionaladelefibering}An example
where $A(\triangle,\mathcal{O}_{X})$ has two summands arises at singularities.
Note that for $\operatorname*{char}k\neq2$ the prime ideal $(s^{3}+s^{2}%
-t^{2})$ in $k[s,t]$ does not remain prime under the adelic completion because
the new element $\sqrt{1+s}=\sum_{k\geq0}\binom{1/2}{k}s^{k}$ enables a
factorization. Instead, we get two irreducible components.%
\[%
{\includegraphics[
height=0.5878in,
width=1.2478in
]%
{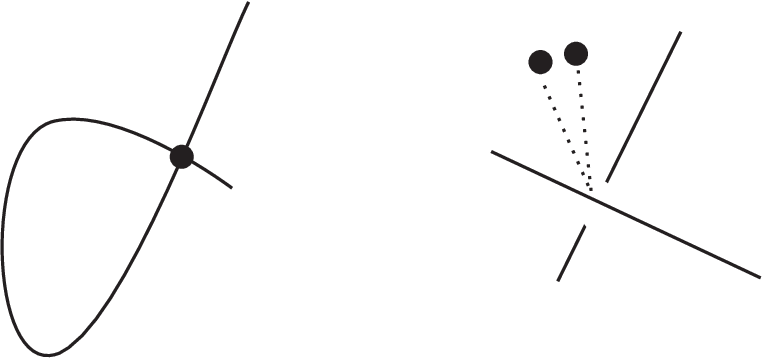}%
}%
\]
For the flag $\triangle:=\{(0)>(s,t)\}$ we obtain%
\begin{align*}
A(\triangle,\mathcal{O}_{X})  &  =A(\left.  _{(0)}\triangle\right.
,k(s)[t]/(s^{3}+s^{2}-t^{2}))\\
&  =\underset{f\notin(0)}{\underrightarrow{\operatorname*{colim}}}A(\left.
_{(0)}\triangle\right.  ,k[s,t]/(s^{3}+s^{2}-t^{2})\left\langle f^{-\infty
}\right\rangle )\\
&  =\underset{f\notin(0)}{\underrightarrow{\operatorname*{colim}}}\underset
{i}{\underleftarrow{\lim}}\underset{g\notin(s,t)}{\underrightarrow
{\operatorname*{colim}}}k[s,t]/(t^{i},s^{i},s^{3}+s^{2}-t^{2})\left\langle
f^{-\infty}\right\rangle \left\langle g^{-\infty}\right\rangle \\
&  =\underset{f\notin(0)}{\underrightarrow{\operatorname*{colim}}%
}k[[s,t]]/(s\sqrt{1+s}+t)(s\sqrt{1+s}-t)\left\langle f^{-\infty}\right\rangle
\\
&  =k((s))\oplus k((s))\text{,}%
\end{align*}
so that the image of $t$ is $(-s\sqrt{1+s},+s\sqrt{1+s})$. For the last step
in the computation note that the colimit is an Artinian ring, so it is
isomorphic to the product over the localizations at its maximal ideals.
\end{example}

A detailed description of the behaviour of ad\`{e}les especially for flags
along singular subvarieties can be found in \cite{MR697316},
\cite{ParshinAdeleTheory}. One can give a precise dictionary between direct
summand decompositions in ad\`{e}les and fibers of singularities under
normalization. We recommend Yekutieli \cite[\S 3.3]{MR1213064} for a thorough discussion.

\begin{definition}
[see \cite{MR1804915}]\label{def_nlocalfield}For $n\geq1$ an $n$\emph{-local
field with last residue field} $k$ is a complete discrete valuation field
whose residue field is an $(n-1)$-local field with last residue field $k$.
Moreover, we call $k$ itself the only $0$-local field with last residue field
$k$.
\end{definition}

In the formulation of the following proposition, we write $A_{Y}(-,-)$ to
refer to the ad\`{e}les belonging to the scheme $Y$.

\begin{proposition}
[{Structure theorem, \cite[p. $2$, $2^{\text{nd}}$ paragr.]{MR565095}}%
]\label{TATE_StructureOfLocalAdelesProp}Suppose $X$ is a finite type reduced
scheme of pure dimension $n$ over a field $k$ and $\triangle$ a finite subset
of%
\[
\{(\eta_{0}>\cdots>\eta_{n})\text{ such that }\operatorname*{codim}%
\nolimits_{X}\eta_{i}=i\}\qquad\subseteq S(X)_{n}\text{.}%
\]
Define%
\[
\triangle^{\prime}:=\{(\eta_{1}>\cdots>\eta_{n})\text{ such that }(\eta
_{0}>\cdots>\eta_{n})\in\triangle\text{ for some }\eta_{0}\}\text{.}%
\]

\begin{enumerate}
\item Then $A(\triangle,\mathcal{O}_{X})$ is a finite direct product of
$n$-local fields $\prod K_{i}$ such that each last residue field is a finite
field extension of $k$. Moreover,%
\begin{equation}
A(\triangle^{\prime},\mathcal{O}_{X})\overset{(\ast)}{\subseteq}%
{\textstyle\prod}
\mathcal{O}_{i}\subseteq%
{\textstyle\prod}
K_{i}=A(\triangle,\mathcal{O}_{X})\text{,}\label{lb2}%
\end{equation}
where $\mathcal{O}_{i}$ denotes the ring of integers of $K_{i}$ and $(\ast)$
is a finite ring extension. Each $K_{i}$ is non-canonically isomorphic to
$k^{\prime}((t_{1}))\cdots((t_{n}))$ for $k^{\prime}/k$ finite.

\item If we instead regard $\triangle^{\prime}$ as a flag in the closed
subscheme $\overline{\{\eta_{1}\}}$ the decomposition as in Eq. \ref{lb2} also
exists for $A_{\overline{\{\eta_{1}\}}}(\triangle^{\prime},\mathcal{O}%
_{\overline{\{\eta_{1}\}}})$. Its field factors equal the residue fields of
the $\mathcal{O}_{i}$ in Eq. \ref{lb2}. In particular, up to the finite
extensions $(\ast)$, the $n$-local field structure of the $K_{i}$ in
$A_{X}(\triangle,\mathcal{O}_{X})$ is induced from%
\[%
\xymatrix{
{A_{\overline{\{\eta_{0}\}}}(\triangle,\mathcal{O}_{X})} \\
{A_{\overline{\{\eta_{0}\}}}(\triangle^{\prime},\mathcal{O}_{X})} \ar
[u] \ar[r]
& {A_{\overline{\{\eta_{1}\}}}(\triangle^{\prime},\mathcal{O}_{X})} \\
& {A_{\overline{\{\eta_{1}\}}}(\triangle^{\prime\prime},\mathcal{O}_{X})}
\ar[u] \ar[r]
& {A_{\overline{\{\eta_{2}\}}}(\triangle^{\prime\prime},\mathcal{O}_{X})} \\
& & \ar[u] \vdots}%
\]

\item For a coherent sheaf $\mathcal{F}$, $A(\triangle,\mathcal{F}%
)\cong\mathcal{F}\otimes_{\mathcal{O}_{X}}A(\triangle,\mathcal{O}_{X})$.
\end{enumerate}
\end{proposition}

\textit{Beware:} Even if $\triangle$ consists only of one flag, the products
in Eq. \ref{lb2} may have several factors. See Example
\ref{example_exceptionaladelefibering}.

The first published proof (of a mild variation) of the above result was given
by Yekutieli \cite[Thm. 3.3.2]{MR1213064}. We now have described the
multidimensional generalization of the infinite-dimensional $k$-vector space
$\widehat{\mathcal{K}}_{X,x}$ appearing in
\S \ref{marker_sect_TateOriginalConstruction}.\medskip

Next, we need to describe the higher analogues of the operator ideals
$I^{+},I^{-}$. Since these might seem quite involved, let us axiomatize the
precise input datum which the following constructions require:

\begin{definition}
[\cite{MR565095}]\label{BT_DefCubicallyDecompAlgebra}Let $k$ be a field. An
\emph{(}$n$\emph{-fold)} \emph{cubically decomposed algebra}\footnote{This
definition is slightly more general than in \cite[Definition 6]{MR3207578}
because we do not demand that $A$ is unital.} over $k$ is the datum
$(A,(I_{i}^{\pm}),\tau)$:

\begin{itemize}
\item an associative $k$-algebra $A$;

\item two-sided ideals $I_{i}^{+},I_{i}^{-}$ such that $I_{i}^{+}+I_{i}^{-}=A
$ for $i=1,\ldots,n$;

\item writing $I_{i}^{0}:=I_{i}^{+}\cap I_{i}^{-}$ and $I_{tr}:=I_{1}^{0}%
\cap\cdots\cap I_{n}^{0}$, a $k$-linear map (called \emph{trace})%
\[
\tau:I_{tr}/[I_{tr},A]\rightarrow k\text{.}%
\]

\end{itemize}
\end{definition}

The essence of Beilinson's residue construction uses nothing but the above
datum. The reader should therefore not be discouraged by the involved actual
construction of it:\medskip

Below $\operatorname*{Hom}\nolimits_{k}(-,-)$ refers to plain $k$-vector space
homomorphisms without any further conditions.

\begin{definition}
[\cite{MR565095}]\label{def_HigherAdeleOperatorIdeals}Suppose $X/k$ is a
finite type reduced scheme of pure dimension $n$.

\begin{enumerate}
\item Let $\triangle=\{(\eta_{0}>\cdots>\eta_{i})\}$ be given and $M$ a
finitely generated $\mathcal{O}_{\eta_{0}}$-module. Then a \emph{lattice} in
$M$ is a finitely generated $\mathcal{O}_{\eta_{1}}$-module $L\subseteq M $
such that $\mathcal{O}_{\eta_{0}}\cdot L=M$.

\item For any quasi-coherent sheaf $M$ on $X$ define $M_{\triangle
}:=A(\triangle,M)$.

\item Write $\triangle^{\prime}:=\left.  _{\eta_{0}}\triangle\right.
=\{(\eta_{1}>\cdots>\eta_{n})\}$. Suppose $M_{1},M_{2}$ are finitely generated
$\mathcal{O}_{\eta_{0}}$-modules. Let $\operatorname*{Hom}\nolimits_{\triangle
}(M_{1},M_{2})$ be the $k$-submodule of those $f\in\operatorname*{Hom}%
\nolimits_{k}(M_{1\triangle},M_{2\triangle})$ such that for all lattices
$L_{1}\subset M_{1},L_{2}\subset M_{2}$, there exist lattices $L_{1}^{\prime
}\subset M_{1},L_{2}^{\prime}\subset M_{2}$ such that%
\begin{equation}
L_{1}^{\prime}\subseteq L_{1},\qquad L_{2}\subseteq L_{2}^{\prime},\qquad
f(L_{1\triangle^{\prime}}^{\prime})\subseteq L_{2\triangle^{\prime}},\qquad
f(L_{1\triangle^{\prime}})\subseteq L_{2\triangle^{\prime}}^{\prime
}\label{TATEMATRIX_l9}%
\end{equation}
and for all such $L_{1},L_{1}^{\prime},L_{2},L_{2}^{\prime}$ the induced $k
$-linear map%
\begin{equation}
\overline{f}:(L_{1}/L_{1}^{\prime})_{\triangle^{\prime}}\rightarrow
(L_{2}^{\prime}/L_{2})_{\triangle^{\prime}}\label{TATEMATRIX_l12}%
\end{equation}
lies in $\operatorname*{Hom}\nolimits_{\triangle^{\prime}}(L_{1}/L_{1}%
^{\prime},L_{2}^{\prime}/L_{2})$. Define $\operatorname*{Hom}%
\nolimits_{\varnothing}(-,-) $ as $\operatorname*{Hom}\nolimits_{k}(-,-)$.

\item Define $I_{1\triangle}^{+}(M_{1},M_{2})$ to consist of those
$f\in\operatorname*{Hom}\nolimits_{\triangle}(M_{1},M_{2})$ such that there
exists a lattice $L\subset M_{2}$ with $f(M_{1\triangle})\subseteq
L_{\triangle^{\prime}}$. Respectively, $I_{1\triangle}^{-}(M_{1},M_{2})$
consists of those such that there exists a lattice $L\subset M_{1}$ with
$f(L_{\triangle^{\prime}})=0$. Next, for $i=2,\ldots,n$ and both $+/-$ define
$I_{i\triangle}^{\pm}(M_{1},M_{2})$ as those $f\in\operatorname*{Hom}%
\nolimits_{\triangle}(M_{1},M_{2})$ such that for all lattices $L_{1}%
,L_{1}^{\prime},L_{2},L_{2}^{\prime}$ as in Eq. \ref{TATEMATRIX_l9} we have%
\begin{equation}
\overline{f}\in I_{(i-1)\triangle^{\prime}}^{\pm}(L_{1}/L_{1}^{\prime}%
,L_{2}^{\prime}/L_{2})\text{.}\label{TATEMATRIX_l12p}%
\end{equation}

\end{enumerate}
\end{definition}

A discussion around this type of structure can be found in Osipov
\cite{MR2314612}. It can be related to topologizations of $n$-local fields
\cite{MR3161556}, \cite{MR1213064}. We refer the reader especially to
Yekutieli's work in the context of topological higher local fields
\cite{MR3317764}. Note the similarity to Definitions \ref{def_onedim_lattice}
and \ref{def_onedim_compactanddiscrete}. The above definition leads us to the
central object of study:

\begin{definition}
[\cite{MR565095}]In the context of the previous definition, let%
\[
E_{\triangle}:=\operatorname*{Hom}\nolimits_{\triangle}(\mathcal{O}_{\eta_{0}%
},\mathcal{O}_{\eta_{0}})\subseteq\operatorname*{End}\nolimits_{k}%
(\mathcal{O}_{X\triangle},\mathcal{O}_{X\triangle})\text{.}%
\]
Write $I_{i\triangle}^{\pm}\subseteq E_{\triangle}$ for $I_{i\triangle}^{\pm
}(\mathcal{O}_{\eta_{0}},\mathcal{O}_{\eta_{0}})$ and $i=1,\ldots,n $.
\end{definition}

\begin{example}
[toy example \cite{MR3207578}]\label{example_infinitematrices}The above
definition can easily be confusing. It is helpful to look at the structurally
simpler, but essentially equivalent case of infinite matrix algebras first:
For any associative algebra $R$ define%
\begin{equation}
E(R):=\{\phi=(\phi_{ij})_{i,j\in\mathbf{Z}},\phi_{ij}\in R\mid\exists K_{\phi
}:\left\vert i-j\right\vert >K_{\phi}\Rightarrow\phi_{ij}%
=0\}\label{TATEMATRIX_l1}%
\end{equation}
and equip it with the usual matrix multiplication. Then%
\begin{align*}
I^{+}(R):=  &  \{\phi\in E(R)\mid\exists B_{\phi}:i<B_{\phi}\Rightarrow
\phi_{ij}=0\}\\
I^{-}(R):=  &  \{\phi\in E(R)\mid\exists B_{\phi}:j>B_{\phi}\Rightarrow
\phi_{ij}=0\}
\end{align*}
define two-sided ideals in $E(R)$ with $I^{+}(R)+I^{-}(R)=E(R)$. We may
iterate this construction so that $I_{i}^{\pm}:=(EE\cdots I^{\pm}\cdots E)(R)
$ (with $I^{\pm}$ in the $i$-th place) defines a two-sided ideal of
$E^{n}(R)=E\cdots E(R)$. One checks that $(E^{n}R,\{I_{i}^{\pm}%
\},\operatorname*{tr})$ is an $n$-fold cubically decomposed algebra
\cite[\S 1.1]{MR3207578}.%
\[%
{\includegraphics[
height=1.9515in,
width=2.7031in
]%
{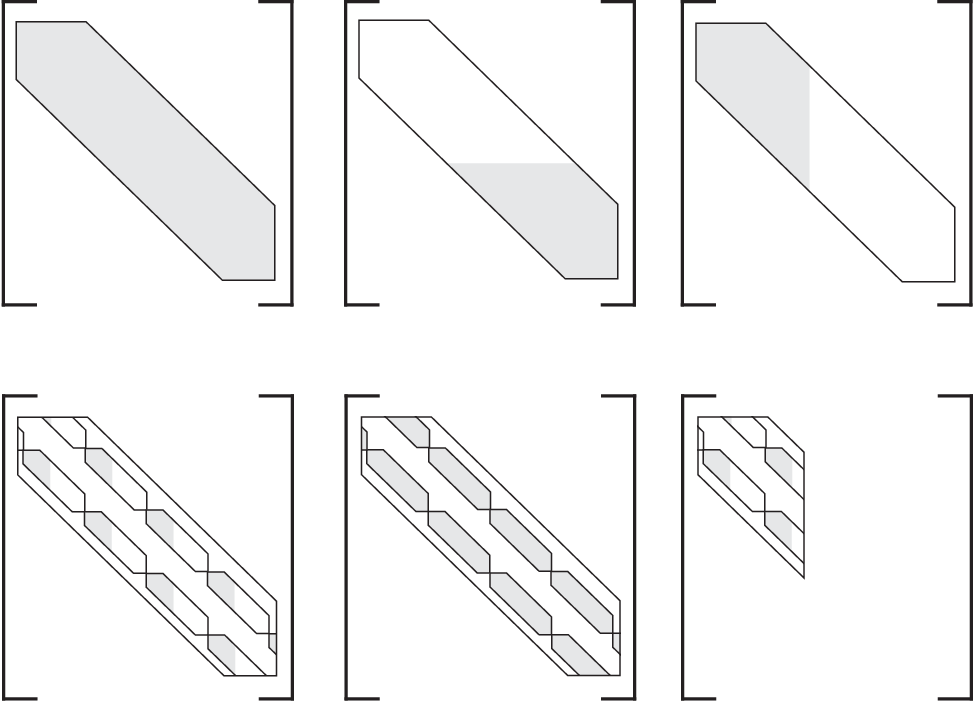}%
}%
\]
The top row displays typical matrices from $E(R)$, $I^{+}(R)$, $I^{-}(R)$
respectively. The lower row illustrates double infinite matrix constructions,
namely $E(I^{-}(R))$, $E(E(R))$ and $I^{-}(I^{-}(R))$ respectively. Although
defined in a more complicated way, the ideals of Definition
\ref{def_HigherAdeleOperatorIdeals} have the same structural properties as
these infinite matrix ideals. Note that $E^{n}(R)$ has a natural $R$-linear
action on the Laurent polynomial ring $R[t_{1}^{\pm},\ldots,t_{n}^{\pm}]$.
\end{example}

\begin{proposition}
[{\cite[Thm. (a)]{MR565095}}]%
\label{Prop_LocalEAtFullFlagIsCubicallyDecomposedAlgebra}Suppose $X/k$ is a
finite type reduced scheme of pure dimension $n$. Suppose $\triangle
=\{(\eta_{0}>\cdots>\eta_{n})\}$ is a single-element set such that
$\operatorname*{codim}_{X}\eta_{i}=i$.

\begin{enumerate}
\item Then $(E_{\triangle},(I_{i\triangle}^{\pm}),\operatorname*{tr}%
\nolimits_{I_{tr}})$ is a unital cubically decomposed algebra over $k$, where
$\operatorname*{tr}\nolimits_{I_{tr}}$ refers to Tate's operator trace (cf.
Prop. \ref{BT_PropTateTraceConstruction}).

\item For every $f\in I_{tr}$ there exists a finite-dimensional $f$-stable $k
$-vector subspace $W\subseteq E_{\triangle}$ such that $\operatorname*{tr}%
\nolimits_{I_{tr}}f=\operatorname*{tr}\nolimits_{W}f$.
\end{enumerate}
\end{proposition}

\begin{proof}
One easily sees that the $I_{i}^{\pm}$ are two-sided ideals. For $I_{i}%
^{+}+I_{i}^{-}=E_{\triangle}$ pick any lattice on the suitable level of the
inductive definition and any vector space idempotent projecting on it, call it
$P^{+}$. Then $P^{-}:=\mathbf{1}-P^{+}$ contains the lattice in the kernel.
Clearly, $\mathbf{1}=P^{+}+P^{-}$ and $P^{\pm}\in I_{i}^{\pm}$. It remains to
check that Tate's trace is defined on $I_{tr}=I_{1}^{0}\cap\cdots\cap
I_{n}^{0}$, i.e. that all operators in this ideal are finite-potent, one can
argue by induction: Suppose $f\in I_{tr}(V,V)$ for some $V$. In particular
$f\in I_{n}^{0}(V,V)$, i.e. there exists a lattice $L\subset V$ such that
$fL=0$ and a lattice $L^{\prime}\subset V$ such that $fV\subseteq L^{\prime}$.
We observe that $f^{\circ3^{n}}:V\rightarrow V$ factors as%
\begin{equation}
f^{\circ3^{n}}:V\overset{f^{\circ3^{n-1}}}{\longrightarrow}L^{\prime}%
\overset{\operatorname*{quot}}{\longrightarrow}\frac{L^{\prime}}{L\cap
L^{\prime}}\overset{\overline{f}^{\circ3^{n-1}}}{\longrightarrow}%
\frac{L^{\prime}}{L\cap L^{\prime}}\overset{f^{\circ3^{n-1}}}{\longrightarrow
}L^{\prime}\overset{\operatorname*{incl}}{\longrightarrow}V\text{.}\label{lb6}%
\end{equation}
As $L,L^{\prime}$ are lattices, $L\cap L^{\prime}$ is a lattice, so we may
take $L_{1}^{\prime}=L_{2}:=L\cap L^{\prime}$ and $L_{1}=L_{2}^{\prime
}:=L^{\prime}$ as choices in Eq. \ref{TATEMATRIX_l9}. As we also have $f\in
I_{n-1}^{0}$, this yields that $\overline{f}\in I_{n-1}^{0}(L^{\prime}/(L\cap
L^{\prime}),L^{\prime}/(L\cap L^{\prime}))$. Thus, using $V:=L^{\prime}/(L\cap
L^{\prime})$ the middle term $\overline{f}^{\circ3^{n-1}}$ in Eq. \ref{lb6}
again satisfies the assumptions for the induction step, just replace $n$ with
$n-1$. Proceed down to $n=1$, where the middle term $\overline{f}^{\circ1}$ is
a morphism of finite-dimensional $k$-vector spaces. Combining all induction
steps, this shows that for every $f\in I_{tr} $, $f^{\circ3^{n}}$ factors
through a finite-dimensional $k$-vector space $W $, so a power of $f$ indeed
has finite-dimensional image over $k$, i.e. $f$ is finite-potent. Similarly,
the computation of the trace can be reduced to a classical trace: Again, we
use induction. Assume $f\in I_{n}^{0} $. As the lattices $L,L^{\prime}$
(chosen as above) are $f$-stable, using axiom \textbf{T2} twice yields%
\[
\operatorname*{tr}\nolimits_{V}f=\operatorname*{tr}\nolimits_{L^{\prime}%
}f+\operatorname*{tr}\nolimits_{V/L^{\prime}}f=(\operatorname*{tr}%
\nolimits_{L\cap L^{\prime}}f+\operatorname*{tr}\nolimits_{L^{\prime}/L\cap
L^{\prime}}f)+\operatorname*{tr}\nolimits_{V/L^{\prime}}f\text{.}%
\]
As $\overline{f}\equiv0$ in the quotient $V/L^{\prime}$ as well as $f\mid
_{L}=0$ when restricted to $L$ (and thus $L\cap L^{\prime}$), axiom
\textbf{T3} reduces the above to $\operatorname*{tr}\nolimits_{L^{\prime
}/L\cap L^{\prime}}\overline{f}$. Hence, we have reduced to $\overline
{f}:L^{\prime}/(L\cap L^{\prime})\rightarrow L^{\prime}/(L\cap L^{\prime})$.
As before it follows that if we also have $f\in I_{n-1}^{0}(V,V)$, then
$\overline{f}\in I_{n-1}^{0}(L^{\prime}/(L\cap L^{\prime}),L^{\prime}/(L\cap
L^{\prime}))$ and using $V:=L^{\prime}/(L\cap L^{\prime})$ we again satisfy
our initial assumptions for the induction step. If $f\in I_{tr}$, this
inductively yields%
\[
\operatorname*{tr}\nolimits_{V}f=\cdots=\operatorname*{tr}\nolimits_{W}%
\overline{f}\text{,}%
\]
where $W$ is a finite-dimensional $k$-vector space. Hence, by \textbf{T1} the
last trace $\operatorname*{tr}\nolimits_{W}\overline{f}$ is the ordinary trace
of an endomorphism. For $f\in\lbrack I_{tr},A]$ use \textbf{T5} to see that
$\operatorname*{tr}_{V}f=0$.
\end{proof}

\section{\label{marker_BeilLocalConstruction}Beilinson's construction}

In this section we try to be brief. A motivated explanation can be found in
\S \ref{section_Etiology}.

\subsection{Beilinson's functional\label{subsect_BeilinsonLieFunctional}}

Let us recall Beilinson's construction of the cocycle \cite{MR565095}. We
begin with some general considerations:

\begin{definition}
\label{marker_DefMixedWedges}For $V$ a vector space and $V^{\prime}\subseteq
V$ a subspace, we define%
\[
V^{\prime}\wedge%
{\textstyle\bigwedge\nolimits^{r-1}}
V=\left\{
\begin{array}
[c]{l}%
\text{subspace of }%
{\textstyle\bigwedge\nolimits^{r}}
V\text{ generated by}\\
v^{\prime}\wedge v_{1}\wedge\cdots\wedge v_{r-1}\text{ with }v^{\prime}\in
V^{\prime}\text{, }v_{i}\in V
\end{array}
\right\}
\]

\end{definition}

\textit{Beware:} Note that $V^{\prime}\wedge(-)$ is by no means an exact
functor in any possible sense. It behaves quite differently from $V^{\prime
}\otimes(-)$.

Let $\mathfrak{g}:=A_{Lie}$ be the Lie algebra of an associative algebra $A$
and $M$ a $\mathfrak{g}$-module. Then one has the Chevalley-Eilenberg complex
$C_{i}^{\operatorname*{Lie}}(\mathfrak{g},M):=M\otimes%
{\textstyle\bigwedge\nolimits^{i}}
\mathfrak{g}$, see \cite[\S 10.1.3]{MR1217970} for details. Its homology is
ordinary Lie homology. We abbreviate $C_{i}^{\operatorname*{Lie}}%
(\mathfrak{g}):=C_{i}^{\operatorname*{Lie}}(\mathfrak{g},k)$ for trivial
coefficients. Let $\mathfrak{j}\subseteq\mathfrak{g}$ be a Lie ideal. Then the
vector spaces%
\begin{equation}
CE(\mathfrak{j})_{r}:=\mathfrak{j}\wedge%
{\textstyle\bigwedge\nolimits^{r-1}}
\mathfrak{g}\label{la10}%
\end{equation}
for $r\geq1$ and $CE(\mathfrak{j})_{0}:=k$ define a subcomplex of
$C_{r}^{\operatorname*{Lie}}(\mathfrak{g},k)$ via the identification%
\[
j\wedge f_{1}\wedge\cdots\wedge f_{r-1}\approx1\otimes j\wedge f_{1}%
\wedge\cdots\wedge f_{r-1}\text{.}%
\]
The differential turns into the nice expression (cf. \cite[first
equation]{MR565095})%
\begin{equation}
\delta(f_{0}\wedge f_{1}\wedge\ldots\wedge f_{r}):=%
{\textstyle\sum\nolimits_{0\leq i<j\leq r}}
(-1)^{i+j}[f_{i},f_{j}]\wedge f_{0}\wedge\ldots\widehat{f_{i}}\ldots
\widehat{f_{j}}\ldots\wedge f_{r}\text{.}\label{lBT_effectiveCEdifferential}%
\end{equation}
\textit{Beware:}\ Due to the difference between $\mathfrak{j}\wedge(-)$ and
$\mathfrak{j}\otimes(-)$ the homology of $CE(\mathfrak{j})_{\bullet}$ does
\textit{not} agree with the Lie homology $H_{n}(\mathfrak{g},\mathfrak{j})$
with $\mathfrak{j}$ seen as a $\mathfrak{g}$-module. It is better viewed as
relative Lie homology, as explained in \S \ref{section_Etiology}.\medskip

Now suppose $A$ is given the extra structure of a cubically decomposed algebra
(cf. Definition \ref{BT_DefCubicallyDecompAlgebra}), i.e.

\begin{itemize}
\item two-sided ideals $I_{i}^{+},I_{i}^{-}$ such that $I_{i}^{+}+I_{i}^{-}=A
$ for $i=1,\ldots,n$;

\item writing $I_{i}^{0}:=I_{i}^{+}\cap I_{i}^{-}$ and $I_{tr}:=I_{1}^{0}%
\cap\cdots\cap I_{n}^{0}$, a $k$-linear map%
\[
\tau:I_{tr}/[I_{tr},A]\rightarrow k\text{.}%
\]

\end{itemize}

For any elements $s_{1},\ldots,s_{n}\in\{+,-,0\}$ we define the \emph{degree
}$\deg(s_{1},\ldots,s_{n}):=1+\#\{i\mid s_{i}=0\}$. Given the above datum,
Beilinson constructs a very interesting family of complexes:

\begin{definition}
[\cite{MR565095}]Define%
\begin{equation}
\left.  ^{\wedge}T_{\bullet}^{p}\right.  :=\bigoplus_{\substack{s_{1}%
,\ldots,s_{n}\in\{\pm,0\} \\\deg(s_{1}\ldots s_{n})=p}}\bigcap_{i=1}%
^{n}\left\{
\begin{array}
[c]{ll}%
CE(I_{i}^{+})_{\bullet} & \text{for }s_{i}=+\\
CE(I_{i}^{-})_{\bullet} & \text{for }s_{i}=-\\
CE(I_{i}^{+})_{\bullet}\cap CE(I_{i}^{-})_{\bullet} & \text{for }s_{i}=0
\end{array}
\right. \label{lBT_DefComplexTWedge}%
\end{equation}
and $\left.  ^{\wedge}T_{\bullet}^{0}\right.  :=CE(\mathfrak{g})_{\bullet}$.
View them as complexes in the subscript index $(-)_{\bullet}$.
\end{definition}

Each $CE(I_{i}^{\pm})_{\bullet}$ is a complex and all their differentials are
defined by the same formula, namely Eq. \ref{lBT_effectiveCEdifferential}.
Thus, the intersection of these complexes has a well-defined differential and
is a complex itself. Next, Beilinson shows that%
\begin{equation}
0\longrightarrow\left.  ^{\wedge}T_{\bullet}^{n+1}\right.  \longrightarrow
\cdots\longrightarrow\left.  ^{\wedge}T_{\bullet}^{1}\right.  \longrightarrow
\left.  ^{\wedge}T_{\bullet}^{0}\right.  \longrightarrow0\label{lwcl9}%
\end{equation}
is an exact sequence (now indexed by the superscript) with respect to a
suitably defined differential coming from a structure as a cubical object (see
\cite[\S 1]{MR565095} or \cite[Lemma 18]{MR3207578}). Thus, we obtain a
bicomplex%
\begin{equation}%
\begin{array}
[c]{ccccccl}
&  &  &  & \rightarrow\cdots\rightarrow & \left.  ^{\wedge}T_{2}^{0}\right.  &
\rightarrow0\\
& \text{etc.} &  & \downarrow &  & \downarrow & \\
0\rightarrow & \left.  ^{\wedge}T_{1}^{n+1}\right.  & \rightarrow & \left.
^{\wedge}T_{1}^{n}\right.  & \rightarrow\cdots\rightarrow & \left.  ^{\wedge
}T_{1}^{0}\right.  & \rightarrow0\\
& \downarrow &  & \downarrow &  & \downarrow & \\
0\rightarrow & \left.  ^{\wedge}T_{0}^{n+1}\right.  & \rightarrow & \left.
^{\wedge}T_{0}^{n}\right.  & \rightarrow\cdots\rightarrow & \left.  ^{\wedge
}T_{0}^{0}\right.  & \rightarrow0\text{.}%
\end{array}
\label{la11}%
\end{equation}
Its support is horizontally bounded in degrees $[n+1,0]$, vertically
$(+\infty,0]$. As a result, the associated two bicomplex spectral sequences
are convergent. Since the rows are exact, the one with $E^{0}$-page
differential in direction `$\rightarrow$' vanishes already on the $E^{1}%
$-page. Thus, this (and therefore both) spectral sequences converge to zero.
Now focus on the second spectral sequence, the one with $E^{0}$-page
differential in direction `$\downarrow$'. Since $E_{\bullet,\bullet}^{n+2}=0$
by horizontal concentration in $[n+1,0]$, the differential $d:E_{n+1,1}%
^{n+1}\rightarrow E_{0,n+1}^{n+1}$ on the $\left(  n+1\right)  $-st page must
be an isomorphism. Upon composing its inverse with suitable edge maps,
Beilinson gets a morphism%
\begin{equation}
\phi_{Beil}:H_{n+1}(\mathfrak{g},k)\overset{\sim}{\longrightarrow}%
H_{n+1}(CE(\mathfrak{g}))\overset{\text{edge}}{\longrightarrow}E_{0,n+1}%
^{n+1}\underset{d^{-1}}{\overset{\sim}{\longrightarrow}}E_{n+1,1}%
^{n+1}\overset{\text{edge}}{\longrightarrow}H_{1}(^{\wedge}T_{\bullet}%
^{n+1})\overset{\tau}{\longrightarrow}k\text{.}\label{la29}%
\end{equation}
For the left-hand side isomorphism note that $H_{n+1}(\mathfrak{g},k)\cong
H_{n+1}(CE(\mathfrak{g}))$ just by definition of Lie homology
(\textit{Beware:} this is true for $CE(\mathfrak{j})$ if and only if
$\mathfrak{j}=\mathfrak{g}$), and $\left.  ^{\wedge}T_{\bullet}^{0}\right.
:=CE(\mathfrak{g})_{\bullet}$ by definition. For the right-hand side map
$\tau$ observe that
\begin{equation}
H_{1}(^{\wedge}T_{\bullet}^{n+1})=H_{1}(%
{\textstyle\bigcap\nolimits_{i=1}^{n}}
{\textstyle\bigcap\nolimits_{s=\{+,-\}}}
CE(I_{i}^{s})_{\bullet})=\frac{\mathfrak{j}}{[\mathfrak{j},\mathfrak{g}%
]}\label{la12}%
\end{equation}
for $\mathfrak{j}:=%
{\textstyle\bigcap\nolimits_{i=1}^{n}}
{\textstyle\bigcap\nolimits_{s=\{+,-\}}}
I_{i}^{s}=I_{tr}$. Using the Universal Coefficient Theorem in Lie algebra
homology, this is the same as giving an element in $H^{n+1}(\mathfrak{g}%
,k)\cong H_{n+1}(\mathfrak{g},k)^{\ast}$. This is the proof for Beilinson's
result \cite[Lemma 1 (a)]{MR565095}. We summarize:

\begin{proposition}
\label{Prop_BeilResidueMapConstructed}(Beilinson) For every cubically
decomposed algebra $(A,(I_{i}^{\pm}),\tau)$ and $\mathfrak{g}:=A_{Lie}$ there
is a canonical morphism%
\[
\phi_{Beil}:H_{n+1}(\mathfrak{g},k)\longrightarrow k\text{,}%
\]
or equivalently a canonical Lie cohomology class in $H^{n+1}(\mathfrak{g},k)$.
It is functorial in morphisms of cubically decomposed algebras.
\end{proposition}

Thus, if a commutative $k$-algebra $K$ embeds as $K\hookrightarrow A$, we get
a morphism%
\[
\operatorname*{res}:\Omega_{K/k}^{n}\overset{(\diamond)}{\longrightarrow
}H_{n+1}(\mathfrak{g},k)\overset{\phi_{Beil}}{\longrightarrow}k
\]%
\[
f_{0}\mathrm{d}f_{1}\wedge\cdots\wedge\mathrm{d}f_{n}\longmapsto f_{0}\wedge
f_{1}\wedge\cdots\wedge f_{n}\longmapsto\phi_{Beil}(f_{0}\wedge\cdots\wedge
f_{n})
\]
It turns out to be the residue. This is essentially \cite[Lemma 1 (b) and Thm.
(a)]{MR565095}. For a very explicit proof of this see \cite[Thm. 4 and Thm.
5]{MR3207578}. Note that $(\diamond)$ is not really a morphism; it does not
respect the relation $\mathrm{d}(xy)=x\mathrm{d}y+y\mathrm{d}x$. This washes
out after composing with $\phi_{Beil}$.

\begin{remark}
[reduces to Tate's theory]It is a general fact from homological algebra that
the connecting morphism coming from the snake lemma agrees with the inverse of
the suitable differential in the bicomplex spectral sequence applied to the
two-row bicomplex which one feeds into the snake lemma. If we apply this
remark to Eq. \ref{la6}, we readily see how Eq. \ref{la29} transforms into Eq.
\ref{la7}. This also justifies why $d^{-1}:E_{0,n+1}^{n+1}\rightarrow
E_{n+1,1}^{n+1}$ is a natural choice to consider.
\end{remark}

\section{\label{section_Etiology}Etiology}

I will try to explain how one could read Tate's original article and naturally
be led to Beilinson's generalization. Clearly, I\ am just writing down a
possible interpretation here and quite likely it has no connection whatsoever
with the actual development of the ideas. Since the original papers
\cite{MR0227171}, \cite{MR565095} say very little about the underlying
creative process, this might be of some use. Of course, logically, this
section is superfluous.\medskip

I would have liked to begin by explaining Cartier's idea. Tate writes
\textquotedblleft I arrived at this treatment of residues by considering the
special features of the one-dimensional case, after discussing with Mumford an
approach of Cartier to Grothendieck's higher dimensional residue
symbol\textquotedblright\ \cite[p. 1]{MR0227171}. Pierre Cartier told me that
he has never published his approach, it was only disseminated in seminar talks
by Adrien Douady, whom we sadly cannot ask anymore. It seems possible that the
original formulation of Cartier's method has fallen into oblivion. Similarly,
John Tate told me that he does not remember more about the history than what
is documented in his article. So allow me to take Tate's method for granted
and proceed to Beilinson's generalization.\medskip

Firstly, let us reformulate Tate's original construction. As explained in
\S \ref{marker_sect_TateOriginalConstruction}, it begins with an exact
sequence of Lie modules%
\begin{equation}
0\longrightarrow I^{0}\longrightarrow I^{+}\oplus I^{-}\longrightarrow
E\longrightarrow0\text{.}\label{lwcl4}%
\end{equation}
We may read $I^{+}\oplus I^{-}$ as a Lie algebra itself and hope for $I^{0}$
being a Lie ideal in there, so that we could view the sequence as an extension
of Lie algebras. However, this fails (e.g. $[x\oplus x,a\oplus b]=[x,a]\oplus
\lbrack x,b]$ has no reason to be diagonal). There is an easy remedy, we
quotient out
\begin{equation}
0\longrightarrow I^{0}\longrightarrow I^{+}\oplus I^{-}\longrightarrow\left(
I^{+}\oplus I^{-}\right)  /I^{0}\longrightarrow0\label{lwcl2}%
\end{equation}
by $I^{-}$ and push the sequence out along the quotient map, giving%
\begin{equation}
0\longrightarrow I^{0}\overset{i}{\longrightarrow}I^{+}\overset{j}%
{\longrightarrow}I^{+}/I^{0}\longrightarrow0\text{.}\label{lwcl3}%
\end{equation}
Now $I^{0}$ is indeed a Lie ideal in $I^{+}$ so that this is an extension of
Lie algebras. We may take the homology of Lie algebras with trivial
coefficients, i.e. $H_{i}(-):=H_{i}(-,k)$. If $C_{i}^{\operatorname*{Lie}}(-)$
denotes the underlying Chevalley-Eilenberg complex, we get an obvious induced
morphism $j_{\ast}:C_{i}^{\operatorname*{Lie}}(I^{+})\rightarrow
C_{i}^{\operatorname*{Lie}}(I^{+}/I^{0})$, which we would like to fit into a
long exact sequence. To this end, define \emph{relative Lie homology}
$H_{i}(I^{+}\left.  \mathsf{rel}\right.  I^{0})$ simply as the co-cone of this
morphism $j_{\ast}$, so that we get a long exact sequence%
\begin{equation}
\cdots\rightarrow H_{i+1}(I^{+}/I^{0})\overset{d}{\rightarrow}H_{i}%
(I^{+}\left.  \mathsf{rel}\right.  I^{0})\rightarrow H_{i}(I^{+})\rightarrow
H_{i}(I^{+}/I^{0})\overset{d}{\rightarrow}\cdots\text{.}\label{lwcl1}%
\end{equation}

\begin{remark}
\label{rmk_DontUseCoefficients}This is not to be confused with the long exact
sequence in Lie homology $H_{i}(E,-)$ coming from viewing Eq. \ref{lwcl3} as a
short exact sequence of coefficient modules. In Eq. \ref{lwcl1} we change the
Lie algebra, not the coefficients.
\end{remark}

It would be nice to have a more explicit description of the relative homology
groups. Instead of just defining them as an abstract co-cone of complexes,
define it (quasi-isomorphically) as the kernel of the map $j_{\ast}$ of
Chevalley-Eilenberg complexes. Explicitly, this means that it is the kernel in%
\begin{equation}
0\rightarrow C_{i}^{\operatorname*{Lie}}(I^{+}\left.  \mathsf{rel}\right.
I^{0})\rightarrow\bigwedge\nolimits^{i}I^{+}\rightarrow\bigwedge
\nolimits^{i}I^{+}/I^{0}\rightarrow0\text{.}\label{lwcl6}%
\end{equation}
We see that $C_{i}^{\operatorname*{Lie}}(I^{+}\left.  \mathsf{rel}\right.
I^{0})=I^{0}\wedge\bigwedge\nolimits^{i-1}I^{+}$, the subspace spanned by
those exterior tensors with at least one slot lying in $I^{0}$; see Definition
\ref{marker_DefMixedWedges}. Next, let us address the question to compute the
connecting homomorphism $d$ in Eq. \ref{lwcl1}. Recall that it is constructed
by spelling out the underlying complexes and applying the snake lemma. In the
homological degree $H_{2}\overset{d}{\rightarrow}H_{1}$, this unravels as the
snake map of%
\begin{equation}
\xymatrix{ 0 \ar[r] & I^{0}\wedge I^{+} \ar[r] \ar[d]^{[-,-]} & I^{+}\wedge I^{+} \ar[r] \ar[d]^{[-,-]} & (I^{+}/I^{0})\wedge (I^{+}/I^{0}) \ar[r] \ar[d]^{[-,-]} & 0 \\ 0 \ar[r] & I^{0} \ar[r] & I^{+} \ar[r] & I^{+}/I^{0} \ar[r] & 0 }\label{lwcl10}%
\end{equation}
and by comparison with Diagram \ref{la6} we find that the connecting
homomorphism%
\begin{equation}
H_{2}(I^{+}/I^{0})\longrightarrow H_{1}(I^{+}\left.  \mathsf{rel}\right.
I^{0})\label{lwcl5}%
\end{equation}
agrees (after precomposing with $E\cong\left(  I^{+}\oplus I^{-}\right)
/I^{0}\twoheadrightarrow I^{+}/I^{0}$) with the snake map used in Tate's
construction, see Eq. \ref{la7}. We leave it to the reader to spell this out
in detail. In summary:\ Tate's residue can be read as a connecting
homomorphism in relative Lie homology.\medskip

In the one-dimensional theory we have the notion of a lattice as in Definition
\ref{def_onedim_lattice}, e.g. these are the%
\[
t^{i}k[[t]]\subset k((t))
\]
for any $i\in\mathbf{Z}$ $-$ here we temporarily allow ourselves to use
explicit coordinates for the sake of exposition. As we proceed to the
two-dimensional theory, the analogue of $k((t))$ will look like $k((s))((t))$
and we get a more complicated pattern of lattices: First of all, there are the
\textquotedblleft$t$-lattices\textquotedblright\ like $t^{i}k((s))[[t]]$ and
the quotient of any two such $t$-lattices, say of the pair%
\[
t^{i}k((s))[[t]]\subset t^{j}k((s))[[t]]\qquad\text{with}\qquad j\leq
i\text{,}%
\]
is a finite-dimensional $k((s))$-vector space; in this example it is the span%
\[
\simeq k((s))\left\langle t^{j},t^{j+1},\ldots,t^{i-1}\right\rangle \text{.}%
\]
In any such space we now get a notion of an \textquotedblleft$s$%
-lattice\textquotedblright, namely just in the previous sense, e.g. if $i=j+1
$ the quotient is just the span $\simeq k((s))\left\langle t^{j}\right\rangle
$ and the $s$-lattices would be of the shape $s^{i}k[[s]]\left\langle
t^{j}\right\rangle \subset k((s))\left\langle t^{j}\right\rangle $ for any
$i\in\mathbf{Z}$. Two things are important to keep in mind here:\newline
Firstly, for the sake of presentation we have described this in explicit
coordinates here. Of course we need to replace the vague notion of
\textquotedblleft$t$-lattices\textquotedblright\ and \textquotedblleft%
$s$-lattices\textquotedblright\ by something which makes no reference to
coordinates. See Definition \ref{def_HigherAdeleOperatorIdeals} for
Beilinson's beautiful solution.\newline Secondly, there is a true asymmetry
between $t$ and $s$. Note that for a field $k((s))((t))$ the roles of $s$ and
$t$ are not interchangeable, unlike for $k[[s]][[t]]$. For example,
$\sum_{i\geq0}s^{-i}t^{i}$ lies in this field, but $\sum_{i\geq0}t^{-i}s^{i}$
does not describe an actual element of $k((s))((t))$. This is why we chose to
speak of \textquotedblleft$s$-lattices\textquotedblright\ in a
\textit{quotient} of $t$-lattices, rather than trying to deal with something
like $s^{i}k[[s]]((t))$. Note for example that $\bigcup_{i\in\mathbf{Z}}%
s^{i}k[[s]]((t))\subsetneqq k((s))((t))$. To avoid all pitfalls, it would be
best to work in appropriate categories of ind-pro limits right from the start,
as in \cite{MR3510209}.\medskip

Based on having two lattice structures instead of just one, in dimension two
Beilinson deals with four ideals $I_{1}^{\pm},I_{2}^{\pm}$ instead of just a
single pair as in Tate's construction. We may read the exact sequence in Eq.
\ref{lwcl4} as a quasi-isomorphism%
\[
\left[  I^{0}\longrightarrow I^{+}\oplus I^{-}\right]  _{1,0}\overset{\sim
}{\longrightarrow}E
\]
with a two-term complex concentrated in homological degrees $[1,0]$. View
these ideals as representing the $t$-lattices of above (e.g. $I_{1}^{+}$ would
be endomorphisms whose image lies in some $t$-lattice). Then replicating the
analogous structure for $s$-lattices leads to the bicomplex%
\[%
\begin{bmatrix}
I_{1}^{0}\cap I_{2}^{0} & \longrightarrow & I_{1}^{0}\cap I_{2}^{+}\oplus
I_{1}^{0}\cap I_{2}^{-}\\
\downarrow &  & \downarrow\\
I_{1}^{+}\cap I_{2}^{0}\oplus I_{1}^{-}\cap I_{2}^{0} & \longrightarrow &
I_{1}^{+}\oplus I_{1}^{-}\oplus I_{2}^{+}\oplus I_{2}^{-}%
\end{bmatrix}
\overset{\sim}{\longrightarrow}E\text{.}%
\]
Accordingly, in the theory for $n$ dimensions one gets a structure of $n$
cascading notions of lattices, and correspondingly $2^{n}$ ideals $I_{i}^{\pm
}$. The above gets replaced by a quasi-isomorphism with an $n$-hypercube. It
is a matter of taste whether one prefers to work with multi-complexes or with
the ordinary total complex. We prefer the latter, giving a complex
concentrated in homological degrees $[n+1,0]$, see Eq. \ref{lwcl9} and Eq.
\ref{la16}.\medskip

In order to construct the residue map in dimension two, it seems natural to
perform the mechanism of dimension one twice, once for each layer of lattices.
Hence, one should study the connecting homomorphism analogous to the one in
Eq. \ref{lwcl5}. However, things get a bit more complicated, because if we try
to compose two such connecting homomorphisms, we find that the input of the
second step should be the relative Lie homology group which is the output of
the first step. This leads to \emph{bi-relative Lie homology}, defined just as
the kernel on the left-hand side in%
\[
0\rightarrow C_{i}^{\operatorname*{Lie}}(I_{1}^{+}\left.  \mathsf{rel}\right.
I_{1}^{0}\left.  \mathsf{rel}\right.  I_{2}^{0})\rightarrow C_{i}%
^{\operatorname*{Lie}}(I_{1}^{+}\left.  \mathsf{rel}\right.  I_{1}%
^{0})\rightarrow C_{i}^{\operatorname*{Lie}}(I_{1}^{+}/I_{2}^{0}\left.
\mathsf{rel}\right.  I_{1}^{0}/I_{2}^{0})\rightarrow0\text{.}%
\]
Here we allow ourselves to write $I_{1}^{+}/I_{2}^{0}$ as a shorthand for
$\frac{I_{1}^{+}}{I_{2}^{0}\cap I_{1}^{+}}$ to improve legibility. Now we are
able to compose the associated connecting homomorphism with the one of Eq.
\ref{lwcl5}, giving something like%
\[
H_{3}(I^{+}/I_{1}^{0}I_{2}^{0})\overset{d}{\longrightarrow}H_{2}(I^{+}%
/I_{2}^{0}\left.  \mathsf{rel}\right.  I_{1}^{0})\overset{d}{\longrightarrow
}H_{1}(I^{+}\left.  \mathsf{rel}\right.  I_{1}^{0}\left.  \mathsf{rel}\right.
I_{2}^{0})\text{.}%
\]
We should make the bi-relative Lie homology more explicit: Unwinding complexes
as in Eq. \ref{lwcl6}, we see that%
\[
0\rightarrow C_{i}^{\operatorname*{Lie}}(I_{1}^{+}\left.  \mathsf{rel}\right.
I_{1}^{0}\left.  \mathsf{rel}\right.  I_{2}^{0})\rightarrow I_{1}^{0}%
\wedge\bigwedge\nolimits^{i-1}I_{1}^{+}\rightarrow I_{1}^{0}/I_{2}^{0}%
\wedge\bigwedge\nolimits^{i-1}(I_{1}^{+}/I_{2}^{0})\rightarrow0
\]
and therefore%
\begin{equation}
C_{i}^{\operatorname*{Lie}}(I_{1}^{+}\left.  \mathsf{rel}\right.  I_{1}%
^{0}\left.  \mathsf{rel}\right.  I_{2}^{0})=\bigcap_{i=1,2}\left(  I_{i}%
^{0}\wedge\bigwedge\nolimits^{i-1}I_{1}^{+}\right)  \text{.}\label{lwcl7}%
\end{equation}
The reader will have no difficulty in checking that $i$-fold multi-relative
Lie homology can be defined accordingly, and leads to further intersections of
subcomplexes as in Eq. \ref{lwcl7}. This explains the underlying structure of
Beilinson's complex $\left.  ^{\wedge}T_{\bullet}^{p}\right.  $, see Eq.
\ref{lBT_DefComplexTWedge}. In fact, $\left.  ^{\wedge}T_{\bullet}^{p}\right.
$ is a tiny bit more complicated because it works with all $2^{n}$ ideals
$I_{i}^{\pm}$ and $E$ instead of quotienting out the $I^{-}$-ideals and
working with $I^{+}$ only, i.e. without the simplification coming from
switching from Eq. \ref{lwcl2} to Eq. \ref{lwcl3}.\medskip

Let us pause for a second. What happens if we ignore Remark
\ref{rmk_DontUseCoefficients} and phrase Tate's construction in terms of a
long exact sequence, this time with varying coefficients? The diagram
\ref{lwcl10}\ turns into%
\[
\xymatrix{ 0 \ar[r] & I^{0}\otimes E \ar[r] \ar[d]^{[-,-]} & I^{+}\otimes E \ar[r] \ar[d]^{[-,-]} & (I^{+}/I^{0})\otimes E \ar[r] \ar[d]^{[-,-]} & 0 \\ 0 \ar[r] & I^{0} \ar[r] & I^{+} \ar[r] & I^{+}/I^{0} \ar[r] & 0 }
\]
and Eq. \ref{lwcl5} gets replaced by%
\[
H_{1}(E,I^{+}/I^{0})\longrightarrow H_{0}(E,I^{0})\text{.}%
\]
Besides the index shift, this map also gives Tate's residue\footnote{I find it
noteworthy that essentially the same computation admits at least two (quite
different) homological interpretations.}. Hence, it is actually possible to
set up the entire theory using Lie homology with coefficients instead of
relative Lie homology. This is the path taken in the previous paper
\cite{MR3207578}; the corresponding variant of Beilinson's complex $\left.
^{\wedge}T_{\bullet}^{p}\right.  $ is called $\left.  ^{\otimes}T_{\bullet
}^{p}\right.  $ in \textit{loc. cit.} Both variants in general give different
maps (and begin and end in different homology groups), but still they are
largely compatible \cite[Lemma 23]{MR3207578} and both give the
multi-dimensional residue \cite[Thm. 4 and 5]{MR3207578}.

The coefficient variant is more manageable for explicit computations: The
problem with complexes like $I^{0}\wedge\bigwedge\nolimits^{i-1}I^{+}$ is that
it is difficult to write down explicit bases for these spaces because the only
natural candidate are pure tensors%
\[
f_{0}\otimes f_{1}\otimes\cdots\otimes f_{i-1}%
\]
with $f_{0},\ldots,f_{i-1}$ ascendingly taken from an ordered basis of $I^{+}
$ so that $f_{0}\in I^{0}$.\ Performing calculations, it quickly becomes very
tedious to maintain elements in this standard ordered shape.\medskip

In the next section \S \ref{section_HochschildUnitalIntro} we propose yet
another point of view. First of all, motivated by the strong relation between
the Hodge $n$-part of Hochschild homology and Lie homology, we replace Lie
homology by (the full)\ Hochschild homology. This poses no problem since all
the Lie algebras/ideals we have encountered above are actually coming from
associative algebras and ordinary ideals. For example, the sequence in Eq.
\ref{lwcl1} will be replaced by%
\[
\cdots\rightarrow HH_{i+1}(I^{+}/I^{0})\overset{d}{\rightarrow}HH_{i}%
(I^{+}\left.  \mathsf{rel}\right.  I^{0})\rightarrow HH_{i}(I^{+})\rightarrow
HH_{i}(I^{+}/I^{0})\overset{d}{\rightarrow}\cdots\text{.}%
\]
However, now a substantial simplification occurs: In certain circumstances
relative Hochschild homology agrees with absolute Hochschild homology, in the
sense that the natural morphism%
\[
HH_{i}(I^{0})\longrightarrow HH_{i}(I^{+}\left.  \mathsf{rel}\right.  I^{0})
\]
sometimes happens to be an isomorphism. This is known as \emph{excision}; it
is easily seen to be wrong for arbitrary ideals but it turns out that the
ideals $I_{i}^{0}$ have the necessary property. This spares us from having to
work with multi-relative homology at all. Instead, we can just compose the
corresponding $n$ connecting maps, one by one, and we will prove that this
again gives the same map, but now its construction necessitates much less
effort. We will also see that it is much easier to compute this map
explicitly, saving us from a lot of trouble we had to go through in
\cite{MR3207578}.

\section{Hochschild and cyclic picture\label{section_HochschildUnitalIntro}}

In this section we will formulate an analogue of Beilinson's construction in
the context of Hochschild (and later also cyclic) homology. We follow the
natural steps:

\begin{enumerate}
\item We replace Lie homology with Hochschild homology. This is harmless since
cubically decomposed algebras come with an associative product structure
anyway. There is a natural map%
\[
\varepsilon:H_{\bullet}(A_{Lie},M_{Lie})\longrightarrow H_{\bullet
}(A,M)\text{,}%
\]
ultimately explaining numerous similarities.

\item The Hochschild complex is modelled on chain groups $A\otimes
\cdots\otimes A$ instead of exterior powers. Thus, the only reasonable
replacement of the mixed exterior powers/relative homology groups%
\[
CE(\mathfrak{j})_{r}:=\mathfrak{j}\wedge%
{\textstyle\bigwedge\nolimits^{r-1}}
\mathfrak{g}%
\]
in the original construction are the groups $J\otimes A\otimes\cdots\otimes A
$ for $J$ an ideal. This is very convenient, as this just gives Hochschild
homology with coefficients $H_{r}(A,J)$. Alternatively, one could work with
relative\ Hochschild groups. We will return to a relative perspective in
\S \ref{marker_SectAlternativeApproach}.
\end{enumerate}

To set up notation, let us very briefly recall the necessary structures in
Hochschild homology. See \cite[Ch. I]{MR1217970} for a detailed treatment.
Suppose $A$ is an arbitrary (not necessarily unital) associative $k$-algebra.
Let $M$ be an $A$-bimodule over $k$, or equivalently a left-$A\otimes
_{k}A^{\operatorname*{op}}$-module. Define chain groups $C_{i}(A,M):=M\otimes
_{k}A^{\otimes i}$ and a differential $b:C_{i}(A,M)\rightarrow C_{i-1}(A,M)$,
given by%
\begin{align}
m\otimes a_{1}\otimes\cdots\otimes a_{i}  &  \mapsto ma_{1}\otimes
a_{2}\otimes\cdots\otimes a_{i}\nonumber\\
&  +%
{\textstyle\sum\nolimits_{j=1}^{i-1}}
\left(  -1\right)  ^{j}m\otimes a_{1}\otimes\cdots\otimes a_{j}a_{j+1}%
\otimes\cdots\otimes a_{i}\label{la27}\\
&  +\left(  -1\right)  ^{i}a_{i}m\otimes a_{1}\otimes\cdots\otimes
a_{i-1}\text{.}\nonumber
\end{align}
We call the homology of the complex $(C_{\bullet}(A,M),b)$ its
\emph{Hochschild homology}, denoted by $H_{i}(A,M)$. Write $A_{Lie}$ for the
Lie algebra associated to $A$ via $[x,y]:=x\cdot y-y\cdot x$. There is a
canonical morphism%
\begin{align}
\varepsilon:C_{i}^{\operatorname*{Lie}}(A_{Lie},M_{Lie})  &  \rightarrow
C_{i}(A,M)\label{lb1}\\
m\otimes a_{1}\wedge\cdots\wedge a_{i}  &  \mapsto m\otimes\sum_{\pi
\in\mathfrak{S}_{i}}\operatorname*{sgn}(\pi)a_{\pi^{-1}(1)}\otimes
\cdots\otimes a_{\pi^{-1}(i)}\text{,}\nonumber
\end{align}
where $\mathfrak{S}_{i}$ is the symmetric group on $i$ letters. This is a
morphism of complexes, in particular it induces a morphism $H_{i}%
(A_{Lie},M_{Lie})\rightarrow H_{i}(A,M)$.

For the rest of this section assume $A$ is \textit{unital}. Clearly $A$ is a
bimodule over itself and we write $HH_{i}(A):=H_{i}(A,A)$ as an abbreviation
(see \S \ref{marker_sect_ReminderHochschildNonUnital} for the correct
definition when $A$ is not unital). A $k$-algebra morphism $f:A\rightarrow
A^{\prime}$ induces a map $f_{\ast}:HH_{i}(A)\rightarrow HH_{i}(A^{\prime})$.
The motivation for using Hochschild homology in the context of residue theory
stems from the following famous isomorphism:

\begin{proposition}
(Hochschild-Kostant-Rosenberg) Suppose $A/k$ is a commutative smooth
$k$-algebra. Then the morphism%
\begin{align}
\Omega_{A/k}^{n} &  \longrightarrow HH_{n}(A)\nonumber\\
f_{0}\mathrm{d}f_{1}\wedge\cdots\wedge\mathrm{d}f_{n} &  \longmapsto\sum
_{\pi\in\mathfrak{S}_{n}}\operatorname*{sgn}(\pi)f_{0}\otimes f_{\pi^{-1}%
(1)}\otimes\cdots\otimes f_{\pi^{-1}(n)}\label{lb5}%
\end{align}
is an isomorphism of graded commutative algebras.
\end{proposition}

See \cite[Thm. 3.4.4]{MR1217970}. Let us now assume that $\mathbf{Q}\subseteq
k$: On $A^{\otimes(i+1)}$ recall that there is an action by Connes' cyclic
permutation operator%
\[
t:a_{0}\otimes a_{1}\otimes\cdots\otimes a_{i}\mapsto\left(  -1\right)
^{i}a_{i}\otimes a_{0}\otimes a_{1}\otimes\cdots\otimes a_{i-1}\text{.}%
\]
Define the cyclic chain groups by $CC_{i}(A):=A^{\otimes(i+1)}/(1-t)$; this is
the quotient by the action of $t$ on pure tensors. As was discovered by
Connes, it turns out that the differential $b$ remains well-defined on these
quotients. Its homology is known as \emph{cyclic homology} and denoted by
$HC_{i}(A)$. We shall also need \emph{Connes' periodicity sequence} \cite[Thm.
2.2.1]{MR1217970}: There is a long exact sequence
\begin{equation}
\cdots\longrightarrow HH_{i}(A)\overset{I}{\longrightarrow}HC_{i}%
(A)\overset{S}{\longrightarrow}HC_{i-2}(A)\overset{B}{\longrightarrow}%
HH_{i-1}(A)\longrightarrow\cdots\label{lml_35}%
\end{equation}
where $I$ is induced from the obvious inclusion/quotient map on the level of complexes.

\begin{remark}
At the expense of a more complicated definition of the cyclic chain groups,
all of these facts remain available without the simplifying assumption
$\mathbf{Q}\subset k$; see \cite[Thm. 2.1.5, we work with $H^{\lambda}$ of
\textit{loc. cit.}]{MR1217970}. We leave the necessary modifications to the reader.
\end{remark}

We shall moreover employ the map (recall that $\mathfrak{g}:=A_{Lie}$)%
\begin{align}
I^{\prime}:H_{n}(\mathfrak{g},\mathfrak{g})  &  \longrightarrow H_{n+1}%
(\mathfrak{g},k)\label{l_mapIprime}\\
f_{0}\otimes f_{1}\wedge\cdots\wedge f_{n}  &  \longmapsto(-1)^{n}\otimes
f_{0}\wedge\cdots\wedge f_{n}\nonumber
\end{align}
in Lie homology. The $(-1)^{n}$ is needed to make the differentials compatible.

\begin{proposition}
\label{prop_ConnesLodayQuillenHCVersionOfHKRIso}(Connes, Loday-Quillen)
Suppose $A/k$ is a commutative smooth $k$-algebra and $\operatorname*{char}%
k=0$. Then there is a canonical isomorphism%
\[
HC_{n}(A)\rightarrow\Omega_{A/k}^{n}/\mathrm{d}\Omega_{A/k}^{n-1}\oplus%
{\textstyle\bigoplus\nolimits_{i\geq1}}
H_{\operatorname*{dR}}^{n-2i}(A)
\]
so that $I:HH_{n}(A)\rightarrow HC_{n}(A)$ identifies with the quotient map
$\Omega_{A/k}^{n}\rightarrow\Omega_{A/k}^{n}/\mathrm{d}\Omega_{A/k}^{n-1}$ and
zero on the lower deRham summands.
\end{proposition}

See \cite[Thm. 3.4.12 and remark]{MR1217970}. The direct summand decomposition
on the right-hand side can be identified with the Hodge decomposition of
cyclic homology due to Gerstenhaber and\ Schack \cite{MR917209}.

\subsection{Hochschild setup\label{subsect_HochschildConstruction}}

Let $A$ be a cubically decomposed algebra over $k$. We define $A$-bimodules
$N^{0}:=A$ and for $p\geq1$%
\begin{equation}
N^{p}:=\bigoplus\nolimits_{\substack{s_{1},\ldots,s_{n}\in\{+,-,0\}
\\\deg(s_{1},\ldots,s_{n})=p}}I_{1}^{s_{1}}\cap I_{2}^{s_{2}}\cap\cdots\cap
I_{n}^{s_{n}}\label{lBTA_12}%
\end{equation}
with degree\emph{\ }$\deg(s_{1},\ldots,s_{n}):=1+\#\{i\mid s_{i}=0\}$ as
before. Each $I_{i}^{\pm}$ is a two-sided ideal and thus an $A$-bimodule.

We shall denote the components $f=(f_{s_{1}\ldots s_{n}})$ of elements in
$N^{p}$ with indices in terms of $s_{1},\ldots,s_{n}\in\{+,-,0\}$. Clearly
$N^{p}=0$ for $p>n+1$. We get an exact sequence of $A$-bimodules%
\begin{equation}
0\longrightarrow N^{n+1}\overset{\partial}{\longrightarrow}N^{n}%
\overset{\partial}{\longrightarrow}\cdots\overset{\partial}{\longrightarrow
}N^{0}\longrightarrow0\label{la16}%
\end{equation}
by using the following differential%
\begin{align*}
\left(  \partial f\right)  _{s_{1}\ldots s_{n}}:=  &
{\textstyle\sum\limits_{\{i\mid s_{i}=+,-\}}}
\left(  -1\right)  ^{\#\left\{  j\mid j>i\text{ and }s_{j}=0\right\}
}f_{s_{1}\ldots0\ldots s_{n}}\text{ (for }N^{i}\rightarrow N^{i-1}\text{,
}i\geq2\text{)}\\
\partial f:=  &
{\textstyle\sum\limits_{s_{1}\ldots s_{n}\in\{+,-\}}}
\left(  -1\right)  ^{s_{1}+\cdots+s_{n}}f_{s_{1}\ldots s_{n}}\text{ (for
}N^{1}\rightarrow N^{0}\text{)}%
\end{align*}
It is straight-forward to check that $\partial^{2}=0$ holds, but more details
are found in \cite[\S 4]{MR3207578} nonetheless. As tensoring with $(-)\otimes
A^{\otimes(r-1)}$ is exact, we can functorially take the Hochschild complex
and obtain a bicomplex with exact rows, fairly similar to the bicomplex that
we have encountered before in Eq. \ref{la11},%
\begin{equation}%
\begin{array}
[c]{ccccccc}
&  &  &  & \rightarrow\cdots\rightarrow & C_{2}(A,N^{0}) & \rightarrow0\\
& \text{etc.} &  & \downarrow &  & \downarrow & \\
0\rightarrow & C_{1}(A,N^{n+1}) & \rightarrow & C_{1}(A,N^{n}) &
\rightarrow\cdots\rightarrow & C_{1}(A,N^{0}) & \rightarrow0\\
& \downarrow &  & \downarrow &  & \downarrow & \\
0\rightarrow & C_{0}(A,N^{n+1}) & \rightarrow & C_{0}(A,N^{n}) &
\rightarrow\cdots\rightarrow & C_{0}(A,N^{0}) & \rightarrow0
\end{array}
\label{la15}%
\end{equation}
As before its support is horizontally bounded in degrees $[n+1,0]$, vertically
$(+\infty,0]$; we get an analogous differential on the $E^{n+1}$-page, which
is an isomorphism. Proceeding as before, \textit{but this time considering
degree }$n$\textit{\ instead of }$n+1$, we obtain%
\begin{equation}
\phi_{HH}:HH_{n}(A)\overset{\sim}{\longrightarrow}H_{n}(A,N^{0})\overset
{\text{edge}}{\longrightarrow}E_{0,n}^{n+1}\underset{d^{-1}}{\overset{\sim
}{\longrightarrow}}E_{n+1,0}^{n+1}\overset{\text{edge}}{\longrightarrow}%
H_{0}(A,N^{n+1})\overset{\tau}{\longrightarrow}k\text{.}\label{la14}%
\end{equation}
The consideration with the trace $\tau$ of the cubically decomposed algebra is
exactly the same as before since%
\[
H_{0}(A,N^{n+1})=\frac{N^{n+1}}{[N^{n+1},A]}\text{,}%
\]
but $N^{n+1}=I_{1}^{s_{1}}\cap I_{2}^{s_{2}}\cap\cdots\cap I_{n}^{s_{n}%
}=I_{tr}$, so we obtain exactly the same object as in the Lie counterpart, see
Eq. \ref{la12}. In particular, the trace $\tau$ is applicable for the same
reasons as before. This leads to the following new construction:

\begin{proposition}
For every cubically decomposed algebra $(A,(I_{i}^{\pm}),\tau)$ over $k$,
there is a canonical morphism%
\[
\phi_{HH}:HH_{n}(A)\longrightarrow k\text{.}%
\]
It is functorial in morphisms of cubically decomposed algebras.
\end{proposition}

Let us explain how to obtain an explicit formula for the fairly abstract
construction of $\phi_{HH}$. To this end, we employ the following tool from
the theory of spectral sequences:

\begin{lemma}
[{\cite[Lemma 19]{MR3207578}}]Suppose we are given a bounded exact sequence%
\[
S^{\bullet}=[S^{n+1}\rightarrow S^{n}\rightarrow\cdots\rightarrow
S^{0}]_{n+1,0}%
\]
of bounded below complexes of $k$-vector spaces; or equivalently a
correspondingly bounded bicomplex.

\begin{enumerate}
\item There is a second quadrant homological spectral sequence $(E_{p,q}%
^{r},d_{r})$ converging to zero such that%
\[
E_{p,q}^{1}=H_{q}(S_{\bullet}^{p})\text{.\qquad\emph{(}}d_{r}:E_{p,q}%
^{r}\rightarrow E_{p-r,q+r-1}^{r}\text{\emph{)}}%
\]

\item The following differentials are isomorphisms:%
\[
d_{n+1}:E_{n+1,0}^{n+1}\rightarrow E_{0,n}^{n+1}\,\text{.}%
\]

\item If $H_{p}:S^{p}\rightarrow S^{p+1}$ is a contracting homotopy for
$S^{\bullet}$, then%
\begin{equation}
(d_{n+1})^{-1}=H_{n}\delta_{1}H_{n-1}\cdots\delta_{n-1}H_{1}\delta_{n}%
H_{0}=H_{n}%
{\textstyle\prod\nolimits_{i=1,\ldots,n}}
(\delta_{i}H_{n-i})\text{.}\label{la24}%
\end{equation}

\end{enumerate}
\end{lemma}

This result can be applied to the bicomplex of Eq. \ref{la15}. The required
contracting homotopy can be constructed from a suitable family of commuting
idempotents in the cubically decomposed algebra as in Definition
\ref{marker_DefSysOfGoodIdempotents}:

\begin{definition}
[{\cite[Def. 14]{MR3207578}}]\label{marker_DefSysOfGoodIdempotents}Suppose $A$
is an $n$-fold unital cubically decomposed algebra. A \emph{system of good
idempotents} are pairwise commuting elements $P_{i}^{+}\in A$ (with
$i=1,\ldots,n$) such that the following conditions are met:

\begin{itemize}
\item $P_{i}^{+2}=P_{i}^{+}$.

\item $P_{i}^{+}A\subseteq I_{i}^{+}$.

\item $P_{i}^{-}A\subseteq I_{i}^{-}\qquad$(and we define $P_{i}%
^{-}:=\mathbf{1}_{A}-P_{i}^{+}$).
\end{itemize}
\end{definition}

The elements $P_{i}^{-}$ then are pairwise commuting idempotents as well. We
can use the contracting homotopy developed in an earlier paper:

\begin{lemma}
[{\cite[Lemma 16]{MR3207578}}]Let $A$ be unital and $\{P_{i}^{+}\}$ a system
of good idempotents. An explicit contracting homotopy $H:N^{i}\rightarrow
N^{i+1}$ for the complex $N^{\bullet}$ of Eq. \ref{la16} is given by%
\begin{align}
&  (Hf)_{s_{1}\ldots s_{n}}=\left(  -1\right)  ^{\deg(s_{1}\ldots s_{n}%
)}\left(  -1\right)  ^{s_{1}+\cdots+s_{b}}P_{1}^{s_{1}}\cdots P_{b}^{s_{b}%
}\label{la28}\\
&
{\textstyle\sum\limits_{\gamma_{1}\ldots\gamma_{b+1}\in\{\pm\}}}
\left(  -1\right)  ^{\gamma_{1}+\cdots+\gamma_{b}}P_{b+1}^{-\gamma_{b+1}%
}f_{\gamma_{1}\ldots\gamma_{b+1}s_{b+2}\ldots s_{n}}\nonumber\\
\text{(for }N^{i}  &  \rightarrow N^{i+1}\text{ with }i\geq1\text{)}%
\nonumber\\
&  (Hf)_{s_{1}\ldots s_{n}}=(-1)^{s_{1}+\cdots+s_{n}}P_{1}^{s_{1}}\cdots
P_{n}^{s_{n}}f\text{\quad(for }N^{0}\rightarrow N^{1}\text{)}\label{la26}%
\end{align}
where $b$ is the largest index such that $s_{1},\ldots,s_{b}\in\{\pm\}$ or
$b=0$ if none.
\end{lemma}

By tensoring $(-)\otimes A^{\otimes(r-1)}$ this induces a contracting homotopy
for the rows in the bicomplex of Eq. \ref{la15}. The evaluation of the formula
in Eq. \ref{la24} corresponds to following a zig-zag in the bicomplex which
can be depicted graphically as:%
\begin{equation}%
\begin{tabular}
[c]{ccccccc|c}
&  &  &  &  &  & $0$ & \\
&  &  &  &  &  & $\mid$ & \\
&  &  &  & $\theta_{1,n}$ & $\overset{H}{\longleftarrow}$ & $\theta_{0,n}$ &
$n$\\
&  &  &  & $\vdots$ &  &  & $\vdots$\\
&  & $\theta_{n,1}$ & $\overset{H}{\longleftarrow}$ & $\theta_{n-1,1}$ &  &  &
$1$\\
&  & $\downarrow$ &  &  &  &  & \\
$\theta_{n+1,0}$ & $\overset{H}{\longleftarrow}$ & $\theta_{n,0}$ &  &  &  &
& $0$\\\hline
$n+1$ &  & $n$ &  & $n-1$ & $\cdots$ & $0$ &
\end{tabular}
\label{la25}%
\end{equation}
If $\theta_{0,n}=f_{0}\otimes\cdots\otimes f_{n}$ represents an element in
$E_{0,n}^{n+1}$ arising from the first part of the definition of $\phi_{HH}$
(cf. Eq. \ref{la14})%
\[
HH_{n}(A)\overset{\sim}{\longrightarrow}H_{n}(A,N^{0})\overset{\text{edge}%
}{\longrightarrow}E_{0,n}^{n+1}\ni\theta_{0,n}\text{,}%
\]
we can compute $d^{-1}:E_{0,n}^{n+1}\overset{\sim}{\longrightarrow}%
E_{n+1,0}^{n+1}$ by Eq. \ref{la24}. We claim:

\begin{lemma}
Let $A$ be unital and $\{P_{i}^{+}\}$ a system of good idempotents. Starting
with $\theta_{0,n}=f_{0}\otimes\cdots\otimes f_{n}$, we get for $s_{1}%
,\ldots,s_{n-p}\in\{+,-\}$ the formula%
\begin{align*}
&  \theta_{p+1,n-p\mid s_{1}\ldots s_{n-p}0\ldots0}=(-1)^{n+(n-1)+\ldots
+(n-p+1)}\\
&  \qquad\cdots(-1)^{2+3+\cdots+(p+1)}\left(  -1\right)  ^{s_{1}%
+\cdots+s_{n-p}}P_{1}^{s_{1}}\cdots P_{n-p}^{s_{n-p}}\\
&  \qquad\cdots\left(
{\textstyle\sum\limits_{\gamma_{n-p+1}\in\{\pm\}}}
\left(  -1\right)  ^{\gamma_{n-p+1}}P_{n-p+1}^{-\gamma_{n-p+1}}f_{n-p+1}%
P_{n-p+1}^{\gamma_{n-p+1}}\right)  \cdot(\cdots)\cdot\\
&  \qquad\cdots\left(
{\textstyle\sum\limits_{\gamma_{n}\in\{\pm\}}}
(-1)^{\gamma_{n}}P_{n}^{-\gamma_{n}}f_{n}P_{n}^{\gamma_{n}}\right)
f_{0}\otimes f_{1}\otimes\cdots\otimes f_{n-p}%
\end{align*}
\newline for the terms in Fig. \ref{la25}.
\end{lemma}

This is the Hochschild counterpart of \cite[Prop. 24]{MR3207578}. The proof
will be very similar to the one given for the Lie homology counterpart in
\cite{MR3207578}, but actually quite a bit \textit{less} involved.

\begin{proof}
We prove this by induction on $p$, starting from $p=0$. In this case, the
claim reads%
\[
\theta_{1,n\mid s_{1}\ldots s_{n}}=\left(  -1\right)  ^{s_{1}+\cdots+s_{n}%
}P_{1}^{s_{1}}\cdots P_{n}^{s_{n}}f_{0}\otimes f_{1}\otimes\cdots\otimes
f_{n}\text{,}%
\]
which is clearly true in view of Eq. \ref{la26}. Next, assume the claim is
known for a given $p$ and we want to treat the case $p+1$, i.e. we need to
evaluate a Hochschild differential $b$ and pick a preimage as in the step%
\[%
\begin{tabular}
[c]{ccc}
&  & $\theta_{p+1,n-p}$\\
&  & $\downarrow b$\\
$\theta_{p+2,n-p-1}$ & $\overset{H}{\longleftarrow}$ & $\theta_{p+1,n-p-1} $%
\end{tabular}
\]
of Fig. \ref{la25}. According to our induction hypothesis, we get
$\theta_{p+1,n-p\mid s_{1}\ldots s_{n-p}0\ldots0}=Mf_{0}\otimes f_{1}%
\otimes\cdots\otimes f_{n-p}$ with the auxiliary expression%
\begin{align*}
M  &  =(-1)^{n+(n-1)+\ldots+(n-p+1)}(-1)^{2+3+\cdots+(p+1)}\left(  -1\right)
^{s_{1}+\cdots+s_{n-p}}P_{1}^{s_{1}}\cdots P_{n-p}^{s_{n-p}}\\
&  \qquad\cdots\left(
{\textstyle\sum\limits_{\gamma_{n-p+1}\in\{\pm\}}}
\left(  -1\right)  ^{\gamma_{n-p+1}}P_{n-p+1}^{-\gamma_{n-p+1}}f_{n-p+1}%
P_{n-p+1}^{\gamma_{n-p+1}}\right)  \cdot(\cdots)\cdot\\
&  \qquad\cdots\left(
{\textstyle\sum\limits_{\gamma_{n}\in\{\pm\}}}
(-1)^{\gamma_{n}}P_{n}^{-\gamma_{n}}f_{n}P_{n}^{\gamma_{n}}\right)  \text{.}%
\end{align*}
The Hochschild differential $b$ naturally decomposes into three parts (cf. Eq.
\ref{la27})%
\begin{align*}
\theta_{p+1,n-p-1}^{(A)}  &  =Mf_{0}f_{1}\otimes f_{2}\otimes\cdots\otimes
f_{n-p}\text{,}\\
\theta_{p+1,n-p-1}^{(B)}  &  =%
{\textstyle\sum\nolimits_{j=1}^{n-p-1}}
\left(  -1\right)  ^{j}Mf_{0}\otimes f_{1}\otimes\cdots\otimes f_{j}%
f_{j+1}\otimes\cdots\otimes f_{n-p}\text{,}\\
\theta_{p+1,n-p-1}^{(C)}  &  =(-1)^{n-p}f_{n-p}Mf_{0}\otimes f_{1}%
\otimes\cdots\otimes f_{n-p-1}%
\end{align*}
(here we have suppressed the subscript $(-)_{s_{1}\ldots s_{n-p}0\ldots0}$ for
the sake of readability). Next, we need to evaluate $\theta_{p+2,n-p-1}%
^{(-)}:=H\theta_{p+1,n-p-1}^{(-)}$ for the cases $A,B,C$. Let us consider case
$C$: In this case, we just use Eq. \ref{la28} and plugging in $M$,%
\begin{align*}
\theta_{p+2,n-p-1\mid s_{1}\ldots s_{n-p-1}0\ldots0}^{(C)}  &  =(-1)^{n-p}%
\left(  -1\right)  ^{\deg(s_{1}\ldots s_{n-p-1}0\ldots0)}\\
&  \left(  -1\right)  ^{s_{1}+\cdots+s_{n-p-1}}P_{1}^{s_{1}}\cdots
P_{n-p-1}^{s_{n-p-1}}\\
&
{\textstyle\sum\limits_{\gamma_{1}\ldots\gamma_{n-p}\in\{\pm\}}}
\left(  -1\right)  ^{\gamma_{1}+\cdots+\gamma_{n-p-1}}P_{n-p}^{-\gamma_{n-p}%
}f_{n-p}\\
&  (-1)^{n+(n-1)+\ldots+(n-p+1)}(-1)^{2+3+\cdots+(p+1)}\left(  -1\right)
^{\gamma_{1}+\cdots+\gamma_{n-p}}P_{1}^{\gamma_{1}}\cdots P_{n-p}%
^{\gamma_{n-p}}\\
&  \qquad\cdots\left(
{\textstyle\sum\limits_{\gamma_{n-p+1}\in\{\pm\}}}
\left(  -1\right)  ^{\gamma_{n-p+1}}P_{n-p+1}^{-\gamma_{n-p+1}}f_{n-p+1}%
P_{n-p+1}^{\gamma_{n-p+1}}\right)  \cdot(\cdots)\cdot\\
&  \cdots\left(
{\textstyle\sum\limits_{\gamma_{n}\in\{\pm\}}}
(-1)^{\gamma_{n}}P_{n}^{-\gamma_{n}}f_{n}P_{n}^{\gamma_{n}}\right)
f_{0}\otimes f_{1}\otimes\cdots\otimes f_{n-p-1}%
\end{align*}
This fairly complicated expression unwinds into something much simpler by
several observations: (1) There is a large cancellation in the sign terms
$(-1)^{(\ldots)}$, (2) we have $\deg(s_{1}\ldots s_{n-p-1}0\ldots0)=p+2$, (3)
the pairwise commutativity of the idempotents allows us to reorder terms so
that we obtain the expression $%
{\textstyle\sum\nolimits_{\gamma_{1}\ldots\gamma_{n-p-1}\in\{\pm\}}}
P_{1}^{\gamma_{1}}\cdots P_{n-p-1}^{\gamma_{n-p-1}}$, but this is just the
identity operator by using the fact $P_{i}^{+}+P_{i}^{-}=\mathbf{1}$. Finally,
we arrive at%
\begin{align*}
\theta_{p+2,n-p-1\mid s_{1}\ldots s_{n-p-1}0\ldots0}^{(C)}  &
=(-1)^{n+(n-1)+\ldots+(n-p+1)+(n-p)}(-1)^{2+3+\cdots+(p+2)}\\
&  \left(  -1\right)  ^{s_{1}+\cdots+s_{n-p-1}}P_{1}^{s_{1}}\cdots
P_{n-p-1}^{s_{n-p-1}}\\
&  \left(
{\textstyle\sum\limits_{\gamma_{n-p}\in\{\pm\}}}
\left(  -1\right)  ^{\gamma_{n-p}}P_{n-p}^{-\gamma_{n-p}}f_{n-p}%
P_{n-p}^{\gamma_{n-p}}\right)  \cdots\\
&  \cdots\left(
{\textstyle\sum\limits_{\gamma_{n}\in\{\pm\}}}
(-1)^{\gamma_{n}}P_{n}^{-\gamma_{n}}f_{n}P_{n}^{\gamma_{n}}\right)
f_{0}\otimes f_{1}\otimes\cdots\otimes f_{n-p-1}\text{.}%
\end{align*}
In a similar fashion we can deal with the cases $A,B$, however in both these
cases we obtain a term $P_{i}^{\gamma_{i}}P_{i}^{-\gamma_{i}}=P_{i}%
^{\gamma_{i}}(\mathbf{1}-P_{i}^{\gamma_{i}})=0$, so that these terms vanish.
We leave the details to the reader (a similar cancellation occurs in the proof
of \cite[Prop. 24]{MR3207578}, the cancellation is explained by the very
beautiful identity\footnote{pointed out to me by the anonymous referee of
\cite{MR3207578}} $H^{2}=0$, which holds for this particular contracting
homotopy). Hence, $\theta_{p+2,n-p-1}=\theta_{p+2,n-p-1}^{(C)}$, giving the claim.
\end{proof}

\begin{theorem}
\label{thm_FormulaForHochschildResidueSymbol}Let $(A,(I_{i}^{\pm}),\tau)$ be a
unital cubically decomposed algebra over $k$ and $\{P_{i}^{+}\}$ a system of
good idempotents. Then the explicit formula%
\begin{align*}
\phi_{HH}(f_{0}\otimes\cdots\otimes f_{n})  &  =(-1)^{n}\tau\left(
{\textstyle\sum\limits_{\gamma_{1}\in\{\pm\}}}
\left(  -1\right)  ^{\gamma_{1}}P_{1}^{-\gamma_{1}}f_{1}P_{1}^{\gamma_{1}%
}\right)  \cdots\\
&  \qquad\qquad\cdots\left(
{\textstyle\sum\limits_{\gamma_{n}\in\{\pm\}}}
(-1)^{\gamma_{n}}P_{n}^{-\gamma_{n}}f_{n}P_{n}^{\gamma_{n}}\right)  f_{0}%
\end{align*}
holds.
\end{theorem}

\begin{proof}
Use the lemma with $p=n$ and compose with the trace $\tau$ as in the
definition of $\phi_{HH}$ in Eq. \ref{la14}.
\end{proof}

\begin{corollary}
\label{marker_CorCompat1}Let $(A,(I_{i}^{\pm}),\tau)$ be a unital cubically
decomposed algebra over $k$, and let $\mathfrak{g}:=A_{Lie}$ be the
associated\ Lie algebra. Then the diagram%
\[
\xymatrix{ H_{n}(\mathfrak{g},\mathfrak{g}) \ar[r]^{\varepsilon } \ar[d]_{I^{\prime }} & HH_{n}(A) \ar[d]^{\phi_{HH}} \\ H_{n+1}(\mathfrak{g},k) \ar[r]_-{\phi_{Beil}} & k }
\]
commutes up to sign. Here $\varepsilon$ refers to the comparison map from Eq.
\ref{lb1}. The composition $\phi_{HH}\circ\varepsilon$ is given by the
commutator formula%
\begin{align}
f_{0}\otimes f_{1}\wedge\cdots\wedge f_{n} &  \mapsto(-1)^{n}\tau%
{\textstyle\sum\limits_{\sigma\in\mathfrak{S}_{n}}}
\operatorname*{sgn}(\sigma)%
{\textstyle\sum\limits_{\gamma_{1}\ldots\gamma_{n}\in\{\pm\}}}
\left(  -1\right)  ^{\gamma_{1}+\cdots+\gamma_{n}}\label{l_commutatorformula}%
\\
&  (P_{1}^{-\gamma_{1}}\operatorname*{ad}(f_{\sigma^{-1}(1)})P_{1}^{\gamma
_{1}})\cdots(P_{n}^{-\gamma_{n}}\operatorname*{ad}(f_{\sigma^{-1}(n)}%
)P_{n}^{\gamma_{n}})f_{0}\text{.}\nonumber
\end{align}
If $n=1$ and $[f_{0},f_{1}]=0$, then this specializes to%
\begin{equation}
f_{0}\otimes f_{1}\mapsto\tau\lbrack P_{1}^{+}f_{0},f_{1}]\text{.}\label{lam1}%
\end{equation}

\end{corollary}

The last equation links these formulae with the classical one-dimensional case
as found in Eq. \ref{la8}.

\begin{proof}
Let $\{P_{i}^{+}\}$ be any system of good idempotents. A direct computation of
$\phi_{HH}\circ\varepsilon$ yields the explicit formula%
\begin{align*}
f_{0}\otimes f_{1}\wedge\cdots\wedge f_{n} &  \mapsto(-1)^{n}\tau%
{\textstyle\sum\limits_{\sigma\in\mathfrak{S}_{n}}}
\operatorname*{sgn}(\sigma)%
{\textstyle\sum\limits_{\gamma_{1}\ldots\gamma_{n}\in\{\pm\}}}
\left(  -1\right)  ^{\gamma_{1}+\cdots+\gamma_{n}}\\
&  (P_{1}^{-\gamma_{1}}f_{\sigma^{-1}(1)}P_{1}^{\gamma_{1}})\cdots
(P_{n}^{-\gamma_{n}}f_{\sigma^{-1}(n)}P_{n}^{\gamma_{n}})f_{0}\text{,}%
\end{align*}
which agrees (up to sign) with the morphism $\left.  ^{\otimes}%
\operatorname*{res}\nolimits_{\ast}\right.  $ described in \cite[Thm. 25, and
following discussion]{MR3207578}.\ The commutativity then follows from
\cite[Lemma 23]{MR3207578}: Extended on the right with the trace, this reads%
\[%
\xymatrix{ H_{n}(\mathfrak{g},\mathfrak{g}) \ar[d]_{I^{\prime}} \ar
[r] \ar@/^2pc/[rrr]^{ {\left. ^{\otimes}\mathrm{res}\right.} } & {\left
. ^{\otimes}E^{n+1}_{0,n+1}\right.} \ar[d] & {\left. ^{\otimes}E^{n+1}%
_{n+1,1}\right.} \ar[d] \ar[l]^{d_{n+1}}_{\cong}
\ar[r] & k \ar[d]^{\cong} \\ H_{n+1}(\mathfrak{g},k) \ar[r] \ar@
/_2pc/[rrr]_{ {\phi}_{Beil} } & {\left. ^{\wedge}E^{n+1}_{0,n+1}\right.}
& {\left. ^{\wedge}E^{n+1}_{n+1,1}\right.} \ar[l]_{d_{n+1}}^{\cong}
\ar[r] & k. }%
\]
in the notation of the reference. The formula $P^{-\gamma}\operatorname*{ad}%
(f)P^{\gamma}g=P^{-\gamma}[f,P^{\gamma}g]=P^{-\gamma}fP^{\gamma}g-P^{-\gamma
}P^{\gamma}gf=P^{-\gamma}fP^{\gamma}g$ (since $P^{-\gamma}P^{\gamma}=0$)
implies Eq. \ref{l_commutatorformula}. For $n=1$, this specializes to%
\begin{align*}
f_{0}\otimes f_{1} &  \mapsto-\tau%
{\textstyle\sum\limits_{\gamma\in\{\pm\}}}
(-1)^{\gamma}P_{1}^{-\gamma}[P_{1}^{\gamma}f_{0},f_{1}]\\
&  =\tau(-P_{1}^{+}([f_{0},f_{1}]-[P_{1}^{+}f_{0},f_{1}])+P_{1}^{-}[P_{1}%
^{+}f_{0},f_{1}])
\end{align*}
and if $[f_{0},f_{1}]=0$ (as would be the case if $f_{0},f_{1}$ are functions
on a variety) this simplifies to Eq. \ref{lam1} by using $P_{1}^{+}+P_{1}%
^{-}=\mathbf{1}$.
\end{proof}

\begin{proposition}
\label{prop_coordinatecomp}Let $k$ be a field and $k^{\prime}/k$ a finite
field extension. For the equicharacteristic $n$-local field%
\[
K:=k^{\prime}((t_{1}))\cdots((t_{n}))\text{,}%
\]
consider the $\phi_{HH}$ associated to its standard cubically decomposed
algebra $E_{K}$ (\cite{MR3317764}, \cite{MR3536437}, \cite{bgwTateModule}, or
see the proof for an explanation).

\begin{enumerate}
\item Then for all $\beta\in k^{\prime}$, we have
\[
\phi_{HH}^{E_{j}}(\beta\cdot t_{1}^{c_{0,1}}\ldots t_{n}^{c_{0,n}}%
\otimes\cdots\otimes t_{1}^{c_{n,1}}\ldots t_{n}^{c_{n,n}})=\operatorname*{Tr}%
\nolimits_{k_{j}/k}(\beta)%
{\textstyle\prod\nolimits_{i=1}^{n}}
c_{i,i}%
\]
whenever $\forall i:\sum_{p=0}^{n}c_{p,i}=0$ and zero otherwise.

\item Precomposed with the HKR isomorphism (cf. Eq. \ref{lb5}), this yields%
\begin{align*}
\Omega_{K/k}^{n}\longrightarrow HH_{n}(K) &  \longrightarrow k\\
\beta\cdot f_{0}\mathrm{d}f_{1}\wedge\cdots\wedge\mathrm{d}f_{n} &
\longmapsto\operatorname*{Tr}\nolimits_{k^{\prime}/k}(\beta)\det%
\begin{pmatrix}
c_{1,1} & \cdots & c_{n,1}\\
\vdots & \ddots & \vdots\\
c_{1,n} & \cdots & c_{n,n}%
\end{pmatrix}
\end{align*}
for $f_{p}=t_{1}^{c_{p,1}}\cdots t_{n}^{c_{p,n}}$ $(0\leq p\leq n)$ whenever
$\forall i:\sum_{p=0}^{n}c_{p,i}=0$, and zero otherwise.

\item For $f\in K$ given by $f=\sum f_{\alpha_{1}\ldots\alpha_{n}}%
t_{1}^{\alpha_{1}}\cdots t_{n}^{\alpha_{n}}$ (with coefficients $f_{\alpha
_{1}\ldots\alpha_{n}}\in k^{\prime}$), we have%
\begin{align*}
\Omega_{K/k}^{n}\longrightarrow HH_{n}(K) &  \longrightarrow k\\
f\mathrm{d}t_{1}\wedge\cdots\wedge\mathrm{d}t_{n} &  \longmapsto
\operatorname*{Tr}\nolimits_{k^{\prime}/k}(f_{-1,\ldots,-1})\text{.}%
\end{align*}

\end{enumerate}
\end{proposition}

\begin{proof}
\textbf{(1)} Yekutieli gives a construction of the cubically decomposed
algebra $E_{K}$ \cite{MR3317764}. Alternatively, write the underlying vector
space of the $n$-local field as%
\[
k^{\prime}((t_{1}))((t_{2}))\ldots((t_{n}))=\underset{i_{n}}{\underrightarrow
{\operatorname*{colim}}}\underset{j_{n}}{\underleftarrow{\lim}}\cdots
\underset{i_{1}}{\underrightarrow{\operatorname*{colim}}}\underset{j_{1}%
}{\underleftarrow{\lim}}\,\frac{1}{t_{1}^{i_{1}}\cdots t_{n}^{i_{n}}}%
k[t_{1},\ldots,t_{n}]/(t_{1}^{j_{1}},\ldots,t_{n}^{j_{n}})\text{.}%
\]
Following \cite[Example 10]{bgwTateModule}, this defines an $n$-Tate object in
finite-dimensional $k$-vector spaces and the main results of
\cite{bgwTateModule} imply that its endomorphism algebra in the category of
$n$-Tate objects carries a cubically decomposed structure, call it $E_{K}$.
The equivalence of both approaches was shown in \cite[Theorem 3.8]{MR3536437}.
Moreover, loc. cit. shows that viewing this $n$-Tate object as a $k^{\prime}
$-vector space yields a faithful functor, i.e. any such endomorphism can be
thought of as a $k^{\prime}$-linear map. For $f\in I_{tr}$, the trace is
evaluated as follows: First, pick $i_{n}$ small enough such that the image
lies in%
\[
L_{1}:=\underset{j_{n}}{\underleftarrow{\lim}}\cdots\underset{i_{1}%
}{\underrightarrow{\operatorname*{colim}}}\underset{j_{1}}{\underleftarrow
{\lim}}\,\frac{1}{t_{1}^{i_{1}}\cdots t_{n}^{i_{n}}}k[t_{1},\ldots
,t_{n}]/(t_{1}^{j_{1}},\ldots,t_{n}^{j_{n}})\text{,}%
\]
and then $j_{n}$ great enough such that $f$ sends%
\[
L_{1}^{\prime}:=\underset{i_{n-1}}{\underrightarrow{\operatorname*{colim}}%
}\underset{j_{n-1}}{\underleftarrow{\lim}}\cdots\underset{i_{1}}%
{\underrightarrow{\operatorname*{colim}}}\underset{j_{1}}{\underleftarrow
{\lim}}\,\frac{1}{t_{1}^{i_{1}}\cdots t_{n}^{i_{n}}}k[t_{1},\ldots
,t_{n}]/(t_{1}^{j_{1}},\ldots,t_{n}^{j_{n}})
\]
to zero. Such values for $i_{n}$ and $j_{n}$ exist since $f$ lies (in
particular) in $I_{1}^{0}$. Using axiom \textbf{T2} of Tate's trace, Prop.
\ref{BT_PropTateTraceConstruction}, the trace of $f$ agrees with the trace of
$f\mid_{L_{1}/L_{1}^{\prime}}$. We see that this step has reduced computing
the trace of an endomorphism of $n$ limit-colimit pairs (over
finite-dimensional vector spaces), to computing the trace for just $(n-1)$
limit-colimit pairs. Repeating this reduction, it suffices to evaluate the
trace on a finite-dimensional vector space, where by axiom \textbf{T1} it
agrees with the ordinary trace. Moreover, as these reduction steps just
restrict to ranges of exponents of the $t_{1}^{?}\cdots t_{n}^{?}$ appearing,
some finite system of such monomials forms a $k^{\prime}$-basis.\newline%
\textbf{(2)} Henceforth, in order to distinguish clearly between $t_{i}$ as a
multiplication operator $x\mapsto t_{i}\cdot x$, or as $k^{\prime}$-vector
space basis elements, we write the latter in bold letters $\mathbf{t}_{i}$.
Define idempotents $P_{i}^{+}$ by%
\[
P_{i}^{+}\sum f_{\lambda_{1}\ldots\lambda_{n}}\mathbf{t}_{1}^{\lambda_{1}%
}\cdots\mathbf{t}_{n}^{\lambda_{n}}:=\sum\delta_{\lambda_{i}\geq0}%
f_{\lambda_{1}\ldots\lambda_{n}}\mathbf{t}_{1}^{\lambda_{1}}\cdots
\mathbf{t}_{n}^{\lambda_{n}}\text{.}%
\]
Define $P_{i}^{-}=\mathbf{1}-P_{i}^{+}$. We know that $\operatorname*{im}%
P_{i}^{+}\subseteq I_{i}^{+}$ is a lattice and $P_{i}^{-}(\operatorname*{im}%
P_{i}^{+})=0$, so we have a system of good idempotents in the sense of
Definition \ref{marker_DefSysOfGoodIdempotents}. Thus, by Thm.
\ref{thm_FormulaForHochschildResidueSymbol} we have%
\begin{equation}
\phi_{HH}(f_{0}\otimes\cdots\otimes f_{n})=(-1)^{n}\operatorname*{tr}%
\nolimits_{k}M=(-1)^{n}\operatorname*{Tr}\nolimits_{k^{\prime}/k}%
(\operatorname*{tr}\nolimits_{k^{\prime}}M)\label{lbb101}%
\end{equation}
for the operator $M$ defined by
\[
M:=%
{\textstyle\sum\limits_{\gamma_{1}\ldots\gamma_{n}\in\{\pm\}}}
\left(  -1\right)  ^{\gamma_{1}+\cdots+\gamma_{n}}P_{1}^{-\gamma_{1}}%
f_{1}P_{1}^{\gamma_{1}}\cdots P_{n}^{-\gamma_{n}}f_{n}P_{n}^{\gamma_{n}}%
f_{0}\text{.}%
\]
The remaining computation is essentially the same as in the proof of
\cite[Thm. 26]{MR3207578}, so we just sketch the key steps: Letting
$f_{m}:=\mathbf{t}_{1}^{c_{m,1}}\cdots\mathbf{t}_{n}^{c_{m,n}}$ for
$c_{m,i}\in\mathbf{Z}$ (and $1\leq m\leq n$; $1\leq i\leq n$) and
$f_{0}:=\beta\mathbf{t}_{1}^{c_{0,1}}\cdots\mathbf{t}_{n}^{c_{0,n}}$, one
easily computes%
\[
P_{m}^{-}f_{m}P_{m}^{+}\mathbf{t}_{1}^{\lambda_{1}}\cdots\mathbf{t}%
_{n}^{\lambda_{n}}=\delta_{0\leq\lambda_{m}<-c_{m,m}}\mathbf{t}_{1}%
^{\lambda_{1}+c_{m,1}}\cdots\mathbf{t}_{n}^{\lambda_{n}+c_{m,n}}%
\qquad\text{for }1\leq m\leq n\text{.}%
\]
This formula closely mimics the one-dimensional computation in Lemma
\ref{Lemma_LocalTateFormulaDimOne}. With this we can explicitly compute the
action of $M$ on a monomial. We get%
\begin{align*}
M\mathbf{t}_{1}^{\lambda_{1}}\cdots\mathbf{t}_{n}^{\lambda_{n}}=\beta%
{\textstyle\prod\limits_{i=1}^{n}}
(\delta_{0\leq\lambda_{i}+c_{0,i}+\sum_{p=i+1}^{n}c_{p,i}<-c_{i,i}} &
-\delta_{-c_{i,i}\leq\lambda_{i}+c_{0,i}+\sum_{p=i+1}^{n}c_{p,i}<0})\\
&  \mathbf{t}_{1}^{\lambda_{1}+c_{0,1}+\sum_{p=1}^{n}c_{p,1}}\cdots
\mathbf{t}_{n}^{\lambda_{n}+c_{0,n}+\sum_{p=1}^{n}c_{p,n}}\text{.}%
\end{align*}
It is immediately clear that this operator can have a non-zero trace only if
$\forall i:\sum_{p=0}^{n}c_{p,i}=0$ holds, because otherwise it is visibly
nilpotent and we can invoke axiom \textbf{T3} of Tate's trace. This proves the
vanishing part of the claim. Now assume this condition holds and simplify the
formula for $M$ accordingly. A simple eigenvalue count reveals%
\begin{equation}
\operatorname*{tr}\nolimits_{k^{\prime}}M=%
{\textstyle\prod\nolimits_{i=1}^{n}}
(-c_{i,i})=(-1)^{n}\beta%
{\textstyle\prod\nolimits_{i=1}^{n}}
c_{i,i}\text{,}\label{lbb102}%
\end{equation}
where $M$ is still viewed as an endomorphism of a $k^{\prime}$-vector space.
See the proof of \cite[Thm. 26]{MR3207578} for the full details. Finally,
$\operatorname*{tr}\nolimits_{k}M=\operatorname*{Tr}\nolimits_{k^{\prime}%
/k}(\operatorname*{tr}\nolimits_{k^{\prime}}M)$ computes the value in
question; the signs $(-1)^{n}$ from Eq. \ref{lbb101} and Eq. \ref{lbb102}
cancel each other out. \textbf{(3)} Plugging in the antisymmetrizer coming
from the HKR isomorphism, we get%
\[
=\beta%
{\textstyle\sum\nolimits_{\pi\in\mathfrak{S}_{n}}}
\operatorname*{sgn}(\pi)%
{\textstyle\prod\nolimits_{i=1}^{n}}
c_{\pi(i),i}\text{,}%
\]
which up to the factor $\beta$ is exactly the Leibniz formula for the
determinant. \textbf{(4)} In this special case, let $f_{0}:=f$ and
$f_{m}=t_{m}$ for $1\leq m\leq n$ and proceed basically as before. We compute%
\[
P_{m}^{-}f_{m}P_{m}^{+}\mathbf{t}_{1}^{\lambda_{1}}\cdots\mathbf{t}%
_{n}^{\lambda_{n}}=\delta_{\lambda_{m}=0}\mathbf{t}_{1}^{\lambda_{1}}%
\cdots\mathbf{t}_{m}^{\lambda_{m}+1}\cdots\mathbf{t}_{n}^{\lambda_{n}}%
\qquad\text{for }1\leq m\leq n
\]
on monomials. As before, we use this to compute the trace of the operator%
\[
M:=%
{\textstyle\sum\limits_{\gamma_{1}\ldots\gamma_{n}\in\{\pm\}}}
\left(  -1\right)  ^{\gamma_{1}+\cdots+\gamma_{n}}P_{1}^{-\gamma_{1}}%
f_{1}P_{1}^{\gamma_{1}}\cdots P_{n}^{-\gamma_{n}}f_{n}P_{n}^{\gamma_{n}}%
f_{0}\text{,}%
\]
which this time unwinds as%
\[
M\mathbf{t}_{1}^{\lambda_{1}}\cdots\mathbf{t}_{n}^{\lambda_{n}}=\sum
_{c_{0,1},\ldots,c_{0,n}}f_{c_{0,1}\ldots c_{0,n}}%
{\textstyle\prod\limits_{i=1}^{n}}
(-\delta_{\lambda_{i}+c_{0,i}=-1})\mathbf{t}_{1}^{\lambda_{1}+c_{0,1}+1}%
\cdots\mathbf{t}_{n}^{\lambda_{n}+c_{0,n}+1}%
\]
and we see that only the summand with $c_{0,i}=-1-\lambda_{i}$ remains, giving%
\[
=(-1)^{n}f_{(-1-\lambda_{1})\ldots(-1-\lambda_{n})}\mathbf{t}_{1}^{\lambda
_{1}+(-1-\lambda_{1})+1}\cdots\mathbf{t}_{n}^{\lambda_{n}+(-1-\lambda_{n}%
)+1}\text{.}%
\]
This is nilpotent unless $\lambda_{1}=\cdots=\lambda_{n}=0$ and in this case
has trace $\operatorname*{Tr}\nolimits_{k^{\prime}/k}(f_{-1,\ldots,-1})$,
proving the claim.
\end{proof}

Next, we shall relate various $\phi_{HH}$ for changing cubically decomposed
algebras. To clarify the distinction, let us agree to write $\phi_{HH}%
^{A}:HH_{n}(A)\rightarrow k$ instead of $\phi_{HH}$ plain.

\begin{theorem}
[Local formula]\label{thm_localformula}Suppose $X/k$ is a reduced finite type
scheme of pure dimension $n$ over a perfect field $k$. Suppose $\triangle
=(\eta_{0}>\cdots>\eta_{n})\in S\left(  X\right)  _{n}$ with
$\operatorname*{codim}\nolimits_{X}\overline{\{\eta_{i}\}}=i$. Then there is a
canonical finite decomposition%
\[
A(\triangle,\mathcal{O}_{X})\cong%
{\textstyle\prod}
K_{j}%
\]
with each $K_{j}$ an $n$-local field.

\begin{enumerate}
\item Each $E_{j}:=\{f\in E_{\triangle}\mid fK_{j}\subseteq K_{j}$,
$fK_{r}=0\ $for $r\neq j\}$ with ideals $J_{i}^{\pm}:=I_{i}^{\pm}\cap E_{j}$
is a cubically decomposed algebra over $k$ and for $f\in HH_{n}(\mathcal{O}%
_{\eta_{0}})$ we have%
\begin{equation}
\phi_{HH}^{E_{\triangle}}(f)=%
{\textstyle\sum\nolimits_{j}}
\phi_{HH}^{E_{j}}(f)\text{.}\label{lb3}%
\end{equation}

\item There exists (non-canonically) an isomorphism $K_{j}\simeq k_{j}%
((t_{1}))\cdots((t_{n}))$ with $k_{j}/k$ a finite field extension such that
for all $\beta\in k_{j}$%
\[
\phi_{HH}^{E_{j}}(\beta\cdot t_{1}^{c_{0,1}}\ldots t_{n}^{c_{0,n}}%
\otimes\cdots\otimes t_{1}^{c_{n,1}}\ldots t_{n}^{c_{n,n}})=\operatorname*{Tr}%
\nolimits_{k_{j}/k}(\beta)%
{\textstyle\prod\nolimits_{i=1}^{n}}
c_{i,i}%
\]
whenever $\forall i:\sum_{p=0}^{n}c_{p,i}=0$ and zero otherwise.

\item Precomposed with the HKR isomorphism (cf. Eq. \ref{lb5}), this yields%
\begin{align*}
\Omega_{K_{j}/k}^{n}\longrightarrow HH_{n}(K_{j}) &  \longrightarrow k\\
\beta\cdot f_{0}\mathrm{d}f_{1}\wedge\cdots\wedge\mathrm{d}f_{n} &
\longmapsto\operatorname*{Tr}\nolimits_{k_{j}/k}(\beta)\det%
\begin{pmatrix}
c_{1,1} & \cdots & c_{n,1}\\
\vdots & \ddots & \vdots\\
c_{1,n} & \cdots & c_{n,n}%
\end{pmatrix}
\end{align*}
for $f_{p}=t_{1}^{c_{p,1}}\cdots t_{n}^{c_{p,n}}$ $(0\leq p\leq n)$ whenever
$\forall i:\sum_{p=0}^{n}c_{p,i}=0$ and zero otherwise.

\item For $f\in K_{j}$ given by $f=\sum f_{\alpha_{1}\ldots\alpha_{n}}%
t_{1}^{\alpha_{1}}\cdots t_{n}^{\alpha_{n}}$ (with coefficients $f_{\alpha
_{1}\ldots\alpha_{n}}\in k_{j}$) we have%
\begin{align*}
\Omega_{K_{j}/k}^{n}\longrightarrow HH_{n}(K_{j}) &  \longrightarrow k\\
f\mathrm{d}t_{1}\wedge\cdots\wedge\mathrm{d}t_{n} &  \longmapsto
\operatorname*{Tr}\nolimits_{k_{j}/k}(f_{-1,\ldots,-1})\text{.}%
\end{align*}

\end{enumerate}
\end{theorem}

A word of warning: In (2), while there always exists an isomorphism
$K_{j}\simeq k_{j}((t_{1}))\cdots((t_{n}))$ such that the above claims hold,
it is by no means true that any isomorphism between these fields has these
properties. This very subtle behaviour is discussed extensively in
\cite{MR3317764} and \cite{MR3536437}.

\begin{proof}
Almost all of the first claim follows directly from Prop.
\ref{TATE_StructureOfLocalAdelesProp}. \textbf{(1)} Observe that the $E_{j}$
are associative algebras. Define $J_{i}^{\pm}:=I_{i}^{\pm}\cap E_{j}$ with
$I_{i}^{\pm}$ the ideals of the cubically decomposed algebra structure of
$E_{\triangle}$, cf. Prop.
\ref{Prop_LocalEAtFullFlagIsCubicallyDecomposedAlgebra}. It is clear that the
$J_{i}^{\pm}$ are two-sided ideals in $E_{j}$ and we claim that $J_{i}%
^{+}+J_{i}^{-}=E_{j}$. To see this, let $x\in E_{j}$ be given. We have
$A(\triangle,\mathcal{O}_{X})=\prod K_{j}$, so let $e_{j}$ be the idempotent
of the $j$-th factor. It is easy to check that $e_{j}\in E_{\triangle}$. Write
$x=x^{+}+x^{-}$ with $x^{\pm}\in I_{i}^{\pm}$. Now $e_{j}xe_{j}=e_{j}%
x^{+}e_{j}+e_{j}x^{-}e_{j}$. Since the $I_{i}^{\pm}$ are ideals, $e_{j}x^{\pm
}e_{j}\in I_{i}^{\pm}$, but also $e_{j}x^{\pm}e_{j}\in E_{j}$. It follows that
$e_{j}x^{\pm}e_{j}\in I_{i}^{\pm}\cap E_{j}=J_{i}^{\pm}$. On the other hand,
$e_{j}xe_{j}=x$. The converse inclusion is obvious, so we have $J_{i}%
^{+}+J_{i}^{-}=E_{\triangle}\cap E_{j}=E_{j}$. Since $J_{tr}=\bigcap
_{i=1,\ldots,n}\bigcap_{s\{\pm\}}J_{i}^{s}\subseteq I_{tr}$ we can use the
trace map of $E_{\triangle}$. This proves that $(E_{j},\{J_{i}^{\pm
}\},\operatorname*{tr}\nolimits_{I_{tr}})$ is a cubically decomposed algebra.
In particular, the maps $\phi_{HH}^{E_{j}}$ exist. The embedding
$\mathcal{O}_{\eta_{0}}\hookrightarrow A(\triangle,\mathcal{O}_{X})\cong%
{\textstyle\prod}
K_{j}$ is diagonal, i.e. $f\mapsto(f,\ldots,f)$. As a result, the associated
multiplication operator in $E_{\triangle}$ is diagonal in the $K_{j}$,
therefore Eq. \ref{lb3} holds. \textbf{(2)} For the evaluation of $\phi
_{HH}^{E_{j}}$, we want to pick an isomorphism of fields%
\[
\rho:K_{j}\overset{\sim}{\longrightarrow}k_{j}((t_{1}))\cdots((t_{n}))
\]
with the following properties: (1) $\rho$ is an isomorphism of fields, (2)
$\rho$ is an isomorphism of $k$-vector spaces, and (3) $\rho$ induces an
isomorphism of cubically decomposed algebras $(E_{j},\{J_{i}^{\pm
}\},\operatorname*{tr}\nolimits_{I_{tr}})$ to the cubically decomposed algebra
structure of $k_{j}((t_{1}))\cdots((t_{n}))$, as in Prop.
\ref{prop_coordinatecomp}:%
\[
(E_{j},\{J_{i}^{\pm}\},\operatorname*{tr}\nolimits_{I_{tr}})\longrightarrow
E_{k_{j}((t_{1}))\cdots((t_{n}))}\text{,}\qquad f\longmapsto\rho\circ
f\circ\rho^{-1}\text{.}%
\]

The existence of such a $\rho$ follows from \cite[Theorem 0.2, (3)]%
{MR3536437}. Since the construction of $\phi_{HH}$ is intrinsic to the
cubically decomposed algebra structure, this isomorphism implies that we may
perform our computation on the level of $k_{j}((t_{1}))\cdots((t_{n}))$, so
the entire claim reduces to Prop. \ref{prop_coordinatecomp}.
\end{proof}

\section{\label{marker_SectAlternativeApproach}A new approach}

\subsection{Introduction}

We want to change our perspective. Let $(A,(I_{i}^{\pm}),\tau)$ be a cubically
decomposed algebra. So far we have always worked in the category of
$A$-bimodules and considered exact sequences of $A$-bimodules like%
\begin{equation}
0\longrightarrow I_{n}^{0}\overset{\operatorname*{diag}}{\longrightarrow}%
I_{n}^{+}\oplus I_{n}^{-}\overset{\operatorname*{diff}}{\longrightarrow
}A\longrightarrow0\label{lml_101}%
\end{equation}
or their higher-dimensional counterparts as in Eq. \ref{la16}. This approach
corresponds to viewing Hochschild homology as a functor%
\[
A\text{-bimodules}\rightarrow k\text{-vector spaces,}\qquad M\mapsto
H_{i}(A,M)\text{.}%
\]
However, Hochschild homology can also be regarded as a functor%
\[
\text{associative }k\text{-algebras}\rightarrow k\text{-vector spaces,}\qquad
A\mapsto HH_{i}(A)\text{.}%
\]
In this section we want to transform the mechanisms of
\S \ref{marker_BeilLocalConstruction}, \S \ref{section_HochschildUnitalIntro}
from the former to the latter perspective.

\subsection{Recollections\label{marker_sect_ReminderHochschildNonUnital}}

We shall need to work with \textit{non-unital} algebras, so let us briefly
recall the necessary material (see \cite{MR997314} for details). Hochschild
homology was defined and described in \S \ref{section_HochschildUnitalIntro}
for an arbitrary associative algebra $A$. We may read $A$ as a bimodule over
itself and if $A$ is unital we write $HH_{i}(A):=H_{i}(A,A)$. If $A$ is not
unital, all definitions still make sense and we write $HH_{i}^{naiv}%
(A):=H_{i}(A,A)$ for these groups, following \cite[\S 1.4.3]{MR1217970}.
However, this is not a good definition in general, so usually one proceeds
differently: There is a unitalization $A^{+}$ along with a canonical map
$k\hookrightarrow A^{+}$ of unital associative algebras, and one
defines\footnote{This is not the definition given in our main reference
\cite{MR997314}; here $HH_{i}(A)$ is the homology of $\mathcal{K}$, cf. p.
$598$, l. $5$ in \textit{loc. cit.}, defined in terms of the bar complex. The
equivalence of definitions follows from the paragraph before Thm. 3.1 in
\textit{loc. cit.}}%
\begin{equation}
HH_{i}(A):=\operatorname*{coker}\left(  HH_{i}(k)\rightarrow HH_{i}%
(A^{+})\right)  \text{,}\label{lml_2}%
\end{equation}
see \cite[\S 1.4]{MR1217970} for details; this parallels a similar
construction in algebraic $K$-theory. If $A$ happens to be unital, this agrees
with the previous definition as in \S \ref{section_HochschildUnitalIntro},
i.e. it agrees with $HH_{i}^{naiv}$. In general, there is only the obvious
morphism $\kappa:HH_{i}^{naiv}(A)\rightarrow HH_{i}(A)$ (sending a pure tensor
to itself in $A^{+}$) which need neither be injective nor surjective.

If $0\rightarrow M^{\prime}\rightarrow M\rightarrow M^{\prime\prime
}\rightarrow0$ is a short exact sequence of $A$-bimodules, the sequence
$0\rightarrow C_{\bullet}(A,M^{\prime})\rightarrow C_{\bullet}(A,M)\rightarrow
C_{\bullet}(A,M^{\prime\prime})\rightarrow0$ is obviously an exact sequence of
complexes, so there is a long exact sequence in Hochschild homology%
\begin{equation}
\cdots\rightarrow H_{i}(A,M^{\prime})\rightarrow H_{i}(A,M)\rightarrow
H_{i}(A,M^{\prime\prime})\overset{\partial}{\rightarrow}H_{i-1}(A,M^{\prime
})\rightarrow\cdots\text{.}\label{lml_29}%
\end{equation}
We denote the connecting homomorphism by $\partial$. If $I$ is a two-sided
ideal in $A$, this yields the sequence%
\begin{equation}
\cdots\rightarrow H_{i}(A,I)\rightarrow H_{i}(A,A)\overset{\mu}{\rightarrow
}H_{i}(A,A/I)\overset{\partial}{\rightarrow}H_{i-1}(A,I)\rightarrow
H_{i-1}(A,A)\rightarrow\cdots\label{lml_1}%
\end{equation}
Moreover, if $M$ is an $A/I$-bimodule, it is also an $A$-bimodule via
$A\twoheadrightarrow A/I$. Then there is an obvious change-of-algebra map
$\nu:C_{i}(A,M)\rightarrow C_{i}(A/I,M)$. Clearly $A/I$ is an $A/I$-bimodule
and thus there are canonical maps%
\[
j:C_{i}(A,A)\overset{\mu}{\rightarrow}C_{i}(A,A/I)\overset{\nu}{\rightarrow
}C_{i}(A/I,A/I)\text{,}%
\]
where $\mu$ is the morphism inducing the respective arrow in Eq. \ref{lml_1}.
One also defines the \emph{relative Hochschild homology} complex
$\mathcal{K}_{\bullet}(A\rightarrow A/I)$, the precise definition is somewhat
involved, see \cite[beginning of \S 3, where instead of $\mathcal{C}$ one uses
the Hochschild version $\mathcal{K}$, defined on the same page $598$ in line
$5$]{MR997314}. We write $HH_{i}(A\left.  \mathsf{rel}\right.  I):=H_{i}%
\mathcal{K}_{\bullet}(A\rightarrow A/I)$ for its homology (\textit{Beware:}%
\ The notation $HH_{i}(A,I)$ is customary. However, it is easily confused with
$H_{i}(A,I)$, which also plays a role here, so we have opted for the present
clearer distinction). We may regard $I$ as an associative algebra itself, but
unless $A=I$ it will not be unital.

\begin{proposition}
[{\cite{MR934604}, \cite[Thm. 3.1]{MR997314}}]%
\label{ml_Lemma_WodzickiLocalUnitsImplyExcision}Suppose $A$ is an associative
algebra and $I$ a two-sided ideal. Suppose both have at least one-sided local
units. Then the canonical morphisms%
\begin{equation}
HH_{i}^{naiv}(I)\overset{\kappa}{\longrightarrow}HH_{i}(I)\overset{\Diamond
}{\longrightarrow}HH_{i}(A\left.  \mathsf{rel}\right.  I)\label{lml_3}%
\end{equation}
are both isomorphisms. There is a quasi-isomorphism%
\begin{equation}
\mathcal{K}_{\bullet}(A\rightarrow A/I)\simeq_{\operatorname*{qis}}%
\ker(C_{\bullet}(A,A)\overset{j}{\rightarrow}C_{\bullet}(A/I,A/I))\text{.}%
\label{lml_30}%
\end{equation}

\end{proposition}

It is noteworthy that only the right-most term in Eq. \ref{lml_3} actually
depends on $A$.

\begin{proof}
For the proof, combine \cite[Thm. 3.1 and Cor. 4.5]{MR997314} for the first
claim: The existence of local units implies $H$-unitality. For the second
claim, $A$ is $H$-unital, so the bar complex in the definiton of $\mathcal{K}
$ in \textit{loc. cit. }p. $598$ in line $5$ is zero up to quasi-isomorphism.
Applying this to the definition of $\mathcal{K}_{\bullet}(A\rightarrow A/I)$
in \S 3 in \textit{loc. cit.} gives the second claim. For an alternative
presentation, combine the treatment \cite[\S 1.4.9]{MR1217970} with the
generality of \cite[E.1.4.6]{MR1217970}. The $H$-unitality of $A/I$ follows
from \cite[Cor. 3.4]{MR997314}.
\end{proof}

Basically by construction, we get a long exact sequence in homology%
\begin{equation}
\cdots\rightarrow HH_{i}(A\left.  \mathsf{rel}\right.  I)\rightarrow
HH_{i}(A)\rightarrow HH_{i}(A/I)\overset{\delta}{\rightarrow}HH_{i-1}(A\left.
\mathsf{rel}\right.  I)\rightarrow\cdots\text{.}\label{lml_31}%
\end{equation}
Although different, it is not unrelated to the sequence in Eq. \ref{lml_1}:

\begin{lemma}
\label{ml_Lemma_HHAlgebraAndModuleDifferentialCompatible}Suppose $A$ is an
associative algebra and $I$ a two-sided ideal with at least one-sided local
units. Then the diagram%
\begin{equation}
\xymatrix{ \cdots \ar[r] & H_{i}(A,I) \ar[r] \ar[d] & H_{i}(A,A) \ar[r] \ar[d]^{\kappa } & H_{i}(A,A/I) \ar[r]^-{\partial } \ar[d]^{\lambda } & \cdots \\ \cdots \ar[r] & HH_{i}(A {\left. \mathsf{rel}\right.} I) \ar[r] & HH_{i}(A) \ar[r] & HH_{i}(A/I) \ar[r]^-{\delta } & \cdots }\label{lml_20}%
\end{equation}
is commutative.
\end{lemma}

\begin{proof}
Trivial if $A$ is unital. In general: We construct this on the level of
complexes $C_{\bullet}(-,-)$. The middle downward arrow maps pure tensors to
themselves, $A\rightarrow A^{+}$ in $HH_{i}(A^{+})$ and then to the cokernel
as given by Eq. \ref{lml_2}. Similarly, the right-hand side downward arrow is
induced by%
\[
a_{0}\otimes a_{1}\otimes\cdots\otimes a_{i}\mapsto a_{0}\otimes
\overline{a_{1}}\otimes\cdots\otimes\overline{a_{i}}\text{,}%
\]
where $a_{0}\in A/I$, $a_{1},\ldots,a_{n}\in A$ and $\overline{\cdot
}:A\twoheadrightarrow A/I$ is the quotient map, again sent to $(A/I)^{+}$ and
then to the respective cokernel. For the left-hand side we can wlog. use the
presentation on the right-hand side of Eq. \ref{lml_30} for $HH_{i}(A\left.
\mathsf{rel}\right.  I)$. The downward arrow is then given by the analogous
formula, but $a_{0}\in I$ and so $\overline{a_{0}}=0$ in $A/I$, so that it is
clear that the image lies in the kernel of $j:C_{i}(A,A)\rightarrow
C_{i}(A/I,A/I)$.
\end{proof}

\subsection{The construction}

Let $(A^{n},(I_{i}^{\pm}),\tau)$ be an $n$-fold cubically decomposed algebra
over $k$. Define%
\begin{equation}
A^{n-1}:=I_{n}^{0}\qquad J_{i}^{\pm}:=I_{i}^{\pm}\cap A^{n-1}\qquad\text{(for
}i=0,\ldots,n-1\text{).}\label{lml_41}%
\end{equation}
Then $(A^{n-1},(J_{i}^{\pm}),\tau)$ is an $(n-1)$-fold cubically decomposed
algebra over $k$.

\begin{definition}
\label{def_deeplocalunits}We say that an $n$-fold cubically decomposed algebra
$(A,(I_{i}^{\pm}),\tau)$ \emph{has local units on all levels} (or is
`\emph{good}') if $A^{s}$ has local left units (or local right units) for
$s=1,\ldots,n$.
\end{definition}

Evaluating Eq. \ref{lml_41} inductively, we find $A^{s}=(I_{s+1}^{0}\cap
\cdots\cap I_{n}^{0})\cap A$. Define%
\begin{equation}
\Lambda:A^{n}\longrightarrow A^{n}/A^{n-1}\text{,}\qquad x\longmapsto
x^{+}\text{,}\label{lToeplitz}%
\end{equation}
where $x=x^{+}+x^{-}$ is any decomposition with $x^{\pm}\in I_{n}^{\pm}$
(always exists and gives well-defined map). This map does \textit{not} equal
the natural quotient map! Using the relative Hochschild homology sequence, Eq.
\ref{lml_31}, coming from the exact sequence of associative algebras%
\begin{equation}
0\longrightarrow A^{n-1}\longrightarrow A^{n}\overset{\operatorname*{quot}%
}{\longrightarrow}A^{n}/A^{n-1}\longrightarrow0\text{,}\label{lml_28}%
\end{equation}
the connecting homomorphism induces a map $\delta$ and we employ it to define
a map%
\begin{equation}
d:HH_{i+1}(A^{n})\overset{\Lambda}{\longrightarrow}HH_{i+1}(A^{n}%
/A^{n-1})\overset{\delta}{\longrightarrow}HH_{i}(A^{n-1})\text{.}%
\label{lml_40}%
\end{equation}
We can repeat this construction and obtain a morphism:

\begin{definition}
\label{def_thePhiC}Suppose $(A,(I_{i}^{\pm}),\tau)$ is an $n$-fold cubically
decomposed algebra over $k$ which has local units on all levels. Then there is
a canonical map%
\[
\phi_{C}:HH_{n}(A)\longrightarrow HH_{0}(I_{tr})\longrightarrow k\text{,}%
\qquad\alpha\mapsto\tau\underset{n\text{ times}}{\underbrace{d\circ\cdots\circ
d}}\alpha\text{.}%
\]
Analogously, for cyclic homology $\phi_{C}:HC_{n}(A)\rightarrow k$ (see lemma
below why we call this $\phi_{C}$ as well).
\end{definition}

\begin{lemma}
\label{marker_CorCompat2}The map $\phi_{C}$ factors over $HH_{n}(A)\overset
{I}{\longrightarrow}HC_{n}(A)\longrightarrow k$.
\end{lemma}

\begin{proof}
Let $d^{\prime}$ be the analogue of the map in Eq. \ref{lml_40} with cyclic
homology. Both $\Lambda$ and the connecting map are compatible with $I$ so
that%
\[
\xymatrix{ HH_{n}(A) \ar[r]^{d\circ \cdots \circ d} \ar[d]_{I} & HH_{0}(I_{tr}) \ar[d]^{I} \\ HC_{n}(A) \ar[r]_{d^{\prime }\circ \cdots \circ d^{\prime }} & HC_{0}(I_{tr}) }
\]
commutes, but the right-hand side downward arrow is an isomorphism, giving the claim.
\end{proof}

\begin{theorem}
\label{thm_phicequalsphihh}Suppose $(A,(I_{i}^{\pm}),\tau)$ is a unital
$n$-fold cubically decomposed algebra over $k$ which has local units on all
levels. Then $\phi_{C}:HH_{n}(A)\rightarrow k$ agrees up to sign with
$\phi_{HH}$, namely%
\[
\phi_{C}=\left(  -1\right)  ^{\frac{n(n-1)}{2}}\phi_{HH}\text{.}%
\]

\end{theorem}

\begin{proof}
\textbf{(1)} We proceed by induction. Firstly, we construct a commutative
diagram and a map $\Psi$:%
\begin{equation}
\xymatrix{ H_{s}(A,A^{s}) \ar[d]_{\Lambda } \ar@/_4pc/[dd]_{\Psi } & H_{s}(A^{s},A^{s}) \ar[l] \ar[r]^{\kappa } \ar[d]_{\Lambda } & HH_{s}(A^{s}) \ar[d]^{\Lambda } \ar@/^4pc/[dd]^{d } \\ H_{s}(A,\frac{A^{s}}{A^{s-1}}) \ar[d]_{\partial } & H_{s}(A^{s},\frac{A^{s}}{A^{s-1}}) \ar[l] \ar[r]^{\lambda } \ar[d]_{\partial } & HH_{s}(\frac{A^{s}}{A^{s-1}}) \ar[d]^{\delta } \\ H_{s-1}(A,A^{s-1}) & H_{s-1}(A^{s},A^{s-1}) \ar[l] \ar[r] & HH_{s-1}(A^{s-1}) }\label{lml_33}%
\end{equation}
The leftward arrows are the change-of-algebra maps along $A^{s}\hookrightarrow
A$. The commutativity of the upper left square is immediate, the one on the
right agrees with the rightmost square in Lemma
\ref{ml_Lemma_HHAlgebraAndModuleDifferentialCompatible}. The downward arrows
in the middle row come from the connecting homomorphism in the long exact
sequences (as in Eq. \ref{lml_29} and Eq. \ref{lml_31}, combined with Wodzicki
excision) arising from Eq. \ref{lml_28}. The commutativity of the lower
squares then follows from Lemma
\ref{ml_Lemma_HHAlgebraAndModuleDifferentialCompatible}. \textbf{(2)} Next, we
patch the outer columns of the diagram as in Eq. \ref{lml_33} for
$s=n,n-1,\ldots,1$ under each other, giving%
\[
\xymatrix{ H_{n}(A,A^{n}) \ar[d]_{\Psi } & H_{n}(A^{n},A^{n}) \ar[l]_{\cong } \ar[r]^{\cong }_{\kappa} & HH_{n}(A^{n}) \ar[d]^{d } \\ H_{n-1}(A,A^{n-1}) \ar[d]|{\vdots }_{\Psi } & & HH_{n-1}(A^{n-1}) \ar[d]|{\vdots }^{d } \\ H_{0}(A,A^{0}) & H_{0}(A^{1},A^{0}) \ar[l] \ar[r] & HH_{0}(A^{0}) }
\]
The middle column of the previous diagram does not fit to be glued into this
pattern, so we omit it, except for the top and bottom row. The morphisms in
the top row are isomorphisms since $A$ (unlike the $A^{s}$ for $s<n$) is
unital. We evaluate the terms in the lowest row and compose with the trace
$\tau$, giving the diagram%
\[
\xymatrix{ \frac{A^{0}}{[A,A^{0}]} \ar[d] & \frac{A^{0}}{[A^{1},A^{0}]} \ar[l] \ar[d] \ar[r] & HH_{0}(A^{0}) \ar[d] \\ k & k \ar[l]^{\cong } \ar[r]_{\cong} & k. }
\]
Since the trace $\tau$ factors through $[A,A^{0}]$ (note that $A^{0}=I_{tr}$),
it is clear that the arrows in the bottom row must be isomorphisms. Thus,
$\phi_{C}=\tau d^{\circ n}=\tau\Psi^{\circ n}$. Note that this comparison only
works because in the top and bottom row all terms are isomorphic, whereas on
the intermediate rows it is not clear whether there should exist arrows from
the left to the right column (or reversely). It remains to compute $\tau
\Psi^{\circ n}$:\newline\textbf{(3)} Consider the diagram with exact rows%
\begin{equation}
\xymatrix{ I^{0}_{s} \ar[r]^-{\text{diag}} \ar[d]_{=} & I^{+}_{s} \oplus I^{-}_{s} \ar[r] \ar[d]_{pr_{I^{+}_{s}}} & A^{s} \ar[d]^{\text{(1)}} \ar@/^4pc/[dd]^{\Lambda } \\ A^{s-1} \ar[r] \ar[d]_{=} & I^{+}_{s} \ar[r] \ar[d]_{\text{incl}} & A^{s}/{I^{-}_{s}} \ar[d]^{\text{(2)}}\\ A^{s-1} \ar[r] & A^{s} \ar[r] & A^{s}/{A^{s-1}} }\label{lml_34}%
\end{equation}
(here for readability we have omitted intersecting all the ideals with $A^{s}
$; everything is understood to be subobjects of $A^{s}$). The map
$pr_{I_{s}^{+}}$ is the projection $\left(  x^{+},x^{-}\right)  \mapsto x^{+}%
$. Pick the arrows $(1)$ and $(2)$ such that the diagram becomes commutative.
We find both are given by $x\mapsto x^{+}$ where $x=x^{+}+x^{-}$ with $x^{\pm
}\in I_{s}^{\pm}$ is any decomposition of $x$. Moreover, the composition on
the right is indeed $\Lambda$. Taking the long exact sequences in Hochschild
homology of the top and bottom row yields%
\[
\xymatrix{ H_{s}(A,I^{+}_{s} \oplus I^{-}_{s}) \ar[r] \ar[d]_{\text{incl} \circ pr_{I_{s}^{+}}} & H_{s}(A,A^{s}) \ar[r]^{\partial } \ar[d]_{\Lambda} & H_{s-1}(A,I_{s}^{0}) \ar[d]^{\cong } \\ H_{s}(A,A^{s}) \ar[r] & H_{s}(A,A^{s}/{A^{s-1}}) \ar[r]_{\partial} & H_{s-1}(A,A^{s-1}) \\ }
\]
Now by the commutativity of the above diagram $\Psi=\partial\circ\Lambda
:H_{s}(A,A^{s})\rightarrow H_{s-1}(A,A^{s-1})$ (as on the left in Eq.
\ref{lml_33}) can be computed just by unwinding the connecting map in the top
row. It stems from the bimodule exact sequence in the top row of Eq.
\ref{lml_34}: Evaluating this is an easy chase of the snake map, compare with
the proof of Lemma \ref{Lemma_LocalTateFormulaDimOne}: Pick some system of
good idempotents. We need to pick a lift of $a_{0}\otimes a_{1}\otimes
\cdots\otimes a_{s}\in C_{s}(A,A^{s})$ to $C_{s}(A,I_{s}^{+}\oplus I_{s}^{-}%
)$. We may take $f_{\gamma}:=(-1)^{\gamma}P_{s}^{\gamma}a_{0}\otimes
a_{1}\otimes\cdots\otimes a_{s}$ for $\gamma\in\{\pm\}$ respectively. We need
to apply the differential $b$, resulting in%
\begin{align*}
&  bf_{\gamma}=\left(  -1\right)  ^{\gamma}(P_{s}^{\gamma}a_{0}a_{1}\otimes
a_{2}\otimes\cdots\otimes a_{s}+\sum_{j=1}^{s-1}\left(  -1\right)  ^{j}%
P_{s}^{\gamma}a_{0}\otimes\cdots\otimes a_{j}a_{j+1}\otimes\cdots\otimes
a_{s}\\
&  \qquad+(-1)^{s}a_{s}P_{s}^{\gamma}a_{0}\otimes a_{1}\otimes\cdots\otimes
a_{s-1})\in C_{s-1}(A,I_{1}^{\gamma})
\end{align*}
Next, we need to determine the preimage in $C_{s-1}(A,I_{s}^{0})=C_{s-1}%
(A,A^{s-1})$, which is%
\begin{align*}
\Psi(a_{0}\otimes\cdots\otimes a_{s}) &  =%
{\textstyle\sum\nolimits_{\gamma\in\{\pm\}}}
(-1)^{\gamma}P_{s}^{-\gamma}(bf_{\gamma})\\
&  =(-1)^{s}\left(
{\textstyle\sum\nolimits_{\gamma\in\{\pm\}}}
(-1)^{\gamma}P_{s}^{-\gamma}a_{s}P_{s}^{\gamma}\right)  a_{0}\otimes
a_{1}\otimes\cdots\otimes a_{s-1}\text{.}%
\end{align*}
Hence, by applying this formula inductively, we get%
\[
\tau\Psi^{\circ n}(a_{0}\otimes\cdots\otimes a_{n})=(-1)^{1+2+\cdots+n}%
\tau\prod_{s=1\ldots n}\left(
{\textstyle\sum\nolimits_{\gamma\in\{\pm\}}}
(-1)^{\gamma}P_{s}^{-\gamma}a_{s}P_{s}^{\gamma}\right)  a_{0}\text{.}%
\]
This expression clearly coincides (up to sign) with the one of Theorem
\ref{thm_FormulaForHochschildResidueSymbol} so that the previously proven
identity $\phi_{C}=\tau d^{\circ n}=\tau\Psi^{\circ n}$ implies the claim.
\end{proof}

\begin{corollary}
[Comparison diagram]\label{marker_AllComparisonDiagramCor}Under the
assumptions of the theorem and $\mathfrak{g}:=A_{Lie}$,

\begin{enumerate}
\item the diagram%
\[
\xymatrix{ H_{n}(\mathfrak{g},\mathfrak{g}) \ar[r]^-{\varepsilon } \ar[d]_{{I}^{\prime }} & HH_{n}(A) \ar[r]^-{{\phi }_{HH}} \ar[d]_{I} & k \ar[d]^{=} \\ H_{n+1}(\mathfrak{g},k) \ar[r]_-{{(-1)}^{n} \varepsilon} & HC_{n}(A) \ar[r]_-{{\phi }_{C}} & k }
\]
commutes, where for $f_{0},\ldots,f_{n}\in\mathfrak{g}$, the map $\varepsilon$
in the bottom row is given by%
\[
\varepsilon(f_{0}\wedge\cdots\wedge f_{n}):=\sum_{\pi\in\mathfrak{S}_{n}%
}\operatorname*{sgn}(\pi)\,f_{0}\otimes f_{\pi^{-1}(1)}\otimes\cdots\otimes
f_{\pi^{-1}(n)}\text{.}%
\]

\item The composed map $H_{n}(\mathfrak{g},\mathfrak{g})\rightarrow k$ agrees
with $H_{n}(\mathfrak{g},\mathfrak{g})\overset{I^{\prime}}{\longrightarrow
}H_{n}(\mathfrak{g},k)\overset{\phi_{Beil}}{\longrightarrow}k$.
\end{enumerate}
\end{corollary}

\begin{proof}
The left-hand side square commutes by direct inspection. Then combine Cor.
\ref{marker_CorCompat1} and Cor. \ref{marker_CorCompat2}.
\end{proof}

\section{Tate's abstract reciprocity revisited}

A prominent feature of Tate's article \cite{MR0227171} is his slick proof of
the residue theorem for curves. In this section, I want to propose a
formulation of such vanishing statements on the level of cubically decomposed
algebras. In particular, I want to interpret the \textquotedblleft abstract
residue formula\textquotedblright\ of \cite[Lemma 2.4]{MR1013132} in the
Hochschild picture.

\begin{theorem}
[Tautological Reciprocity Law]Suppose $(A,(I_{i}^{\pm}))$ is an $n$-fold
cubically decomposed algebra over $k$ with local units on all levels. Then%
\[
\phi_{C}(x)=0
\]
for any element $x$ in the Hochschild homology of any of the ideals $I_{i}%
^{+}$, $I_{i}^{-}$ for any $i$.
\end{theorem}

\begin{proof}
\textit{(Case A)} Suppose the ideal is $I:=I_{1}^{+}$. Since for $\Lambda$ we
may take \textit{any} decomposition $x=x^{+}+x^{-}$ with $x^{\pm}\in
I_{1}^{\pm}$, we may just as well take $x^{+}:=x$. But that means that
$\Lambda$ acts on $x$ just like the quotient map, and we get the dotted arrow
in%
\[%
\xymatrix{
& & HH_{m}(I) \ar[d]_{\Lambda} \ar@{.>}[dl] \\
\cdots\ar[r] & HH_{m}(A^{n}) \ar[r]^-{\operatorname{quot}} & HH_{m}%
(A^{n}/A^{n-1}) \ar[r]^-{\delta} & HH_{m-1}(A^{n-1}) \ar[r] & \cdots}%
\]
and the exactness of the bottom row implies $d(x)=0$. And therefore, $\phi
_{C}(x)=0$.\newline\textit{(Case B)} Suppose the ideal is $I:=I_{1}^{-}$.
Since for $\Lambda$ we may take \textit{any} decomposition $x=x^{+}+x^{-}$
with $x^{\pm}\in I_{1}^{\pm}$, we may just pick $x^{+}:=0$. Thus, $\phi
_{C}(x)=0$.\newline\textit{(Case C)} Suppose the ideal is $I:=I_{i}^{s}$ for
$i\geq2$ and $s\in\{+,-\}$. Then apply the first $i-1$ maps \textquotedblleft%
$d$\textquotedblright\ in Definition \ref{def_thePhiC}, and observe that its
value lies in $HH_{n-(i-1)}(A^{n-(i-1)}\cap I_{i}^{+})$, but by the inductive
nature of the definition this means that the value lies in the ideal
$I_{1}^{+}$ for the $(n-i+1)$-fold cubically decomposed algebra $A^{n-(i-1)}$,
and thus the above Cases A or B apply to this element. Again, we obtain zero.
\end{proof}

Note that this proof is so simple because of the inductive nature of
Definition \ref{def_thePhiC}. The next vanishing statement is a little more refined.

\begin{theorem}
[Cube Reciprocity Law]\label{thm_multreclaw}Let $(A,(I_{i}^{\pm}))$ be a
unital $n$-fold cubically decomposed algebra with local units on all levels.
Let $P^{\pm}\in A$ be idempotents such that%
\[
P^{+}+P^{-}=1\qquad\text{and}\qquad P^{\pm}A\in I_{1}^{\pm}\text{.}%
\]
If $R\subseteq A$ is a sub-algebra such that $P^{+}A$ (or $P^{-}A$) is a left
$R$-submodule of $A$, then%
\[
\phi_{C}(r)=0
\]
for all $r\in HH_{n}(R)$.
\end{theorem}

\begin{proof}
\textit{(Case A)} Suppose $P^{+}A$ is a left $R$-submodule. We define a
$k$-linear map of Hochschild groups $\psi:C_{i}(R)\rightarrow C_{i}(A)$,
$R^{\otimes i+1}\rightarrow A^{\otimes i+1}$ by%
\[
r_{0}\otimes\cdots\otimes r_{i}\longmapsto r_{0}P^{+}\otimes\cdots\otimes
r_{i}P^{+}\text{.}%
\]
We note that the map $r\mapsto rP^{+}$ would have no reason to be an algebra
homomorphism from $R$ to $A$, so we cannot just induce the above map from a
morphism of algebras. Instead, we need to check that the above describes a
morphism of complexes by hand. We compute%
\begin{align*}
b(\psi(r_{0}\otimes\cdots\otimes r_{i}))  & =\sum_{j=0}^{i-1}(-1)^{j}%
r_{0}P^{+}\otimes\cdots\otimes r_{j}P^{+}r_{j+1}P^{+}\otimes\cdots\otimes
r_{i}P^{+}\\
& +(-1)^{i}r_{i}P^{+}r_{0}P^{+}\otimes r_{1}P^{+}\otimes\cdots\otimes
r_{i-1}P^{+}\text{.}%
\end{align*}
Since the image of $P^{+}$ is a left $R$-module, $r_{j+1}P^{+}\in
\operatorname*{im}P^{+}$, and thus $P^{+}r_{j+1}P^{+}=r_{j+1}P^{+}$, and then
$r_{j}P^{+}r_{j+1}P^{+}=r_{j}r_{j+1}P^{+}$. Thus, we get%
\begin{align*}
b(\psi(r_{0}\otimes\cdots\otimes r_{i}))  & =\sum_{j=0}^{i-1}(-1)^{j}%
r_{0}P^{+}\otimes\cdots\otimes r_{j}r_{j+1}P^{+}\otimes\cdots\otimes
r_{i}P^{+}\\
& +(-1)^{i}r_{i}r_{0}P^{+}\otimes r_{1}P^{+}\otimes\cdots\otimes r_{i-1}%
P^{+}\\
& =\psi b(r_{0}\otimes\cdots\otimes r_{i})\text{.}%
\end{align*}
Thus, $\psi\circ b=b\circ\psi$ and we conclude that $\psi$ is a morphism of
complexes. Next, note that for any $a\in A$, we have $a=aP^{+}+aP^{-}$ with
$aP^{\pm}\in I_{1}^{\pm}$. It follows that our map $\psi$ is a lift of
$\Lambda$, i.e. the diagram%
\begin{equation}%
\xymatrix{
& & HH_{m}(R) \ar[d]_{\Lambda} \ar@{.>}[dl]_{\psi} \\
\cdots\ar[r] & HH_{m}(A^{n}) \ar[r]^-{\operatorname{quot}} & HH_{m}%
(A^{n}/A^{n-1}) \ar[r]^-{\delta} & HH_{m-1}(A^{n-1}) \ar[r] & \cdots}%
\label{lruis_3}%
\end{equation}
commutes. As in the previous proof, the exactness of the row implies that
$d(r)=0$.\newline\textit{(Case B)} Now assume $P^{-}A$ is a left $R$-submodule
of $A$ instead. We define $\psi$ as before, just replacing each $P^{+}$ by
$P^{-}$. Everything goes through, with the exception that $\psi$ now lifts
$x\mapsto x^{-}$ instead of $x\mapsto x^{+}$. However, since $P^{+}+P^{-}=1$,
we can replace Diagram \ref{lruis_3} by%
\[%
\xymatrix{
& & HH_{m}(R) \ar[d]_{\Lambda} \ar@{.>}[dl]_{\iota- \psi} \\
\cdots\ar[r] & HH_{m}(A^{n}) \ar[r]^-{\operatorname{quot}} & HH_{m}%
(A^{n}/A^{n-1}) \ar[r]^-{\delta} & HH_{m-1}(A^{n-1}) \ar[r] & \cdots}%
\]
where $\iota$ is the inclusion of algebras $\iota:R\hookrightarrow A$ (this is
an algebra homomorphism). Thus, again $\Lambda$ lifts and we obtain $d(r)=0$.
\end{proof}

\subsection{Applications of the cube reciprocity law}

\begin{example}
[Curves, Local Theory]\label{Example_Curves_LocalTheory}Let $k$ be a field and
$X/k$ an integral curve. Write $\eta$ for its generic point. Suppose $x\in X$
is a closed point. Then the ad\`{e}le%
\[
A(\eta>x)=\prod_{i}\widehat{\mathcal{K}}_{i}%
\]
is a finite product of $1$-local fields with residue fields finite over $k$.
The number of factors in the product agrees with the number of preimages of
the point $x$ in the normalization of the curve $X^{\prime}\rightarrow X$. If
$X/k$ is regular, there is always just one factor, as in
\S \ref{marker_ForCurves}. Example \ref{example_exceptionaladelefibering}
demonstrates the effect of a singular point. Following our formalism, we get
an abstract residue symbol%
\begin{equation}
\operatorname*{res}\nolimits_{\widehat{\mathcal{K}}_{i}}:\Omega_{\widehat
{\mathcal{K}}_{i}/k}^{1}\longrightarrow HH_{1}(A)\overset{\phi_{C}%
}{\longrightarrow}k\text{.}\label{lruis_2}%
\end{equation}
\newline Now write%
\begin{equation}
\widehat{\mathcal{K}}_{i}=\widehat{\mathcal{O}}_{i}\oplus B\text{,}%
\qquad\text{(as Tate vector spaces)}\label{lruis_1}%
\end{equation}
where $\widehat{\mathcal{O}}_{i}$ is the ring of integers of $\widehat
{\mathcal{K}}_{i}$ (this need not agree with $\widehat{\mathcal{O}}_{X,x}$ if
$x$ did not lie in the smooth locus; rather it would be a finite ring
extension; it always agrees with $\widehat{\mathcal{O}}_{X^{\prime},x^{\prime
}}$, where $x^{\prime}$ is the chosen preimage of $x$ in the normalization
$X^{\prime}$), and $B$ is any $k$-vector space complement. As $\widehat
{\mathcal{O}}_{i}$ is a lattice of the Tate vector space, let $P^{\pm}$ be the
idempotents underlying the direct sum decomposition of Equation \ref{lruis_1}.
Then $\widehat{\mathcal{O}}_{i}\hookrightarrow\widehat{\mathcal{K}}_{i}$ is a
sub-algebra such that $P^{+}A$ is a left-$\widehat{\mathcal{O}}_{i}$-module
(this is true because $P^{+}$ maps everything to $\widehat{\mathcal{O}}%
_{i}\subseteq\widehat{\mathcal{K}}_{i}$, and if we act on $\widehat
{\mathcal{O}}_{i}$ by multiplication with an element $f\in\widehat
{\mathcal{O}}_{i}$, this still lies in $\widehat{\mathcal{O}}_{i}$, and
therefore applying $P^{+}$ again acts as the identity). Hence, by the cube
reciprocity law $HH_{1}(\widehat{\mathcal{O}}_{i})\rightarrow HH_{1}%
(A)\overset{\phi_{C}}{\rightarrow}k$ is the zero map. As a result, we learn
that our residue map in Equation \ref{lruis_2} is trivial on $1$-forms without
poles and factors as $\Omega_{\widehat{\mathcal{K}}_{i}/k}^{1}/\Omega
_{\widehat{\mathcal{O}}_{i}/k}^{1}\rightarrow k$. Of course, this is one of
the most obvious properties the residue map should have. We see here that it
is encoded in\ Theorem \ref{thm_multreclaw}.
\end{example}

\begin{example}
[Curves, Global Theory]We continue the previous example. By Beilinson's
resolution, Theorem \ref{lX_BeilinsonResolutionThm}, we have the flasque
ad\`{e}le resolution of the sheaf $\Omega_{X/k}^{1}$, namely%
\begin{equation}
0\longrightarrow\Omega_{X/k}^{1}\longrightarrow\mathbf{A}_{\Omega^{1}}%
^{(0)}\oplus\mathbf{A}_{\Omega^{1}}^{(1)}\longrightarrow\mathbf{A}_{\Omega
^{1}}^{(01)}\longrightarrow0\text{.}\label{lmfs}%
\end{equation}
Here $\mathbf{A}_{\Omega^{1}}^{(1)}$ denotes the ad\`{e}les running through
all singleton flags $\triangle$ consisting only of closed points
$\{(x)\}_{x\in X_{0}}$, while $\mathbf{A}_{\Omega^{1}}^{(0)}$ denotes the
remaining summand, which agrees with the rational function field $k\left(
X\right)  $ of the curve. Finally, $\mathbf{A}_{\Omega^{1}}^{(01)}$ are the
ad\`{e}les of all length $2$ flags, i.e. those of the shape $(\eta>x)$ for
$\eta$ the generic point and $x$ running through the closed points. The
ad\`{e}les also carry the structure of a cubically decomposed algebra
\cite{bgwTateModule}. One way to see this is by using that they are a $1$-Tate
object, as explained in \cite{bgwTateModule}, and therefore the endomorphism
algebra in the category of Tate vector spaces has a natural structure of a
cubically decomposed algebra, see loc. cit. Feeding this into our abstract
machine, we get a residue symbol on the level of ad\`{e}les,%
\[
\operatorname*{res}\nolimits_{\mathbf{A}}:HH_{1}(\mathbf{A}_{\Omega^{1}%
}^{(01)})\longrightarrow HH_{1}(A)\overset{\phi_{C}}{\longrightarrow}k\text{.}%
\]
Due to the nature of the ad\`{e}les, there is a projection map of $1$-Tate
objects (and rings, simultaneously) $\mathbf{A}_{\Omega^{1}}^{(01)}%
\longrightarrow\widehat{\mathcal{K}}$, where $\widehat{\mathcal{K}}$ is a
local field factor as in the local theory, Example
\ref{Example_Curves_LocalTheory}. As a result, the residue on the ad\`{e}les
is just the sum of the local residues%
\begin{equation}
\operatorname*{res}\nolimits_{\mathbf{A}}((\alpha_{x})_{x})=\sum
\operatorname*{res}\nolimits_{\widehat{\mathcal{K}}}(\alpha_{x})\text{,}%
\label{lruis_4}%
\end{equation}
where $x$ runs through the set of closed points. Thus, we can reduce the
computation of residues to local fields (this is the analogue of \cite[Theorem
3]{MR0227171}). We get two reciprocity laws now: Firstly, $\mathbf{A}%
^{(1)}=\prod_{x\in X}\widehat{\mathcal{O}}_{x}$ is an $\mathbf{A}^{(1)}%
$-submodule of $\mathbf{A}^{(01)}$. We get a direct sum splitting
$\mathbf{A}^{(01)}=\mathbf{A}^{(1)}\oplus B$ and Theorem \ref{thm_multreclaw}
implies that residues of $1$-forms from $\mathbf{A}^{(1)}$ are zero. This is
no real news of course, since this already follows from the local study of
Example \ref{Example_Curves_LocalTheory}. However, we also get a direct sum
splitting $\mathbf{A}^{(01)}=\mathbf{A}^{(0)}\oplus B^{\prime}$, where
$\mathbf{A}^{(0)}=k\left(  X\right)  $ is just the rational function field and
this is a $k\left(  X\right)  $-submodule of $\mathbf{A}^{(01)}$. If $X/k$ is
proper (and only then!), the finite-dimensionality of the cohomology implies
that the assumptions of Theorem \ref{thm_multreclaw} are met: Concretely, we
could write $\mathbf{A}^{(01)}=k\left(  X\right)  \oplus L$ for a suitably
chosen lattice $L$ of the Tate vector space such that, on the level of
$k$-vector spaces, this splitting can be identified with%
\[
\mathbf{A}^{(01)}=k\left(  X\right)  \oplus\underset{\simeq L}{\underbrace
{\frac{\mathbf{A}^{(1)}}{H^{0}(X,\mathcal{O}_{X})}\oplus H^{1}(X,\mathcal{O}%
_{X})}}\text{.}%
\]
This is possible since Theorem \ref{lX_BeilinsonResolutionThm} (applied to
$\mathcal{O}_{X}$) implies that $H^{0}(X,\mathcal{O}_{X})=\mathbf{A}^{(0)}%
\cap\mathbf{A}^{(1)}$ and $H^{1}(X,\mathcal{O}_{X})$ is isomorphic to the
cokernel of $\mathbf{A}^{(0)}+\mathbf{A}^{(1)}$ inside $\mathbf{A}^{(01)}$.
Since both cohomology groups are finite-dimensional $k$-vector spaces, $L$ is
indeed a lattice. Thus, Theorem \ref{thm_multreclaw} tells us that global
rational $1$-forms have vanishing global residue $\operatorname*{res}%
\nolimits_{\mathbf{A}}$. By the global-local formula, Equation \ref{lruis_4},
we conclude the following famous fact: For any global rational $1$-form
$\omega\in\Omega_{X/k}^{1}\left(  X\right)  \otimes k\left(  X\right)  $, the
sum of residues is zero, i.e.
\[
\sum_{x\in X_{(0)}}\operatorname*{res}\nolimits_{\widehat{\mathcal{K}}_{x}%
}(\omega)=0\text{.}%
\]
This is the analogue of \cite[\S 3, Corollary]{MR0227171}, and of course
properness enters our argument in exactly the r\^{o}le as in his paper. If $f$
is a global rational function, $\mathrm{d}\log(f)=\mathrm{d}f/f$ is such a
rational $1$-form and we learn that the total sum of orders of zeros and poles
is zero (when being added up in $k$; so if $\operatorname*{char}(k)>0$ this
statement is not as strong as it could be).
\end{example}

\begin{example}
[Less standard fact]Suppose we are in the situation of Example
\ref{Example_Curves_LocalTheory}. Instead of Equation \ref{lruis_1}, we also
have a direct sum splitting $\widehat{\mathcal{K}}_{i}=\kappa\lbrack
t^{-1}]\oplus t\kappa\lbrack\lbrack t]]$, where we have chosen, for the sake
of exposition, an isomorphism $\widehat{\mathcal{K}}_{i}\simeq\kappa((t))$.
Note that $\kappa\lbrack t^{-1}]$ is also a sub-algebra such that the
Multiplicative Reciprocity Law applies. It tells us that $\operatorname*{res}%
(t^{-n}\mathrm{d}t^{-m})=0$ for all $n,m\geq0$. While one finds this fact
rarely articulated, it is of course also easy to show using the usual calculus
of differentials: $t^{-n}(\mathrm{d}t^{-m})=t^{-n}(-mt^{-m-1}\mathrm{d}%
t)=-mt^{-n-m-1}\mathrm{d}t$. For $n,m\geq0$ this visibly (from the usual
perspective) can only have non-zero residue if $n=m=0$, but then this
expression is zero thanks to the leading coefficient $m$.
\end{example}

So far, we have only used the Cube Reciprocity\ Law to establish statements in
dimension one. We shall address the higher-dimensional story in a sequel.

\section{The bigger picture}

In this paper, we have first tried to explain the construction of the residue
map in \cite{MR565095}. Loc. cit., Beilinson does this using Lie homology, and
specifically relative Lie homology. This word never appears in \cite{MR565095}%
, but we hope to have elucidated why and how this shows up in
\S \ref{section_Etiology}. The essence of the construction lies in%
\[
\phi_{Beil}:H_{n+1}(\mathfrak{g},k)\overset{\sim}{\longrightarrow}%
H_{n+1}(CE(\mathfrak{g}))\overset{\text{edge}}{\longrightarrow}E_{0,n+1}%
^{n+1}\underset{d^{-1}}{\overset{\sim}{\longrightarrow}}E_{n+1,1}%
^{n+1}\overset{\text{edge}}{\longrightarrow}H_{1}(^{\wedge}T_{\bullet}%
^{n+1})\overset{\tau}{\longrightarrow}k
\]
of \S \ref{marker_BeilLocalConstruction}. In the present paper, we have
explained how to remove the presence of any relative Lie homology groups by
(a) reformulating the theory in Hochschild homology, and (b) showing that the
above map can (essentially) also be realized by an iterated use of a
\textsl{modified} boundary map $d$,%
\begin{equation}
\phi_{C}:HH_{n}(A)\longrightarrow HH_{0}(I_{tr})\longrightarrow k\text{,}%
\qquad\alpha\mapsto\tau d\circ\cdots\circ d\,\alpha\text{.}\label{lvio5}%
\end{equation}
This is based on writing the cubically decomposed algebra as an iterated
extension, $A^{n-1}\rightarrow A^{n}\rightarrow A^{n}/A^{n-1}$.

As gets developed in joint work with M. Groechenig and J. Wolfson,
\cite{MR3536437}, one can conveniently package the definition of the
ad\`{e}les of a scheme as an object of the category $\mathcal{T}:=\left.
n\text{-}\mathsf{Tate}(\mathsf{Vect}_{f})\right.  $, and then the complicated
definition of the cubically decomposed algebra structure, Definition
\ref{def_HigherAdeleOperatorIdeals}, simplifies to the plain
$\operatorname*{End}\nolimits_{\mathcal{T}}$ in this category. Now, for any
exact sequence of exact categories $\mathcal{C}^{\prime}\hookrightarrow
\mathcal{C}\twoheadrightarrow\mathcal{C}^{\prime\prime}$, one has an induced
long exact sequence in the Hochschild homology of exact categories
\cite{MR1667558}. Joint work with J. Wolfson in the companion paper
\cite{hochschildconiveau} then shows that $\phi_{C}$ agrees with the iterated
use of the boundary map of this long exact sequence. Thus, unlike the $d$ in
line \ref{lvio5}, which is a slightly modified version of a boundary map (by
the modification we refer to the Toeplitz-like twist by $\Lambda$ in Eq.
\ref{lml_40}), the localization sequence boundary map gives the right map on
the nose. Combined with this paper, we thus can follow the entire journey from
Tate's original approach using commutators in \cite{MR0227171}, to Beilinson's
use of relative Lie homology \cite{MR565095}, to Hochschild homology of
non-unital algebras in the present paper, to the Hochschild homology of
categories in \cite{hochschildconiveau}. The latter paper has a new version of
a Hochschild--Kostant--Rosenberg theorem \textit{with supports}, which also
makes a connection to the local cohomology approach of Grothendieck in
\cite{MR0222093}.

\begin{acknowledgement}
The original ideas belong to John Tate and Alexander Be\u{\i}linson. This
paper is just a lengthy interpretation and elaboration of their beautiful
works \cite{MR565095}, \cite{MR0227171}, possibly adding mistakes, but perhaps
a way of expressing appreciation for their inventiveness.\newline I thank Ivan
Fesenko for explaining the problem to me; Michael Groechenig and Jesse Wolfson
for many interesting discussions on the ind-pro approach, which have led to
numerous joint papers; Amnon Yekutieli for his many insightful remarks,
suggestions and careful reading of an earlier version. This has been most
helpful.\medskip\newline Moreover, I thank the Essen Seminar for Algebraic
Geometry and Arithmetic for the excellent working conditions and friendly atmosphere.
\end{acknowledgement}

\bibliographystyle{amsalpha}
\bibliography{ollinewbib}

\def\cprime{$'$} \def\polhk#1{\setbox0=\hbox{#1}{\ooalign{\hidewidth
  \lower1.5ex\hbox{`}\hidewidth\crcr\unhbox0}}} \def\cprime{$'$}
  \def\cprime{$'$} \def\cprime{$'$} \def\cprime{$'$}
\providecommand{\bysame}{\leavevmode\hbox to3em{\hrulefill}\thinspace}
\providecommand{\MR}{\relax\ifhmode\unskip\space\fi MR }
\providecommand{\MRhref}[2]{%
  \href{http://www.ams.org/mathscinet-getitem?mr=#1}{#2}
}
\providecommand{\href}[2]{#2}
\begin{thebibliography}{BGW16b}

\bibitem[ADCK89]{MR1013132}
E.~Arbarello, C.~De~Concini, and V.~G. Kac, \emph{The infinite wedge
  representation and the reciprocity law for algebraic curves}, Theta
  functions---{B}owdoin 1987, {P}art 1 ({B}runswick, {ME}, 1987), Proc. Sympos.
  Pure Math., vol.~49, Amer. Math. Soc., Providence, RI, 1989, pp.~171--190.
  \MR{1013132 (90i:22034)}

\bibitem[AST07]{MR2360831}
M.~Argerami, F.~Szechtman, and R.~Tifenbach, \emph{On {T}ate's trace}, Linear
  Multilinear Algebra \textbf{55} (2007), no.~6, 515--520. \MR{2360831
  (2008i:15010)}

\bibitem[Be{\u\i}80]{MR565095}
A.~A. Be{\u\i}linson, \emph{Residues and ad\`eles}, Funktsional. Anal. i
  Prilozhen. \textbf{14} (1980), no.~1, 44--45. \MR{565095 (81f:14010)}

\bibitem[BFM91]{BFM_Conformal}
A.~{Be\u\i linson}, B.~Fe{\u\i}gin, and B.~Mazur, \emph{{Notes on conformal
  field theory}}, {unpublished, available on
  http://www.math.sunysb.edu/~kirillov/manuscripts.html}, 1991.

\bibitem[BGW16a]{MR3536437}
O.~Braunling, M.~Groechenig, and J.~Wolfson, \emph{Geometric and analytic
  structures on the higher ad\`eles}, Res. Math. Sci. \textbf{3} (2016), 3:22.
  \MR{3536437}

\bibitem[BGW16b]{bgwTateModule}
\bysame, \emph{{Operator ideals in Tate objects}}, Math. Res. Lett., to appear
  (2016).

\bibitem[BGW16c]{MR3510209}
\bysame, \emph{Tate objects in exact categories}, Mosc. Math. J. \textbf{16}
  (2016), no.~3, 433--504, With an appendix by Jan
  {\v{S}}{\v{t}}ov{\'{\i}}{\v{c}}ek and Jan Trlifaj. \MR{3510209}

\bibitem[Bra14]{MR3207578}
O.~Braunling, \emph{Ad\`ele residue symbol and {T}ate's central extension for
  multiloop {L}ie algebras}, Algebra Number Theory \textbf{8} (2014), no.~1,
  19--52. \MR{3207578}

\bibitem[BW16]{hochschildconiveau}
O.~Braunling and J.~Wolfson, \emph{{Hochschild coniveau spectral sequence and
  the Beilinson residue}}, {preprint, arXiv:1607.07756 [math.KT]} (2016)
  (English).

\bibitem[C{\'a}m13]{MR3161556}
A.~C{\'a}mara, \emph{Functional analysis on two-dimensional local fields},
  Kodai Math. J. \textbf{36} (2013), no.~3, 536--578. \MR{3161556}

\bibitem[FK00]{MR1804915}
I.~B. Fesenko and M.~Kurihara (eds.), \emph{Invitation to higher local fields},
  Geometry \& Topology Monographs, vol.~3, Geometry \& Topology Publications,
  Coventry, 2000, Papers from the conference held in M{\"u}nster, August
  29--September 5, 1999. \MR{1804915 (2001h:11005)}

\bibitem[FV02]{MR1915966}
I.~B. Fesenko and S.~V. Vostokov, \emph{Local fields and their extensions},
  second ed., Translations of Mathematical Monographs, vol. 121, American
  Mathematical Society, Providence, RI, 2002, With a foreword by I. R.
  Shafarevich. \MR{1915966 (2003c:11150)}

\bibitem[GS87]{MR917209}
M.~Gerstenhaber and S.~D. Schack, \emph{A {H}odge-type decomposition for
  commutative algebra cohomology}, J. Pure Appl. Algebra \textbf{48} (1987),
  no.~3, 229--247. \MR{917209 (88k:13011)}

\bibitem[Har66]{MR0222093}
R.~Hartshorne, \emph{Residues and duality}, Lecture notes of a seminar on the
  work of A. Grothendieck, given at Harvard 1963/64. With an appendix by P.
  Deligne. Lecture Notes in Mathematics, No. 20, Springer-Verlag, Berlin, 1966.
  \MR{0222093 (36 \#5145)}

\bibitem[Har77]{MR0463157}
\bysame, \emph{Algebraic geometry}, Springer-Verlag, New York, 1977, Graduate
  Texts in Mathematics, No. 52. \MR{0463157 (57 \#3116)}

\bibitem[Hub91a]{MR1105583}
A.~Huber, \emph{Ad\`ele f\"ur {S}chemata und {Z}ariski-{K}ohomologie},
  Schriftenreihe des {M}athematischen {I}nstituts der {U}niversit\"at
  {M}\"unster, 3.\ {S}erie, {H}eft 3, Schriftenreihe Math. Inst. Univ.
  M\"unster 3. Ser., vol.~3, Univ. M\"unster, M\"unster, 1991, p.~86.
  \MR{1105583 (92h:14014)}

\bibitem[Hub91b]{MR1138291}
\bysame, \emph{On the {P}arshin-{B}e\u\i linson ad\`eles for schemes}, Abh.
  Math. Sem. Univ. Hamburg \textbf{61} (1991), 249--273. \MR{1138291
  (92k:14024)}

\bibitem[HY96]{MR1374916}
R.~H{\"u}bl and A.~Yekutieli, \emph{Ad\`eles and differential forms}, J. Reine
  Angew. Math. \textbf{471} (1996), 1--22. \MR{1374916 (97d:14026)}

\bibitem[Kel99]{MR1667558}
B.~Keller, \emph{On the cyclic homology of exact categories}, J. Pure Appl.
  Algebra \textbf{136} (1999), no.~1, 1--56. \MR{1667558 (99m:18012)}

\bibitem[Lod92]{MR1217970}
J.-L. Loday, \emph{Cyclic homology}, Grundlehren der Mathematischen
  Wissenschaften [Fundamental Principles of Mathematical Sciences], vol. 301,
  Springer-Verlag, Berlin, 1992, Appendix E by Mar{\'{\i}}a O. Ronco.
  \MR{1217970 (94a:19004)}

\bibitem[Osi07]{MR2314612}
D.~V. Osipov, \emph{Adeles on {$n$}-dimensional schemes and categories
  {$C_n$}}, Internat. J. Math. \textbf{18} (2007), no.~3, 269--279. \MR{2314612
  (2008b:14005)}

\bibitem[Par76]{MR0419458}
A.~N. Parshin, \emph{On the arithmetic of two-dimensional schemes. {I}.
  {D}istributions and residues}, Izv. Akad. Nauk SSSR Ser. Mat. \textbf{40}
  (1976), no.~4, 736--773, 949. \MR{0419458 (54 \#7479)}

\bibitem[Par83]{MR697316}
\bysame, \emph{Chern classes, ad\`eles and {$L$}-functions}, J. Reine Angew.
  Math. \textbf{341} (1983), 174--192. \MR{697316 (85c:14015)}

\bibitem[PF99]{ParshinAdeleTheory}
A.~Parshin and T.~Fimmel, \emph{{Introduction to higher adelic theory
  (draft)}}, {unpublished}, 1999.

\bibitem[PR07]{MR2319783}
F.~Pablos~Romo, \emph{On the linearity property of {T}ate's trace}, Linear
  Multilinear Algebra \textbf{55} (2007), no.~4, 323--326. \MR{2319783
  (2008h:15006)}

\bibitem[RGPR14]{MR3177048}
J.~Ramos~Gonz{\'a}lez and F.~Pablos~Romo, \emph{A negative answer to the
  question of the linearity of {T}ate's trace for the sum of two
  endomorphisms}, Linear Multilinear Algebra \textbf{62} (2014), no.~4,
  548--552. \MR{3177048}

\bibitem[Ser97]{serrecft}
J.-P. Serre, \emph{Algebraic groups and class fields}, Graduate Texts in
  Mathematics, Springer Verlag New York, 1997.

\bibitem[Tat68]{MR0227171}
J.~Tate, \emph{Residues of differentials on curves}, Ann. Sci. \'Ecole Norm.
  Sup. (4) \textbf{1} (1968), 149--159. \MR{0227171 (37 \#2756)}

\bibitem[Wod88]{MR934604}
M.~Wodzicki, \emph{The long exact sequence in cyclic homology associated with
  an extension of algebras}, C. R. Acad. Sci. Paris S\'er. I Math. \textbf{306}
  (1988), no.~9, 399--403. \MR{934604 (89i:18012)}

\bibitem[Wod89]{MR997314}
\bysame, \emph{Excision in cyclic homology and in rational algebraic
  {$K$}-theory}, Ann. of Math. (2) \textbf{129} (1989), no.~3, 591--639.
  \MR{997314 (91h:19008)}

\bibitem[Yek92]{MR1213064}
A.~Yekutieli, \emph{An explicit construction of the {G}rothendieck residue
  complex}, Ast\'erisque (1992), no.~208, 127, With an appendix by Pramathanath
  Sastry. \MR{1213064 (94e:14026)}

\bibitem[Yek03]{MR2007399}
\bysame, \emph{The action of adeles on the residue complex}, Comm. Algebra
  \textbf{31} (2003), no.~8, 4131--4151, Special issue in honor of Steven L.
  Kleiman. \MR{2007399 (2005h:14045)}

\bibitem[Yek15]{MR3317764}
\bysame, \emph{Local {B}eilinson-{T}ate operators}, Algebra Number Theory
  \textbf{9} (2015), no.~1, 173--224. \MR{3317764}

\end{thebibliography}

\end{document}